\documentclass[12pt,reqno]{
amsart}

\usepackage[margin=1in]{geometry}
\usepackage[tt=false]{libertine}

\usepackage{amssymb,mathrsfs}

\usepackage[varbb]{newpxmath}
\let\savedbigtimes\bigtimes
\let\bigtimes\relax
\usepackage{mathabx} 
\let\bigtimes\savedbigtimes

\usepackage{graphicx}
\usepackage{enumerate}
\usepackage{bm}

\usepackage{bbm}

\usepackage{verbatim}
\usepackage{hyperref,color}
\usepackage[capitalize,nameinlink]{cleveref}
\usepackage[dvipsnames]{xcolor}
\hypersetup{
	colorlinks=true,
	pdfpagemode=UseNone,
    citecolor=OliveGreen,
    linkcolor=NavyBlue,
    urlcolor=black,
	pdfstartview=FitW
}
\usepackage{appendix}
\crefname{appsec}{Appendix}{Appendices}
\usepackage{tikz}

\theoremstyle{plain}
\newtheorem{theorem}{Theorem}[section]
\newtheorem{proposition}[theorem]{Proposition}
\newtheorem{lemma}[theorem]{Lemma}
\newtheorem{corollary}[theorem]{Corollary}

\newtheorem{claim}[theorem]{Claim}
\newtheorem{fact}[theorem]{Fact}

\theoremstyle{definition}
\newtheorem{definition}[theorem]{Definition}
\newtheorem{example}[theorem]{Example}

\newtheorem*{assumption*}{Assumption}

\theoremstyle{remark}
\newtheorem{remark}[theorem]{Remark}

\crefname{lemma}{Lemma}{Lemmas}
\crefname{theorem}{Theorem}{Theorems}
\crefname{definition}{Definition}{Definitions}
\crefname{fact}{Fact}{Facts}
\crefname{claim}{Claim}{Claims}
\crefname{proposition}{Proposition}{Propositions}

\newcommand{\Var}{\mathrm{Var}}

\newcommand{\Binom}{\mathrm{Binomial}}

\renewcommand{\epsilon}{\varepsilon}

\newcommand{\QQ}{\mathbb{Q}}
\newcommand{\PP}{\mathbb{P}}
\newcommand{\EE}{\mathbb{E}}

\newcommand{\beq}{\begin{equation}}
\newcommand{\eeq}{\end{equation}}

\renewcommand{\emptyset}{\varnothing}

\begin{document}
\author
[X. Yu, I. Zadik, P. Zhang]{Xifan Yu$^\dagger$, Ilias Zadik$^{\circ}$ and Peiyuan Zhang$^{\dagger}$}
\thanks{\raggedright$^\circ$Department of Statistics and Data Science, Yale University;
$^\dagger$Department of Computer Science, Yale University.
$^\star$Department of Electrical Engineering, Yale University;\\
Email: \texttt{\{xifan.yu, ilias.zadik, peiyuan.zhang\}@yale.edu}}

\title[Counting Stars is Optimal for Detecting Planted Subgraph]{Counting Stars is Constant-Degree Optimal\\ For Detecting Any Planted Subgraph}

\maketitle

\begin{abstract}%
We study the computational limits of the following general hypothesis testing problem. Let $H=H_n$ be an \emph{arbitrary} undirected graph on $n$ vertices.  We study the detection task between a ``null'' Erd\H{o}s-R\'{e}nyi random graph $G(n,p)$ and a ``planted'' random graph which is the union of $G(n,p)$ together with a random copy of $H=H_n$. Our notion of planted model is a generalization of a plethora of recently studied models initiated with the study of the planted clique model (Jerrum 1992), which corresponds to the special case where $H$ is a $k$-clique and $p=1/2.$

 Over the last decade, several papers have studied the power of low-degree polynomials for limited choices of $H$'s in the above task. In this work, we adopt a unifying perspective and characterize the power of \emph{constant degree} polynomials for the detection task, when \emph{$H=H_n$ is any arbitrary graph} and for \emph{any $p=\Omega(1).$} Perhaps surprisingly, we prove that the optimal constant degree polynomial is always given by simply \emph{counting stars} in the input random graph. As a direct corollary, we conclude that the class of constant-degree polynomials is only able to ``sense'' the degree distribution of the planted graph $H$, and no other graph theoretic property of it.
\end{abstract}

\section{Introduction}
During the last decade, researchers have revealed the existence of several computational-statistical trade-offs, that is parameter regimes where several statistical inference tasks are information theoretically possible but, conjecturably, computationally hard. A simple, yet rich, family of settings that these phenomena tend to appear are ``planted'' subgraph detection (or hypothesis testing) tasks where the goal is to detect the presence of a subgraph planted in otherwise Erd\H{o}s-R\'{e}nyi random graph $G(n,p)$\footnote{For $n \in \mathbb{N}, p=p_n$, an instance of the Erd\H{o}s-R\'{e}nyi $G(n,p)$ is an $n$ vertex undirected random graph where each edge appears independently with probability $p$.}. A notable example is the planted clique problem \cite{Jer92}, where one seeks to detect between a ``null'' model which is the Erd\H{o}s-R\'{e}nyi $G(n,1/2)$, and a ``planted subgraph'' model which is the union of $G(n,1/2)$ \emph{union} with a randomly chosen $k$-clique. Other planted subgraph models that have been studied in this line of work include: (a) the planted dense subgraph problem \cite{hajek2015computational}, where one plants in $G(n,p)$ an instance of $G(n,q)$ for $q>p$, (b) the planted tree model \cite{massoulie2019planting}, where one plants in $G(n,p)$ a $D$-ary tree, (c) the planted Hamiltonian cycle problem \cite{bagaria2020hidden} where one plants a Hamiltonian cycle, and (d) the planted matching problem \cite{moharrami2021planted} where one plants a perfect matching.

Notably, while at a high-level the methods to understand each planted subgraph model share similarities, the actual technical analysis is often quite intricate and tailored to the setting. It is natural to wonder if one could simultaneously study all planted subgraph detection tasks by focusing on the properties of a general framework. Motivated by exactly this desire, \cite{mossel2023sharp} studied the information-theoretic transitions of a general planted subgraph model, which includes all the above mentioned planted models as special cases. While \cite{mossel2023sharp} focused on the recovery task of estimating the hidden subgraph, we focus here on the detection variant of it. 

\begin{definition}[Planted subgraph detection task]\label{dfn:planted}
    Let $n \in \mathbb{N}, p=p_n \in (0,1)$ and $H=H_n$ be an \emph{arbitrary} undirected graph. We consider the following detection task.
   
    \begin{enumerate}
        \item (Null distribution $\mathbb{Q}$) In this case, the statistician observes an instance from the Erd\H{o}s-R\'{e}nyi random graph distribution $\mathbb{Q}=G(n,p)$.
        \item (Planted-$H$ distribution $\mathbb{P}$) In this case, the statistician observes the union of an Erd\H{o}s-R\'{e}nyi random graph $G(n,p)$ with a random copy of $H$. The random copy of $H$ is chosen uniformly at random from all the labelled copies of $H$ in the complete graph. 
        
    \end{enumerate}
\end{definition} It is rather straightforward to see how the general Definition \ref{dfn:planted} contains the mentioned detection tasks as special cases; e.g., when $H$ is a $k$-clique and $p=1/2$ we recover the planted clique task. 

\subsection*{Searching for universal structure: null and planted models} It is worth mentioning that this desire for generality shares roots with a fascinating line of work on the Kahn-Kalai conjecture from probabilistic combinatorics (see \cite{kahn2007thresholds} for the conjecture, and \cite{park2022proof} for a recent breakthrough proof). The context of the conjecture has similarities with our setting. It studies our null distribution $\mathbb{Q}=G(n,p),$ and it is about characterizing the thresholds $p$ for which an instance of a $G(n,p)$ contains a specific subgraph $H=H_n$ of interest. Similar again to the literature of planted models, a plethora of works have studied the thresholds for specific choices of subgraphs (e.g., see the classic work on Hamiltonian cycles \cite{posa1976hamiltonian} and the very technical recent work on spanning trees of bounded degree \cite{montgomery2019spanning}). The Kahn-Kalai conjecture offers a formula for the threshold for \emph{any subgraph} H. It is remarkable how the, now proven, Kahn-Kalai conjecture directly implies multiple previous notable results in random graph theory as direct corollaries (including the mentioned examples). It is also remarkable that the proof of the conjecture was only a few pages long. This line of work offers at least an argument that seeking a general and unifying structure in the analysis of such random graph models can be very fruitful. 

Returning now to planted subgraph models, similar to \cite{mossel2023sharp}, the question of finding a general structure underlying all these models, similar to the line of work on the Kahn-Kalai conjecture, is the primary motivation of our work. While \cite{mossel2023sharp} studied the information-theoretic limits of planted subgraph detection tasks, in this work we investigate a common structure on their computational limits, i.e., in their \emph{computational-statistical trade-offs}. Unfortunately, given that the $\mathcal{P} \not = \mathcal{NP}$ question remains unsolved, identifying the ``true'' computational limit of any detection task, that is characterize exactly when some polynomial-time test succeeds or not, appears to be well beyond the current mathematical abilities. For this reason, researchers on computational-statistical trade-offs have turned to studying multiple powerful restricted class of test statistics, often containing the best known polynomial-time test, and offering their proven failure point as evidence the existing computational limits are fundamental.

\subsection*{Low-degree polynomials} Motivated by connections with the Sum-of-Squares hierarchy, the study of the powerful class of \emph{low-degree polynomials} to construct test statistics has played a key role in this direction. First, it can be verified in a plethora of cases that the best known polynomial-time test statistics (e.g., based on spectral methods, or message passing methods) can be well approximated by low-degree polynomials (commonly $O(1)$ or $O(\log n)$ degree suffices). On top of that, the class of low-degree polynomials is believed to be very powerful, and a now well-known ``low-degree conjecture'' \cite{hopkins2018statistical}, \cite{kunisky2019notes} states that for a general class of detection problems when all $O(\log n)$-degree polynomials fail to strongly separate the two distributions (see Definition \ref{dfn:strong_sep} below), then no polynomial-time test will be able to detect between the two. The performance of the class of constant $O(1)$-degree polynomials has also been used as (less strong but still quite interesting) evidence of hardness. For example, a recent work \cite{montanari2022equivalence} established that Approximate Message Passing is optimal among $O(1)$-degree polynomial in a spiked matrix estimation setting.

For these reasons, multiple papers have studied so far the power of low-degree polynomials to achieve strong separation for a number of different planted subgraph detection tasks. For example, in the planted clique model it is known that if $k=\Omega(\sqrt{n})$ some $O(\log n)$-degree polynomial succeeds, while if $k=o(\sqrt{n})$ all $O(\log n)$-degree polynomials fail to strongly separate the two distributions (see e.g., \cite{kunisky2019notes} and references therein). It should be noted though that for every new planted subgraph detection task that has been analyzed a new careful analysis is usually needed, which often brings its own challenges (similar to the literature of the Kahn-Kalai conjecture). The main focus of this work is to explore the \emph{simultaneous} study of the class of low-degree polynomials for all planted subgraph detention tasks i.e., for any planted subgraph $H$.

\subsection*{Abscence of structural positive results} Moreover, prior work on low-degree polynomials has built a powerful technique, based on what is called the low-degree advantage (or low-degree likelihood ratio) \cite{kunisky2019notes} (see Definition \ref{dfn:advantage}), to prove the \emph{failure}\footnote{From this point on in the Introduction, by ``success'' (or ``failure'') of a polynomial in a detection task we strictly refer to whether it strongly separates $\mathbb{P}$ and $\mathbb{Q}$ (or not), per Definition \ref{dfn:strong_sep}. } of the class of low-degree polynomials for detection tasks. To be more precise, as long as the degree-$D$ advantage remains bounded, we know that no degree-$D$ polynomial can strongly separate $\mathbb{Q}$ and $\mathbb{P}$. Yet, our understanding of how to argue about \emph{positive results} for the class of low-degree polynomials is significantly more limited and much less automated. One natural candidate would be to consider the low-degree advantage again and use it as a criterion if it is unbounded. Unfortunately, this suggestion is not generally true. For example, for a regime of the so-called planted dense hypergraph problem, the low-degree ratio explodes (due to rare events) but in fact no low-degree polynomial succeeds \cite{dhawan2023detection}.  Understanding whether the low-degree advantage exploding is a sufficient criterion for the success of low-degree polynomials when we focus on planted subgraph detection tasks is also a partial motivation for the present work.

A related struggle is that even if a low-degree polynomial is ``predicted'' to work, there is no known general tool to understand the structure of this ``optimal'' low-degree polynomial. Yet, the best known algorithms (and therefore their polynomial approximations) appear to be significantly different among different settings. For example, the best known polynomial-time for planted clique is a spectral method together with a combinatorial search post-processing step \cite{alon1998finding}, while for the planted Hamiltonian cycle problem it is a linear program relaxation of a TSP problem \cite{bagaria2020hidden}. This is a significant issue, as for any new detection task that a statistician is facing, they need to design the ``correct'' polynomial-time test statistic mostly based on their intuition and the unique properties of each subgraph $H$. 

Summarizing the above, this work is motivated by the following three key questions on planted subgraph detection tasks.
\begin{quote}
\centering{\emph{ (Q1) For any $H=H_n$, can we automate when a degree-$D$ polynomial works?}}\\
    \centering{\emph{(Q2) Can we characterize the structure of the optimal degree-$D$ polynomial?}}\\
    \centering{\emph{ (Q3) Which features of $H=H_n$ should a degree-$D$ polynomial be exploiting?}}
\end{quote}

\subsection*{Main Contributions (Informal)}

In this work, our main contribution is to offer an answer to the above questions (Q1), (Q2), (Q3) for any planted subgraph detection task with arbitrary $H=H_n$ and for any $p=\Omega(1)$, when we focus on the class of \emph{constant $D=O(1)$-degree polynomials}. Informally, under these assumptions, a summary of our contributions is as follows. \begin{itemize}

    \item We start with our main result (Theorem \ref{thm:main}). We prove that for all choices of $H=H_n$, the optimal $D=O(1)$-degree polynomial is always given by the \textbf{signed count of a $t$-star graph} in the input graph for some $1 \leq t \leq D$. In other words, some degree-$D$ polynomial succeeds in detecting if and only if for some $1 \leq t \leq D$ the polynomial which counts $t$-star graphs works. This reveals an interesting new underlying structure shared by all planted subgraph detection tasks, and an easy-to-check criterion for the success of constant degree polynomials.

    \item A moment analysis of the star counting polynomials, together with our main result, implies that the success of constant degree polynomials is a function solely of the \emph{degree distribution} of the planted $H$ (see Theorem \ref{thm:degree-characterization}). In other words, for any two $H_1$ and $H_2$ with the same degree distribution, either some degree-$D$ polynomial can succeed for the planted detection tasks corresponding to both $H_1$ and $H_2,$ or all degree-$D$ polynomials fail for the detection tasks corresponding to both $H_1$ and $H_2$. We find this a surprising conclusion of our work, given all the potential other features of $H$ a constant $D$-degree polynomial could be exploiting (e.g., the $D>1$-neighborhood structure of each vertex).

\item We describe how our results implies a series of old and new results on low-degree polynomials for planted detection tasks (see Section \ref{sec:apps}). 
    \item We prove that our main result is tight. We provide counterexamples for the optimality of star counts when either $p=o(1)$ or $D=\omega(1).$  (see Section \ref{subsection:failure-of-counting-stars}).
\end{itemize}

\subsection*{Further comparison with previous work}

We would like to expand here briefly on our literature review. We are not aware of any other work attempting to understand the above questions for planted subgraph detection tasks per Definition \ref{dfn:planted} in that level of generality. Yet, we should mention a relevant work by \cite{huleihel2022inferring}.

First, \cite{huleihel2022inferring} studies a similar, yet incomparable, general planted subgraph setting where one seeks to detect between an Erd\H{o}s-R\'{e}nyi $G(n,p)$ and a planted distribution where one plants a copy of a subgraph $H$ as an \emph{induced subgraph} in a $G(n,p).$

We first note this is not the same setting with Definition \ref{dfn:planted}, as in our planted model we observe the \emph{union} of a copy of $H$ with $G(n,p)$. In particular, in our model the planted $H$ is not assumed to be an induced subgraph of the input graph. Interestingly notice that the ``induced'' and ``union'' models are two distinct generalizations of the planted clique setting.

\cite{huleihel2022inferring} analyzed the power of low-degree polynomial in the induced setting for all $n^{-o(1)} \leq p \leq 1-n^{-o(1)}$ and for all $H$ that have edge density bounded away from $p$. In that case, the author proved that the computational trade-off is similar with the case of the planted clique model. For any such $H$, by simply counting edges one can detect whenever $v(H)=\omega(\sqrt{n \log n}),$ and the authors prove that a spectral method can improve this to $v(H)=\Omega(\sqrt{n}).$ Moreover, if $v(H)=o(\sqrt{n})$ the author establish that no $O(\log n)$-degree polynomial works for the $H$'s of interest. We remark again that we analyze the incomparable union model, and on top of this, we note that our results do not restrict $H$ at all.

\section{Main Results}
To describe our main contribution in more detail, we need first to give a few definitions. Recall that a graph on $n$ vertices can be represented as a list of Boolean variables $G = (G_{\{i,j\}})_{\{i,j\} \in \binom{[n]}{2} }$, where each variable is the indicator variable of one edge in the complete graph $K_n$. We start with the notion of the strong separation which will be our focus of ``success'' for a polynomial in a planted subgraph detection task.
\begin{definition}[Strong separation]\label{dfn:strong_sep}
    For two distributions $\mathbb{P}, \mathbb{Q}$ supported on graphs $\{0,1\}^{\binom{n}{2}}$ we say that an $\binom{n}{2}$-variate polynomial $f(G)$ strongly separates the distributions $\mathbb{P}$ and $\mathbb{Q}$ if it holds 
\[\max\left\{\sqrt{\mathrm{Var}_{\mathbb{P}}(f)},\sqrt{\mathrm{Var}_{\mathbb{Q}}(f)}\right\}=o(|\mathbb{E}_{\mathbb{P}}(f)-\mathbb{E}_{\mathbb{Q}}(f)|).\]
\end{definition}One should think of strong separation as a stronger condition for detection. A simple application of Chebychev's inequality implies if $f$ strongly separates $\mathbb{P}, \mathbb{Q}$ then thresholding $f$ suffices to do detect between the two distributions with vanishing Type I and II errors.

Relevant to strong separation is the concept of the advantage of a test function. \begin{definition}[(Low-degree) advantage]\label{dfn:advantage}
    For two distributions $\mathbb{P}, \mathbb{Q}$ supported on graphs $\{0,1\}^{\binom{n}{2}}$, the advantage of a real-valued test function $f: \{0,1\}^{\binom{n}{2}} \to \mathbb{R}$ for testing distribution $\mathbb{P}$ against distribution $\mathbb{Q}$ is given by $\text{Adv}(f) := \frac{\mathbb{E}_{\mathbb{P}}[f]}{\sqrt{\mathbb{E}_{\mathbb{Q}}[f^2]}},$
    and the degree-$D$ advantage for testing distribution $\mathbb{P}$ against distribution $\mathbb{Q}$ (also called, the low-degree likelihood ratio) is $ \text{Adv}^{\le D} := \max_{f \in \mathbb{R}[X]^{\le D}} \text{Adv}(f) = \max_{f \in \mathbb{R}[X]^{\le D}}\frac{\mathbb{E}_{\mathbb{P}}[f]}{\sqrt{\mathbb{E}_{\mathbb{Q}}[f^2]}}.$
\end{definition}It is known (see e.g., \cite[Lemma 7.3]{coja2022statistical}) that if $\text{Adv}^{\le D} =O(1)$ then no degree-$D$ polynomial can strongly separate $\mathbb{P}$ and $\mathbb{Q}.$

\subsection*{Walsh-Fourier basis}Recall the Walsh-Fourier basis $(\chi_{S})_{S \subseteq \binom{[n]}{2}}$ (the degree-$D$ Walsh-Fourier basis\\ $(\chi_{S})_{S \subseteq \binom{[n]}{2}: |S| \le D}$ resp.) with respect to the Erd\H{o}s-R\'{e}nyi distribution $G(n,p)$ is defined by \[\chi_{S}(G) := \prod_{\{i,j\} \in S} \frac{G_{\{i,j\}} - p}{\sqrt{p(1-p)}}.\]

\subsection*{Signed subgraph counts} Of special role in this work are the degree-$D$ polynomials called (signed) subgraph counts. That is, for any shape\footnote{A shape is an edge-induced graph.} $\mathbf{S}$ with at most $D$ edges, the signed count of $\mathbf{S}$ is the degree-$D$ polynomial given by
\begin{align}\label{eq:count}
    f_{\mathbf{S}}(G) = \sum_{S \subseteq \binom{n}{2}: S \cong \mathbf{S} } \chi_S(G),
    \end{align}
    where $S \cong \mathbf{S}$ denotes the graph isomorphism relation. 
    
Finally, we remind the reader the definition of a star graph.
\begin{definition}\label{dfn:star}
    A star graph with $t$ edges is a tree with $t+1$ vertices consisting of $t$ leaves and $1$ internal ``central" vertex, as shown in Figure \ref{fig:star-graph}. In this paper, this graph is denoted as $K_{1,t}$, as it can be viewed as the complete bipartite graph with $1$ vertex in one part and $t$ vertices in the other.
\end{definition}

\begin{figure}[h]
    \centering

\tikzset{every picture/.style={line width=0.75pt}} 

\begin{tikzpicture}[x=0.5pt,y=0.5pt,yscale=-1,xscale=1]

\draw    (99.6,109.95) -- (40.1,110.45) ;
\draw    (99.6,109.95) -- (58.1,157.45) ;
\draw    (99.6,109.95) -- (107.6,171.45) ;
\draw    (58.1,66.95) -- (99.6,109.95) ;
\draw    (99.6,109.95) -- (108.1,54.45) ;
\draw    (99.6,109.95) -- (147.1,77.95) ;
\draw    (99.6,109.95) -- (148.1,151.45) ;
\draw    (281.1,109.45) -- (221.6,109.95) ;
\draw    (281.1,109.45) -- (261.6,169.95) ;
\draw    (260.1,49.45) -- (281.1,109.45) ;
\draw    (281.1,109.45) -- (332.1,72.95) ;
\draw    (281.1,109.45) -- (333.1,148.45) ;

\draw (134.1,93.4) node [anchor=north west][inner sep=0.75pt]    {$\vdots $};

\end{tikzpicture}
    \caption{A generic star graph on the left, and the star graph $K_{1,5}$ on the right.}
    \label{fig:star-graph}
\end{figure}
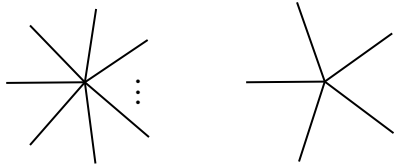

    \subsection{Main Result: Optimality of Star Counts}
    Our main result is a generic result that holds for any $p=\Omega(1)$ and any $D=O(1).$ We prove that for all planted subgraph detection tasks (per Definition \ref{dfn:planted}), i.e., for any $H=H_n$, there exists $t \leq D$ for which the $t$-star signed subgraph count is optimal among all degree-$D$ polynomials to strongly separate $\mathbb{P}$ and $\mathbb{Q}$. Formally, our main result is the following.




\begin{theorem} \label{thm:main}
    Suppose $H=H_n$ is an arbitrary subgraph, $D = O(1)$ and $p = \Omega(1)$. Then, the following holds for testing $\mathbb{P}$ and $\mathbb{Q}$ in the planted subgraph detection task corresponding to planting a copy of $H$ per Definition \ref{dfn:planted}.
    \begin{itemize}
        \item If $\limsup_{n\to \infty} \text{Adv}^{\le D} < \infty$, then no degree-$D$ polynomial $f \in \mathbb{R}[X]_{\le D}$ achieves strong separation between $\mathbb{P}$ and $\mathbb{Q}$.
        \item If $\lim_{n \to \infty} \text{Adv}^{\le D} = \infty$, then there exists $1 \leq t \leq D$ such that the signed count of a star $f_{\mathbf{K}_{1,t}}$ achieves strong separation. In particular, $\lim_{n \to \infty} \text{Adv}(f_{\mathbf{K}_{1,t}})=\infty.$
    \end{itemize}
\end{theorem}

A few remarks are in order.
\begin{remark}
   As we mentioned in the Introduction, it is not generally true that a growing advantage $\text{Adv}^{\le D} =\omega(1)$ implies that some degree-$D$ polynomial achieves strong separation (e.g., see \cite{dhawan2023detection} for a planted hypegraph setting that this fails). This is in contrast to the well-known fact (the first bullet point in Theorem \ref{thm:main}) that bounded degree-$D$ advantage rules out the existence of degree-$D$ polynomials that achieve strong separation (see e.g., \cite[Lemma 7.3]{coja2022statistical}). 
   
   Our Theorem \ref{thm:main} shows that, interestingly, for all planted subgraph detection tasks with $p=\Omega(1)$, the ``converse'' does indeed hold for $D=O(1)$: whenever the degree-$D$ advantage blows up, there indeed exists a degree-$D$ polynomial that achieves strong separation. In particular our result offers, to the best of our knowledge, the first complete characterization of the power of constant-degree polynomials in such settings.
\end{remark}

\begin{remark} Theorem \ref{thm:main} implies that for all planted $H$'s, the simple choice of counting signed star graphs is always the optimal choice for strong separation between $\mathbb{P}$ and $\mathbb{Q}$, among all constant-degree polynomials. To the best of our knowledge, this is the first result revealing this universal optimality of counting stars for all planted subgraph detection tasks in our regime.

It is natural to wonder what is the reason counting stars enjoy such general optimality. We point the reader to Section \ref{sec:proof}, and specifically Proposition \ref{prop:star-shape} for details and a proof sketch. We only would like to note here that the star structure appears naturally as a subgraph $\mathbf{S}$ whose signed count has (almost) maximum advantage $\max_{\mathbf{S}} \text{Adv}(f_{\mathbf{S}})$ among all constant-sized graphs. 
\end{remark}

\begin{remark}
    Our theorem needs two assumptions for the optimality of signed $t$-star counts. First, that the class is constant degree polynomials, i.e., $D=O(1),$ and second that $p=\Omega(1).$ It turns out that both assumptions are necessary: if either is not satisfied then the signed count of stars may fail to be optimal among constant-degree polynomials (see Section \ref{subsection:failure-of-counting-stars}).
\end{remark}

\subsection{A simple criterion using only the degree profile of $H$}

The fact that counting star graphs is optimal among $D=O(1)$-degree graphs for strong separation implies a very simple criterion for the success of the class of polynomial test functions. To present this, we fix an arbitrary subgraph $H$ and consider the planted subgraph detection task for $H$ with any noise level $p=\Omega(1)$. Suppose a statistician wishes to understand the power of degree-$D$ polynomial test functions for this setting. Based on the literature, the currently natural approach would be as follows. First, the statistician would try to show that in one regime the low-degree advantage $\text{Adv}^{\le D}=\max_{f \in \mathbb{R}^{\leq D}[X]} \text{Adv}(f)$ remains bounded and then, in the remaining regime, to design (from scratch!)  a constant-degree polynomial that works. 

Based on Theorem \ref{thm:main} we arrive at a much simpler and automated approach for how to understand both directions when $D=O(1)$. It is in fact sufficient for the statistician to only calculate the $D$ advantages of the $t$-star counts, i.e. the $D$ numbers $\text{Adv}(f_{\mathbf{K}_{1,t}}), t=1,\ldots,D.$ Indeed, by our Theorem \ref{thm:main} there exists a degree-$D$ polynomial that can strongly separate $\mathbb{P}$ and $\mathbb{Q}$, if and only if  
\begin{align}\label{eq:crit1}
    \max_{t=1,\ldots,D}\text{Adv}(f_{\mathbf{K}_{1,t}})=\omega(1).
\end{align} Moreover, it can be proven in the latter case that the signed star count $f_{\mathbf{K}_{1,t}}$ achieving the maximum advantage\footnote{We actually look at the advantage rescaled by the square root of the size of the automorphism group of $\mathbf{K}_{1,t}$, which will be later referred to as the rescaled advantage. For the ease of presentation, we defer the precise statement to the proof strategy described in Section \ref{subsection:proof-strategy} and the proof of Theorem \ref{thm:main}.} among $1 \leq t \leq D$ does also achieve strong separation, so by calculating the $D$ advantages the statistician would also arrive to the optimal test function.

On top of that, one can simplify further the condition \eqref{eq:crit1}. The advantages of star counts are in fact approximately the (rescaled) empirical moment of the degree distribution of $H$, $(d_1,\ldots,d_{v(H)})$. By straightforward mean and variance computations, it holds that
\begin{align*}
    \text{Adv}(f_{\mathbf{K}_{1,t}}) = (1 + o(1))\sum_{v \in V(H)} \frac{(d_v)_{(t)}}{|\text{Aut}(\mathbf{K}_{1,t})|^{1/2} \cdot n^{(1+t)/2}} \left(\frac{1-p}{p}\right)^{t/2}, 1 \leq t \leq D.
\end{align*} Together with the above discussion this leads to the following simple and general condition for all planted subgraph detection tasks.

\begin{corollary}\label{cor:main}[``A simple condition on the degree profile of $H$'']
    For any $H=H_n$, $p=\Omega(1)$ and the corresponding planted subgraph detection task the following holds. There exists a $D=O(1)$-degree polynomial that can achieve strong separation if and only if
    \begin{align}\label{eq:limit}
        \lim_{n \to \infty} \max_{1\le t \le D} \frac{\sum_{v \in V(H)} d_v^t }{n^{\frac{1+t}{2}} \left(\frac{p}{1-p}\right)^{\frac{t}{2}}}=\infty,
    \end{align}where $d_v := \deg_H(v)$ denotes the degree of $v$ in $H$. Moreover, if \eqref{eq:limit} holds, by choosing 
    \[t^*=\arg \max_{1 \leq t \leq D } \frac{\sum_{v \in V(H)} d_v^t }{n^{\frac{1+t}{2}} \left(\frac{p}{1-p}\right)^{\frac{t}{2}}},\]
    then we can strongly separates $\mathbb{P}$ and $\mathbb{Q}$ using $f_{\mathbf{K}_{1,t^*}}$.
\end{corollary}
A potential striking aspect of Corollary \ref{cor:main} is that to judge whether some constant-degree polynomial works \emph{one needs to only know the degree profile of $H$}, and no other graph property of it. For example, for a $d$-regular $H$ with vertex size $v,$ constant degree polynomials can strongly separate $\mathbb{P}$ and $\mathbb{Q}$ either \emph{for all} such graphs $H$ or \emph{for no} such $H$ at all. Other more specific properties of $H$ (e.g., the $H$'s clique number, or its girth, or even spectral properties like $H$ being an expander or not), which could naturally motivate the study of several other candidate degree-$D$ polynomials, make \emph{no difference} in whether some constant-degree polynomial can strongly separate or not.

\section{Characterization of the Optimal Test}
Based on Corollary \ref{cor:main}, the optimal constant degree polynomial test function for the planted subgraph detection task is given by the count of $t$-stars for the maximizer of $t \in [1,D]$ satisfying \eqref{eq:limit}. Narrowing down the characterization, it is the next natural question to decide for which $H=H_n$ should we be choosing smaller $t$, e.g., $t=1$ which corresponds to counting edges, or larger $t.$  

It turns out that, because of a convexity argument, it is always optimal to either choose $t=1$ or $t=D.$ Moreover, the decision between $t=1$ or $t=D$ is (almost\footnote{Excluding a small ``gray'' area, see Figure \ref{fig:char}}) entirely based on the maximum degree of the planted subgraph $H$. If the maximum degree is less than an appropriate threshold $\approx (np/(1-p))^{1/2}$, one should simply count edges. If the maximum degree is much bigger than $\approx (np/(1-p))^{1/2}$, the optimal test is to count large stars (see Figure \ref{fig:char}). To provide the details, we first need a definition.
\begin{definition}
    We say a test $T$ is the optimal test among a class of tests $\mathcal{T}$ for testing $\mathbb{P}$ against $\mathbb{Q}$ if whenever there exists a test $\Tilde{T} \in \mathcal{T}$ that strongly separate $\mathbb{P}$ and $\mathbb{Q}$, the test $T$ also does so.
\end{definition}We now state our main theorem for this section discussing when $t=1$ or $t$ large is optimal, and give an (almost) complete and simple characterization for when constant-degree polynomials work.
\begin{theorem} \label{thm:degree-characterization}
    Denote $\Delta = \max_{i \in V(H)} d_i$, where $d_i = \deg_H(i)$ and $m=|E(H)|$. Then, for $p = \Omega(1)$, we have the following characterization for the optimal test among constant degree polynomial tests for the planted subgraph detection problem of testing $\mathbb{P}$ against $\mathbb{Q}$ as defined in Definition \ref{dfn:planted}: 
    \begin{itemize}
        \item If $\Delta=O\left( \left(n \frac{p}{1-p}\right)^{1/2}\right)$, then the optimal test is to count signed edges, i.e., $f = \sum_{\{i,j\} \in \binom{V}{2}} \chi_{\{i,j\}}$. In particular, some constant-degree polynomial works if and only if $m = \omega\left(n \left(\frac{p}{1-p}\right)^{1/2}\right)$.
        \item If $\Delta \ge \left(n \frac{p}{1-p}\right)^{1/2 + \varepsilon}$ for some constant $\varepsilon > 0$, then the optimal test is to count signed ``large'' stars , i.e., $f = \sum_{S \subseteq \binom{V}{2}: S \cong \mathbf{S} } \chi_S$ where $S = \mathbf{K}_{1,t}$ for a large enough constant $t$, and moreover, it is always possible to set $t = \lceil \frac{3}{2\varepsilon} \rceil + 1$ so that the above $f$ achieves strong separation.
    \end{itemize}

    Moreover, for any $D = O(1)$, the optimal test among degree-$D$ polynomial tests for testing $\mathbb{P}$ against $\mathbb{Q}$ is either to count signed edges or to count signed $D$ edge stars $K_{1,D}$.
\end{theorem}
The statement of Theorem \ref{thm:degree-characterization} is best visualized in Figure \ref{fig:char}.
\begin{remark}
    Although the characterization in Theorem \ref{thm:degree-characterization} does not capture a middle region of \\$\omega\left(\left(n \frac{p}{1-p}\right)^{1/2}\right) \le \Delta \le \left(n \frac{p}{1-p}\right)^{1/2 + o(1)}$, the optimal test for this region can still be found by looking for the constant-sized star graph that maximizes the advantage of its signed counting polynomial as stated in Theorem \ref{thm:main}. Moreover, our characterization in Theorem \ref{thm:degree-characterization} implies that the optimal test in this region is also one of the two extremes: to count signed edges or signed large stars.
\end{remark}

\begin{figure}[ht]
    \centering
\tikzset{every picture/.style={line width=0.75pt}} 
\begin{tikzpicture}[x=0.75pt,y=0.75pt,yscale=-0.9,xscale=0.9]


\draw  [line width=1.5]  (55,50) -- (445,50) -- (445,300) -- (55,300) -- cycle ;
\draw  [fill={rgb, 255:red, 130; green, 150; blue, 255 }  ,fill opacity=0.4 ] (245,175) -- (445,175) -- (445,300) -- (245,300) -- cycle ;
\draw  [fill={rgb, 255:red, 130; green, 255; blue, 135 }  ,fill opacity=0.4 ] (245,50) -- (445,50) -- (445,175) -- (245,175) -- cycle ;
\draw  [fill={rgb, 255:red, 255; green, 40; blue, 50 }  ,fill opacity=0.4 ] (55,50) -- (245,50) -- (245,175) -- (55,175) -- cycle ;
\draw  [fill={rgb, 255:red, 155; green, 155; blue, 155 }  ,fill opacity=0.5 ] (230,175) -- (245,175) -- (245,300) -- (230,300) -- cycle ;
\draw [line width=1.5]    (55,300) -- (455,300) ;
\draw [shift={(470,300)}, rotate = 180] [fill={rgb, 255:red, 0; green, 0; blue, 0 }  ][line width=0.08]  [draw opacity=0] (15.6,-3.9) -- (0,0) -- (15.6,3.9) -- cycle    ;
\draw [line width=1.5]    (55,300) -- (55,24) ;
\draw [shift={(55,20)}, rotate = 90] [fill={rgb, 255:red, 0; green, 0; blue, 0 }  ][line width=0.08]  [draw opacity=0] (15.6,-3.9) -- (0,0) -- (15.6,3.9) -- cycle    ;


\footnotesize\draw (285,220) node [anchor=north west][inner sep=0.75pt]   [align=left] { Counting Large Stars\\ \ \ \ \ \ \ \ \ Succeeds};
\footnotesize\draw (280,90) node [anchor=north west][inner sep=0.75pt]   [align=left] { Both Counting Edges \\ \ \ \ and Large Stars\\ \ \ \ \ \ \ \ \ \ Succeed};
\footnotesize\draw (95,100) node [anchor=north west][inner sep=0.75pt]   [align=left] {Counting Edges\\ \ \ \ \ \ \ Succeeds};
\footnotesize\draw (100,225) node [anchor=north west][inner sep=0.75pt]   [align=left] {CD Impossible};
\draw (180,305) node [anchor=north west][inner sep=0.75pt]    {\small $\left(\tfrac{np}{1-p}\right)^{1/2}$};
\draw (245,305) node [anchor=north west][inner sep=0.75pt]    {\small $\left(\tfrac{np}{1-p}\right)^{1/2+o(1)}$};
\draw (440,312) node [anchor=north west][inner sep=0.75pt]  [font=\small]  {$\Delta$};
\draw (38,310) node [anchor=north west][inner sep=0.75pt]  [font=\small]  {$0$};
\draw (30,43.4) node [anchor=north west][inner sep=0.75pt]  [font=\small]  {$m$};
\draw (3,170.4) node [anchor=north west][inner sep=0.75pt]  [font=\footnotesize]  {$n\sqrt{\tfrac{p}{1-p}}$};
\end{tikzpicture}
    \caption{
    Phase-transition diagram characterized by Theorem~\ref{thm:degree-characterization}. The x-axis is the maximum degree $\Delta$ and the y-axis is the number of edges $m$ of subgraph $H$. CD is an abbreviation of constant degree.}
    \label{fig:char}
    
\end{figure}
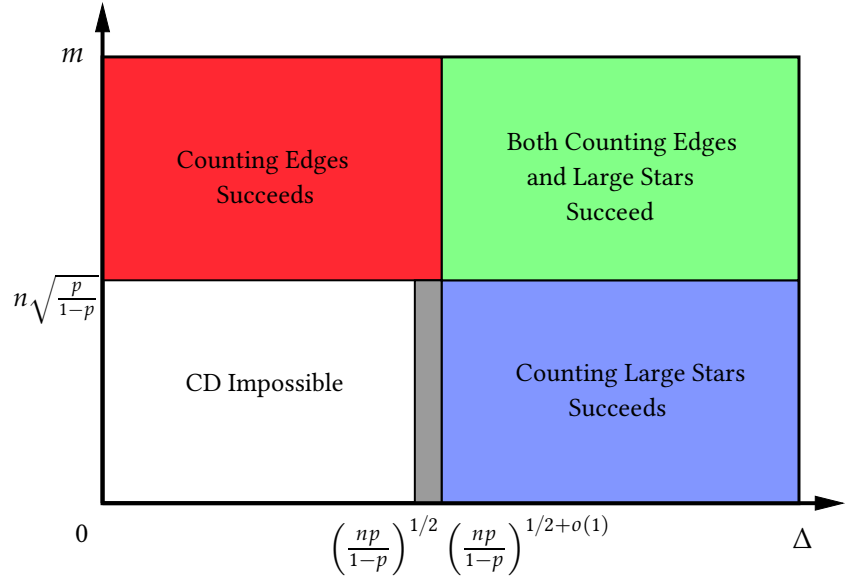
\subsection{Applications}\label{sec:apps} To show the applicability of Theorem~\ref{thm:degree-characterization}, we apply it to a number of examples.

\subsection*{Planted dense subgraph.}
Let $n, k=k_n \in \mathbb{N}$ and $p=p_n,q=q_n \in (0,1)$. The first example we consider is the planted dense subgraph (PDS) setting, denoted by $\mathrm{PDS}(k,p,q)$. To describe it, we first draw $H$ from $G(k,q).$ Then we consider the planted detection task per Definition \ref{dfn:planted} for planted $H$ and $p.$ Notice that this is called a PDS setting, as in the planted model $\mathbb{P},$ the induced subgraph on the $k$ vertices corresponding to $H$ is ``denser'' compared to the rest edges of the graph. Indeed, in the planted instance, every edge using only vertices of $H$ appears marginally with probability $p+(1-p)q>p$ while the rest of the edges with probability $p$.

Notice that, as long as $kq=\omega(\log k)$, $H$ has maximum degree $(1+o(1))(k-1)q$ and edges $(1+o(1))k^2q/2$ with high probability as $n$ grows to infinity (see Lemma \ref{lemma: degree-edge-bound} in the appendix). Using Theorem \ref{thm:degree-characterization} we directly conclude the following.
\begin{corollary}\label{cor:const}
     If $p=\Omega(1)$, a constant-degree polynomial can achieve strong separation in $\mathrm{PDS}(k,p,q)$ if and only if $ k = \omega\left(\frac{\sqrt{n}}{\sqrt{q}(1-p)^{1/4}}\right).$ Moreover, if $ k = \omega\left(\frac{\sqrt{n}}{\sqrt{q}(1-p)^{1/4}}\right)$, counting edges achieves strong separation.
\end{corollary}For instance, if $k=n^{\beta}$, $p=1-n^{-\gamma}$ and $q=n^{-\alpha}$ for constants $\alpha,\beta, \gamma \in (0,1)$ we immediately conclude the phase diagram for this $\mathrm{PDS}(k,p,q)$ for constant-degree polynomials: they work if and only if $\beta>1/2+\alpha/2+\gamma/4$ (see Figure \ref{fig:PDS}). 

 It is worth noting that PDS has been commonly studied under the slightly different following definition we call $\mathrm{PDS'}(k,p,p')$  for $0<p=p_n<p'=p'_n<1$ and $k=k_n$ (see e.g., \cite{brennan2018reducibility}). $\mathrm{PDS'}(k,p,p')$ is the detection task between the null $\mathbb{Q}=G(n,p)$ and the planted model $\mathbb{P}$ where now the observed graph is sampled like a $G(n,p)$ except that the edges between $k$ vertices, chosen uniformly at random, are sampled now with probability $p'>p$. Notice this definition, differently from the one using Definition \ref{dfn:planted}, is not assuming that the planted instance is the union of $G(n,p)$ with a randomly placed $H \sim G(n,q).$ Yet, for $q=(p'-p)/(1-p)$ the two planted models of $\mathrm{PDS}(k,p,q)$ and $\mathrm{PDS'}(k,p,p')$ have marginally (per edge) the same law and it is natural to expect to have a similar computational transition.
 
 Indeed, for $\mathrm{PDS'}(k,p,p')$  in the specialized regime $\Omega(1)=p=1-\Omega(1)$, $p'-p=\Theta(n^{-\alpha})$ and $k=\Theta(n^{\beta})$ it is known, by using an average-case reduction from planted clique \cite{brennan2018reducibility}, that there is conjectured hard phase if and only if $\beta<1/2+\alpha/2$. By choosing the matching parameter $q=(p'-p)/(1-p)=\Theta(n^{-\alpha})$ and $\gamma=0$ we arrive at the same conclusion for $\mathrm{PDS}(k,p,q)$: constant degree polynomials work if and only if $\beta<1/2+\alpha/2$ and counting edges works in the opposite regime.

 We are not aware of any comparable low-degree lower bound result for either $\mathrm{PDS'}(k,p,p')$ and $\mathrm{PDS}(k,p,q).$


\begin{figure}[ht]
    \centering  
\tikzset{every picture/.style={line width=0.75pt}} 

\begin{tikzpicture}[x=0.75pt,y=0.75pt,yscale=-0.7,xscale=0.7,]

\draw    (30,61) -- (490,61) -- (490,60.07) ;
\draw [shift={(490,60)}, rotate = 178.75] [fill={rgb, 255:red, 0; green, 0; blue, 0 }  ][line width=0.08]  [draw opacity=0] (8.93,-4.29) -- (0,0) -- (8.93,4.29) -- cycle    ;
\draw [line width=2.25]    (30,62) -- (240,62) ;
\draw [color={rgb, 255:red, 0; green, 50; blue, 255 }  ,draw opacity=1 ][line width=2.25]    (240,62) -- (480,62) ;
\draw    (240,68) -- (240,61) ;

\draw (80,21) node [anchor=north west][inner sep=0.75pt]  [font=\footnotesize] [align=left] {\ CD Impossible};
\draw (300,11) node [anchor=north west][inner sep=0.75pt]  [font=\footnotesize] [align=left] {Counting edges\\\ \ \ \ \ succeeds};
\draw (490,64.4) node [anchor=north west][inner sep=0.75pt]  [font=\small]  {$\beta $};
\draw (210,73.4) node [anchor=north west][inner sep=0.75pt]  [font=\footnotesize]  {$\tfrac{2+2\alpha + \gamma }{4}$};
\end{tikzpicture}

    \caption{Phase-transition diagram for $\mathrm{PDS}(k,p,q)$, with $k = n^{\beta}$, $p = 1 - n^{-\gamma}$ and $q = n^{\alpha}$.}
    \label{fig:PDS}
  
\end{figure}
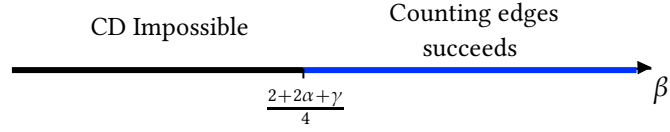

 
\subsection*{Planted clique and planted independent set.}
When we choose $q = 1$ and $p=1/2$ for $\mathrm{PDS}(k,p,q),$ the setting simply corresponds to the detection task of the planted clique problem. Corollary \ref{cor:const} then directly yields the well-known $k=\Theta(\sqrt{n})$ computational phase transition for constant-degree polynomials for planted clique \cite{alon1998finding, barak2019nearly}.

Interestingly, for $p = 1 - \tfrac{d}{n}, d=o(n)$ and $q=1,$ $\mathrm{PDS}(k,p,q)$ now maps to the detection task of the planted independent set problem. Indeeed, $\mathrm{PDS}(k,1-d/n,1)$ is the detection setting between $G(n,1-d/n)$ and $G(n,1-d/n)$ union a random $k$-clique. By equivalently considering the complements of the observation graphs, we need to detect between $G(n,d/n)$ and $G(n,d/n)$ where $k$ vertices are constrained to induce an independent set, known as the planted independent set model. Corollary \ref{cor:const} implies that constant degree polynomials can detect if and only if $k = \omega(n^{3/4}/d^{1/4})$. Besides its widespread study in the literature, we are not aware of any other result studying the computational hardness of detecting a planted independent set (while see e.g., when $d=O(1)$ \cite{wein2022optimal} for low-degree lower bounds for estimation, and  \cite{jones2022sum} for certification).


\subsection*{Planted bipartite clique.}
Given the previous examples where counting edges is always constant-degree optimal, one may wonder if the count of large stars is the optimal constant-degree polynomial in some natural setting. For this reason, we now consider a planted bipartite clique setting $\mathrm{PBC}(a,b,p)$ for $a=a_n,b=b_n \in \mathbb{N}$ and $p=p_n \in (0,1)$ with $a \ge b$, where we simply choose the planted $H$ to be the bipartite clique $K_{a,b}$.
The maximum degree of $H$ is then clearly $\Delta = a$. Using Theorem~\ref{thm:degree-characterization} we arrive at the following richer computational diagram (see Figure \ref{fig:BPC}).
\begin{corollary}
    For $\alpha, \beta, \gamma \in (0,1)$, if $p = 1 - n^{-\gamma}$, $a = n^{\alpha}$, $b = n^{\beta}$ we have for $\mathrm{PBC}(a,b,p)$:
    \begin{itemize}
    \item If $2\alpha \leq 1+\gamma$ and $\beta \leq \tfrac{1}{2}$, constant-degree polynomial test fails;
    \item  $\beta > \tfrac{1}{2}$ if and only if counting edges achieves strong separation;
    \item If  $2\alpha > 1 + \gamma + \epsilon$ for some constant $\varepsilon > 0$, counting large stars achieves strong separation.
\end{itemize}
\end{corollary}

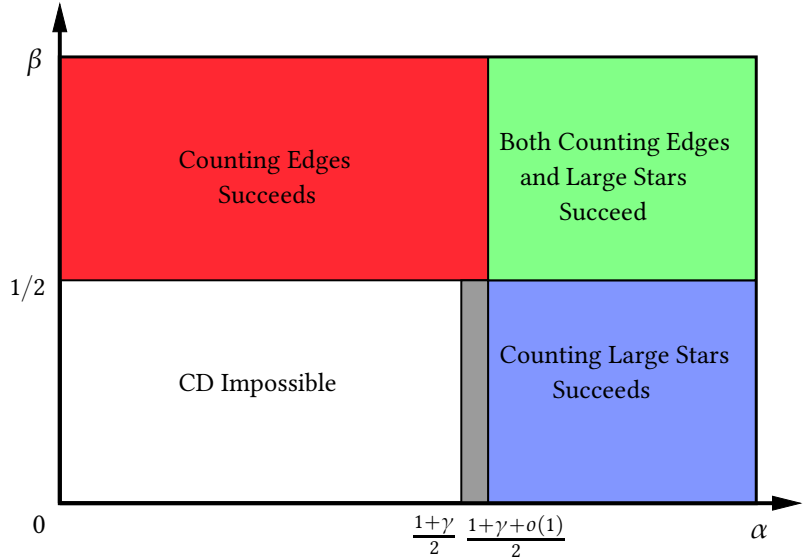
\begin{figure}
    \centering

\tikzset{every picture/.style={line width=0.75pt}} 

\begin{tikzpicture}[x=0.75pt,y=0.75pt,yscale=-0.9,xscale=0.9]


\draw  [line width=1.5]  (55,50) -- (445,50) -- (445,300) -- (55,300) -- cycle ;
\draw  [fill={rgb, 255:red, 130; green, 150; blue, 255 }  ,fill opacity=0.4 ] (295,175) -- (445,175) -- (445,300) -- (295,300)  -- cycle ;
\draw  [fill={rgb, 255:red, 130; green, 255; blue, 135 }  ,fill opacity=0.4 ] (295,50) -- (445,50) -- (445,175) -- (295,175)  -- cycle ;
\draw  [fill={rgb, 255:red, 255; green, 40; blue, 50 }  ,fill opacity=0.4 ] (55,50) -- (295,50) -- (295,175) -- (55,175) -- cycle ;
\draw  [fill={rgb, 255:red, 155; green, 155; blue, 155 }  ,fill opacity=0.5 ] (280,175) -- (295,175) -- (295,300) -- (280,300) -- cycle ;
\draw [line width=1.5]    (55,300) -- (455,300) ;
\draw [shift={(470,300)}, rotate = 180] [fill={rgb, 255:red, 0; green, 0; blue, 0 }  ][line width=0.08]  [draw opacity=0] (15.6,-3.9) -- (0,0) -- (15.6,3.9) -- cycle    ;

\draw [line width=1.5]    (55,300) -- (55,24) ;
\draw [shift={(55,20)}, rotate = 90] [fill={rgb, 255:red, 0; green, 0; blue, 0 }  ][line width=0.08]  [draw opacity=0] (15.6,-3.9) -- (0,0) -- (15.6,3.9) -- cycle    ;


\footnotesize\draw (300,210) node [anchor=north west][inner sep=0.75pt]   [align=left] { Counting Large Stars\\ \ \ \ \ \ \ \ \ Succeeds};
\footnotesize\draw (300,90) node [anchor=north west][inner sep=0.75pt]   [align=left] { Both Counting Edges \\ \ \ \ and Large Stars\\ \ \ \ \ \ \ \ \ \ Succeed};
\footnotesize\draw (120,100) node [anchor=north west][inner sep=0.75pt]   [align=left] {Counting Edges\\ \ \ \ \ \  \ Succeeds};
\footnotesize\draw (120,225) node [anchor=north west][inner sep=0.75pt]   [align=left] {CD Impossible};
\draw (250,305) node [anchor=north west][inner sep=0.75pt]    {\small $\frac{1+\gamma}{2}$};
\draw (280,305) node [anchor=north west][inner sep=0.75pt]    {\small $\frac{1+\gamma+o(1)}{2}$};
\draw (440,310) node [anchor=north west][inner sep=0.75pt]  [font=\small]  {$\alpha$};
\draw (38,305) node [anchor=north west][inner sep=0.75pt]  [font=\small]  {$0$};
\draw (35,43.4) node [anchor=north west][inner sep=0.75pt]  [font=\small]  {$\beta$};
\draw (25,170.4) node [anchor=north west][inner sep=0.75pt]  [font=\footnotesize]  {$1/2$};
\end{tikzpicture}
    \caption{
    Phase-transition diagram of planted bipartite clique, where $H = K_{a,b}$ with $a \geq b$. We use parameter configuration $p = 1 - n^{-\gamma}$, $a = n^\alpha$ and $b = n^{\beta}$.}
    \label{fig:BPC}
    
\end{figure}

\section{Tightness of Theorem \ref{thm:main}} \label{subsection:failure-of-counting-stars}
In this section, we prove the tightness of our main result (Theorem \ref{thm:main}). Recall, that according to Theorem \ref{thm:main}, under the assumptions $p = \Omega(1)$ and $D = O(1)$, the degree-$D$ polynomials achieve strong separation for detecting a planted subgraph in $G(n, p)$ if and only if the signed count polynomial of a $t \leq D$-star does so too. We provide counterexamples showing that if either $p = \Omega(1)$ or $D = O(1)$ is not satisfied, counting stars could fail to achieve strong separation even when some other degree-$D$ polynomial does so.

\subsection*{Failure of counting stars under vanishing $p$.}  Assume $p = n^{-\gamma}$ for a constant $\gamma \in (0,1).$ Then we show that for some appropriate constant size $k$, the planted detection model with $H$ being a $k$-clique has the following behavior.
\begin{lemma} \label{lemma:small-p}
    For any $p = n^{-\gamma}$ where $\gamma \in (0, 1)$ is a constant, there exists a constant $k$ such that, for the planted subgraph detection task per Definition \ref{dfn:planted} with $H$ being a $k$-clique, counting constant-sized stars fails to achieve strong separation, whereas some constant-degree polynomial test achieves strong separation.
\end{lemma}
\subsection*{Failure of counting stars under non-constant degree $D$.}
Here, we fix $p = \frac{1}{2}$ and $k = C\sqrt{n}$, where $C > 0$ is a large enough constant. By focusing on the performance of $D=O(\log n)$-degree polynomials for the planted subgraph detection task where $H$ is a $k$-clique the following holds.
\begin{lemma} \label{lemma:large-D}
    Let $H$ be a clique of size $C\sqrt{n}$ and $p = 1/2$, where $C>0$ is any constant. Consider the planted subgraph detection task where $H$ is planted in $G(n,p)$. Then, for any $t=t_n \geq 1$ (potentially growing with $n$), counting $t$-stars fails to achieve strong separation. Yet, for large enough $C>0,$ an $O(\log n)$-degree polynomial achieves strong separation.
\end{lemma}



\section{Getting Started for Proof Sections}

We now move to the proof sections of our results. We start by introducing some necessary definitions and notations.

Let $n,k \in \mathbb{N}$ be natural numbers. We will denote $[n] := \{1, 2, \dots, n\}$. We will use $\binom{n}{k}$ to denote the number of ways to choose $k$ elements from $n$ elements, and $n_{(k)} := n(n-1)\dots (n-k+1)$ to denote the $k$-th falling factorial of $n$. By convention, the $0$-th falling factorial of any number is equal to $1$ and $n_{(k)} = 0$ for any $k > n$. Let $X$ be set. We will denote $\binom{X}{k} := \{A \subseteq X: |A| = k\}$.

\subsection*{Asymptotic notations} We will use standard asymptotic notations $O, o, \Omega, \omega, \Theta$. Let $(a_n)_{n \in \mathbb{N}}$, $(b_n)_{n \in \mathbb{N}}$ be two sequences of real numbers. We sometimes write $a_n \lesssim b_n$ when $a_n = O(b_n)$, $a_n \ll b_n$ when $a_n = o(n)$, $a_n \gtrsim b_n$ when $a_n = \Omega(b_n)$, and $a_n \gg b_n$ when $a_n = \omega(b_n)$.

\subsection*{Graph theory basics}
A graph is a pair of $(V, E)$ where $V$ is the vertex set and $E \subseteq \binom{V}{2}$ is the edge set. Sometimes we will identify the vertex set $V$ with $[n]$. $K_n$ will denote the complete graph on $n$ vertices, and $K_{s,t}$ will be the complete bipartite graph with one part having $s$ vertices and the other having $t$ vertices.

We say $H$ is a subgraph of $G$, denoted as $H \subseteq G$, if $V(H) \subseteq V(G)$ and $E(H) \subseteq E(G)$. We say $H$ is a \emph{spanning} subgraph of $G$, if it is a subgraph of $G$ and $V(H) = V(G)$. We say a graph $G$ is \emph{edge-induced} if there are no isolated vertices in $G$.

A homomorphism from graph $G_1$ to graph $G_2$ is a mapping $f: V(G_1) \to V(G_2)$ that preserves adjacency relation, i.e., $\{f(u), f(v)\} \in E(G_2)$ whenever $\{u,v\} \in E(G_1)$. An isomorphism is a homomorphism that is bijective, and we use $G_1 \cong G_2$ to denote that $G_1$ and $G_2$ are isomorphic. An automorphism of a graph $G$ is an isomorphism from $G$ to itself. The set of automorphisms of a graph $G$ forms a group under composition called the automorphism group of $G$, denoted as $\text{Aut}(G)$.

\subsection*{Subgraph copies: notation}
We will use throughout a notion of shape and labelled/unlabelled copies of a shape.
\begin{definition}
    A shape $\mathbf{S}$ is an edge-induced graph. By a slight abuse of notation, we sometimes use $\mathbf{S}$ to refer to the edge set $E(\mathbf{S})$ as an edge-induced subgraph is determined by its edge set.

    Let $\mathbf{S}$ be a shape, and $G$ be a graph. An unlabelled copy of $\mathbf{S}$ in $G$ is a subgraph $S\subseteq G$ such that $\mathbf{S}$ is isomorphic to $S$. A labelled copy of $\mathbf{S}$ in $G$ is a pair $(S, \gamma)$ of a subgraph $S \subseteq G$ together with a labelling $\gamma: V(S) \to V(\mathbf{S})$, such that $\gamma$ is an isomorphism from $S$ to $\mathbf{S}$.

    Unless stated otherwise, in this paper we will use copies to refer to labelled copies.
\end{definition}

Following the notation \cite{mossel2022second} we define for a shape $\mathbf{S}$ and a graph $G,$
\begin{align}\label{dfn:copies}
    M_{\mathbf{S}, G} := \# \{\text{copies of } \mathbf{S} \text{ inside } G \},
\end{align} and when $n$ is clear from context, $M_{\mathbf{S}} := M_{\mathbf{S}, K_n} = n_{(|V(\mathbf{S})|)}$.

\begin{example}
    In the complete graph $K_n$, the number of (labelled) copies of a triangle $K_3$ is $n_{(3)} = n(n-1)(n-2)$, whereas the number of triangles is $\binom{n}{3} = \frac{n(n-1)(n-2)}{6}$. 
\end{example}


\begin{definition}[``Isomorphisic pairs of copies''] \label{def:isomorphism-pair}
    Let $\mathbf{S}_1$ and $\mathbf{S}_2$ be two shapes, and $G$ be a graph. Let $((S_1, \gamma_1), (S_2, \gamma_2))$ and $((\hat{S}_1, \hat{\gamma}_1), (\hat{S}_2, \hat{\gamma}_2))$ be two pairs of copies of $\mathbf{S}_1$ and copies of $\mathbf{S}_2$ in $G$. We say that $((S_1, \gamma_1), (S_2, \gamma_2))$ and $((\hat{S}_1, \hat{\gamma}_1), (\hat{S}_2, \hat{\gamma}_2))$ are isomorphic if
    \begin{itemize}
        \item For every $u \in V(S_1) \cap V(S_2)$, $\hat{\gamma}_1^{-1} (\gamma_1 (u)) = \hat{\gamma}_2^{-1} (\gamma_2 (u))$.
        \item For every $v \in V(\hat{S}_1) \cap V(\hat{S}_2)$, $\gamma_1^{-1} (\hat{\gamma}_1 (u)) = \gamma_2^{-1} (\hat{\gamma}_2 (u))$.
    \end{itemize}
\end{definition}

\begin{example}
    To illusrate Definition \ref{def:isomorphism-pair}, consider the two shapes $\mathbf{S}_1$ and $\mathbf{S}_2$, and the $4$ pairs of copies of them in Figure \ref{fig:isomorphism-pair}. $(1)$ and $(2)$ are isomorphic, as the pairs of labellings on all the vertices in the intersection are the same. $(1)$ and $(3)$ are not isomorphic, as in $(1)$ a vertex in the intersection has a pair of labels $(b,g)$ but in $(3)$ a vertex in the intersection has a pair of labels $(b,f)$. $(1)$ and $(4)$ are not isomorphic, as in $(1)$ the two vertex sets intersect whereas in $(4)$ the two vertex sets are disjoint.
\end{example}

\begin{figure}[h]
    \centering

\tikzset{every picture/.style={line width=0.75pt}} 

\begin{tikzpicture}[x=0.75pt,y=0.75pt,yscale=-1,xscale=1]

\draw    (311.75,50.75) -- (391.25,50.25) ;
\draw    (311.75,50.75) -- (350.75,111.25) ;
\draw    (350.75,111.25) -- (391.25,50.25) ;
\draw    (181.5,51) -- (181.25,110.25) ;
\draw    (181.5,51) -- (241.75,50.75) ;
\draw    (241.75,50.75) -- (241.75,110.25) ;
\draw    (30.5,210) -- (30.25,250.75) ;
\draw    (30.5,210) -- (69.75,209.75) ;
\draw    (69.75,209.75) -- (70.25,249.25) ;
\draw    (69.75,209.75) -- (106.25,227.25) ;
\draw    (70.25,249.25) -- (106.25,227.25) ;
\draw    (180.51,200.5) -- (223.68,200.08) ;
\draw    (180.51,200.5) -- (180.37,237.12) ;
\draw    (180.37,237.12) -- (222.21,237.41) ;
\draw    (180.37,237.12) -- (199.01,271.09) ;
\draw    (222.21,237.41) -- (199.01,271.09) ;
\draw    (531.43,249.98) -- (531.59,207.2) ;
\draw    (531.43,249.98) -- (571.98,229.3) ;
\draw    (571.98,229.3) -- (531.59,207.2) ;
\draw    (461.48,209.97) -- (461.31,249.44) ;
\draw    (461.48,209.97) -- (501.54,209.8) ;
\draw    (501.54,209.8) -- (501.54,249.44) ;
\draw    (311,210) -- (310.75,250.75) ;
\draw    (311,210) -- (350.25,209.75) ;
\draw    (350.25,209.75) -- (350.75,249.25) ;
\draw    (350.25,209.75) -- (386.75,227.25) ;
\draw    (350.75,249.25) -- (386.75,227.25) ;

\draw (169,99.9) node [anchor=north west][inner sep=0.75pt]  [color={rgb, 255:red, 74; green, 144; blue, 226 }  ,opacity=1 ]  {$a$};
\draw (169.5,35.4) node [anchor=north west][inner sep=0.75pt]  [color={rgb, 255:red, 74; green, 144; blue, 226 }  ,opacity=1 ]  {$b$};
\draw (244.5,32.9) node [anchor=north west][inner sep=0.75pt]  [color={rgb, 255:red, 74; green, 144; blue, 226 }  ,opacity=1 ]  {$c$};
\draw (355.5,104.4) node [anchor=north west][inner sep=0.75pt]  [color={rgb, 255:red, 208; green, 2; blue, 27 }  ,opacity=1 ]  {$e$};
\draw (244,103.4) node [anchor=north west][inner sep=0.75pt]  [color={rgb, 255:red, 74; green, 144; blue, 226 }  ,opacity=1 ]  {$d$};
\draw (301.5,31.9) node [anchor=north west][inner sep=0.75pt]  [color={rgb, 255:red, 208; green, 2; blue, 27 }  ,opacity=1 ]  {$f$};
\draw (394.5,30.9) node [anchor=north west][inner sep=0.75pt]  [color={rgb, 255:red, 208; green, 2; blue, 27 }  ,opacity=1 ]  {$g$};
\draw (202,122.4) node [anchor=north west][inner sep=0.75pt]    {$\mathbf{S}_{1}$};
\draw (341,121.9) node [anchor=north west][inner sep=0.75pt]    {$\mathbf{S}_{2}$};
\draw (59,245.9) node [anchor=north west][inner sep=0.75pt]  [color={rgb, 255:red, 74; green, 144; blue, 226 }  ,opacity=1 ]  {$a ,\textcolor[rgb]{0.82,0.01,0.11}{f}$};
\draw (60.5,189.9) node [anchor=north west][inner sep=0.75pt]  [color={rgb, 255:red, 74; green, 144; blue, 226 }  ,opacity=1 ]  {$b ,\textcolor[rgb]{0.82,0.01,0.11}{g}$};
\draw (20.5,191.4) node [anchor=north west][inner sep=0.75pt]  [color={rgb, 255:red, 74; green, 144; blue, 226 }  ,opacity=1 ]  {$c$};
\draw (18.5,246.9) node [anchor=north west][inner sep=0.75pt]  [color={rgb, 255:red, 74; green, 144; blue, 226 }  ,opacity=1 ]  {$d$};
\draw (111.5,214.9) node [anchor=north west][inner sep=0.75pt]  [color={rgb, 255:red, 208; green, 2; blue, 27 }  ,opacity=1 ]  {$e$};
\draw (150,231.9) node [anchor=north west][inner sep=0.75pt]  [color={rgb, 255:red, 74; green, 144; blue, 226 }  ,opacity=1 ]  {$b ,\textcolor[rgb]{0.82,0.01,0.11}{g}$};
\draw (226,229.9) node [anchor=north west][inner sep=0.75pt]  [color={rgb, 255:red, 74; green, 144; blue, 226 }  ,opacity=1 ]  {$a ,\textcolor[rgb]{0.82,0.01,0.11}{f}$};
\draw (192.5,269.9) node [anchor=north west][inner sep=0.75pt]  [color={rgb, 255:red, 208; green, 2; blue, 27 }  ,opacity=1 ]  {$e$};
\draw (167.5,184.9) node [anchor=north west][inner sep=0.75pt]  [color={rgb, 255:red, 74; green, 144; blue, 226 }  ,opacity=1 ]  {$c$};
\draw (227,184.4) node [anchor=north west][inner sep=0.75pt]  [color={rgb, 255:red, 74; green, 144; blue, 226 }  ,opacity=1 ]  {$d$};
\draw (491.15,245.01) node [anchor=north west][inner sep=0.75pt]  [color={rgb, 255:red, 74; green, 144; blue, 226 }  ,opacity=1 ]  {$a$};
\draw (489.49,191.04) node [anchor=north west][inner sep=0.75pt]  [color={rgb, 255:red, 74; green, 144; blue, 226 }  ,opacity=1 ]  {$b$};
\draw (448.36,191.88) node [anchor=north west][inner sep=0.75pt]  [color={rgb, 255:red, 74; green, 144; blue, 226 }  ,opacity=1 ]  {$c$};
\draw (446.52,245.34) node [anchor=north west][inner sep=0.75pt]  [color={rgb, 255:red, 74; green, 144; blue, 226 }  ,opacity=1 ]  {$d$};
\draw (574,210.9) node [anchor=north west][inner sep=0.75pt]  [color={rgb, 255:red, 208; green, 2; blue, 27 }  ,opacity=1 ]  {$e$};
\draw (534.5,247.4) node [anchor=north west][inner sep=0.75pt]  [color={rgb, 255:red, 208; green, 2; blue, 27 }  ,opacity=1 ]  {$f$};
\draw (537.5,188.4) node [anchor=north west][inner sep=0.75pt]  [color={rgb, 255:red, 208; green, 2; blue, 27 }  ,opacity=1 ]  {$g$};
\draw (339.5,246.4) node [anchor=north west][inner sep=0.75pt]  [color={rgb, 255:red, 74; green, 144; blue, 226 }  ,opacity=1 ]  {$a ,\textcolor[rgb]{0.82,0.01,0.11}{g}$};
\draw (341,189.9) node [anchor=north west][inner sep=0.75pt]  [color={rgb, 255:red, 74; green, 144; blue, 226 }  ,opacity=1 ]  {$b ,\textcolor[rgb]{0.82,0.01,0.11}{f}$};
\draw (301,191.4) node [anchor=north west][inner sep=0.75pt]  [color={rgb, 255:red, 74; green, 144; blue, 226 }  ,opacity=1 ]  {$c$};
\draw (299,246.9) node [anchor=north west][inner sep=0.75pt]  [color={rgb, 255:red, 74; green, 144; blue, 226 }  ,opacity=1 ]  {$d$};
\draw (392,214.9) node [anchor=north west][inner sep=0.75pt]  [color={rgb, 255:red, 208; green, 2; blue, 27 }  ,opacity=1 ]  {$e$};
\draw (40.5,299.9) node [anchor=north west][inner sep=0.75pt]    {$(1)$};
\draw (189.5,298.9) node [anchor=north west][inner sep=0.75pt]    {$(2)$};
\draw (329.5,299.4) node [anchor=north west][inner sep=0.75pt]    {$(3)$};
\draw (499,298.9) node [anchor=north west][inner sep=0.75pt]    {$(4)$};

\end{tikzpicture}
    \caption{An example of $4$ pairs of copies of $\mathbf{S}_1$ and $\mathbf{S}_2$ numbered by $(1), (2), (3), (4)$, where the labellings of copies of $\mathbf{S}_1$ are marked with blue letters, and the labellings of copies of $\mathbf{S}_2$ are marked with red letters.}
    \label{fig:isomorphism-pair}
\end{figure}
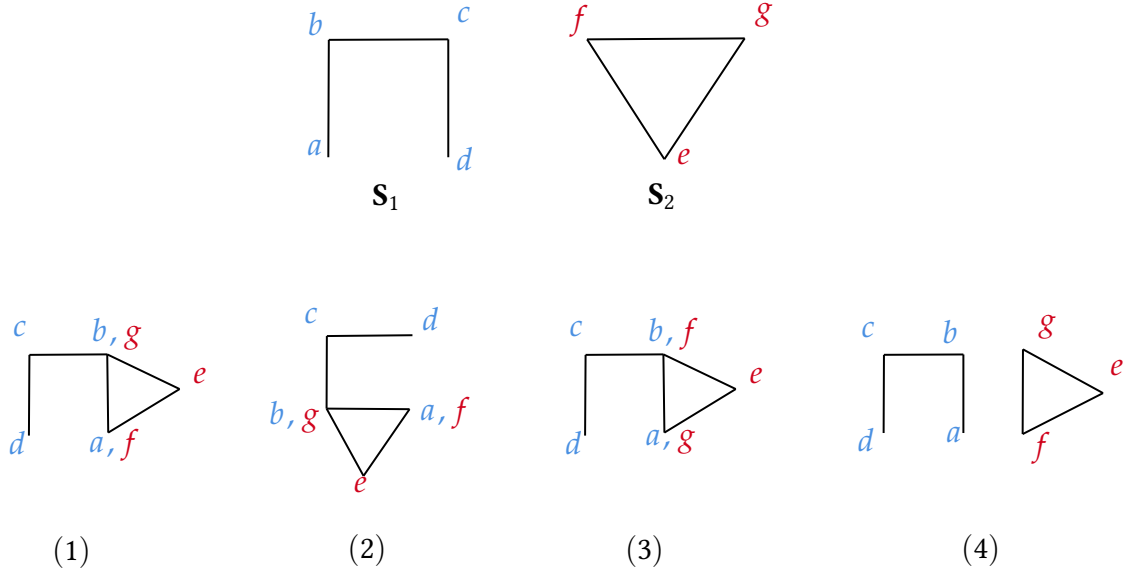

\begin{remark}
    It is not hard to see that the isomorphism of pairs of copies in Definition \ref{def:isomorphism-pair} is an equivalence relation. Therefore, the collection of pairs of copies $((S_1, \gamma_1), (S_2, \gamma_2))$ of $\mathbf{S}_1$ and $\mathbf{S}_2$ in any graph $G$ can be partitioned into isomorphism classes, i.e., equivalence classes defined by the isomorphism relation.
\end{remark}

\begin{definition}[``Intersecting pattern''] \label{def:intersecting-pattern}
    Let $\mathbf{S}_1$ and $\mathbf{S}_2$ be two shapes. An intersecting pattern is an isomorphism class of the pairs of copies of $\mathbf{S}_1$ and $\mathbf{S}_2$.

    The set of all intersecting patterns of $\mathbf{S}_1$ and $\mathbf{S}_2$ is denoted as $\text{Inter}(\mathbf{S}_1, \mathbf{S}_2)$, and an element in it is denoted as $i(S_1, S_2)$. We use non-bold $i(S_1, S_2)$ in order to distinguish a specific intersecting pattern from the given shapes $\mathbf{S}_1$ and $\mathbf{S}_2$.
\end{definition}

\begin{remark}
    Note that an intersecting pattern is all about the way the vertex sets of two labelled copies of two shapes intersect. Let $\mathbf{S}_1$ and $\mathbf{S}_2$ be shapes, with $s_1 = |V(\mathbf{S}_1)|$ and $s_2 = |V(\mathbf{S}_2)|$. The number of intersecting patterns of pairs of copies of $\mathbf{S}_1$ and $\mathbf{S}_2$ is $\sum_{k} \binom{s_1}{k} \binom{s_2}{k} k!$, which counts the total number of ways to choose a subset of $V(\mathbf{S}_1)$ of size $k$, a subset of $V(\mathbf{S}_2)$ of size $k$, and a bijective function (that specifies how the vertices are "glued" in the intersection) between these two subsets, for all values of $k$. Importantly, if $\mathbf{S}$ is a shape with a constant number of edges (thus a constant number of vertices), then the number of intersecting patterns of pairs of copies of $\mathbf{S}$ with $\mathbf{S}$ is also bounded by a constant.
\end{remark}

\begin{definition}\label{dfn:sym-diff-shape}
    Let $\mathbf{S}_1$ and $\mathbf{S}_2$ be shapes, and $i(S_1, S_2)$ be an intersecting pattern of pairs of copies of $\mathbf{S}_1$ and $\mathbf{S}_2$. For $i(S_1, S_2)$, we define the symmetric difference shape $S_1 \triangle S_2$ as the shape obtained by first taking the symmetric difference of the edge sets of a pair of copies with the intersecting pattern $i(S_1, S_2)$ and deleting isolated vertices after the symmetric difference operation. Similarly, we define the union shape $S_1 \cup S_2$ as the shape obtained by taking the union of the edge sets of a pair of copies with the intersecting pattern $i(S_1, S_2)$.
\end{definition}

\begin{remark}
    Let $\mathbf{S}_1$ and $\mathbf{S}_2$ be shapes, $i(S_1, S_2)$ be an intersecting pattern, and $G$ be a graph. The number of pairs of copies of $\mathbf{S}_1$ and $\mathbf{S}_2$ with the intersecting pattern $i(S_1, S_2)$ in $G$ is equal to $M_{S_1 \cup S_2, G}$.
\end{remark}

\section{Proof of Main Theorem }\label{sec:proof}

\subsection{Proof Strategy} \label{subsection:proof-strategy}

We briefly describe our proof strategy for Theorem \ref{thm:main}. The first bullet point is a standard result in the literature of low-degree polynomials, c.f.~\cite[Lemma 7.3]{coja2022statistical}, and our focus will be on the second bullet point of Theorem \ref{thm:main}.

The first key step of the proof is proving that if the low-degree advantage $\text{Adv}^{\le D}$ explodes, then the advantage of the count of a $t$-star $\text{Adv}^{\le D}(f_{\mathbf{K}_{1,t}})$ explodes for some $t\leq D$. It is easy to see that if $\text{Adv}^{\le D}=\omega(1)$ for some $D=O(1)$ then the count of some subgraph $\mathbf{S}$ with at most $D$ edges satisfies $\text{Adv}^{\le D}(f_{\mathbf{S}})=\omega(1)$. This follows by expanding the low-degree advantage (Proposition \ref{prop:Adv}) to get $(\text{Adv}^{\le D})^2 \approx \sum_{\mathbf{S} \leq D} (\text{Adv}^{\le D}(f_{\mathbf{S}}) )^2$ and using that there are finitely many unlabelled subgraphs with at most $D$ edges. Now by using a moment calculation and the subgraph copies notation \eqref{dfn:copies} one can conclude for all $\mathbf{S}$, $$\text{Adv}^{\le D}(f_{\mathbf{S}}) \approx \frac{M_{\mathbf{S}, H}^2}{M_{\mathbf{S}}} \left(\frac{1-p}{p}\right)^{|\mathbf{S}|}.$$ Using now a careful recursive combinatorial argument we then show in Proposition \ref{prop:star-shape} that as long as for some $\mathbf{S},$ $\frac{M_{\mathbf{S}, H}^2}{M_{\mathbf{S}}} \left(\frac{1-p}{p}\right)^{|\mathbf{S}|}=\omega(1)$ then for some star shape $\mathbf{S} = \mathbf{K}_{1,t}$ it also holds
 $\frac{M_{\mathbf{S}, H}^2}{M_{\mathbf{S}}} \left(\frac{1-p}{p}\right)^{|\mathbf{S}|} =\omega(1),$ i.e., $\text{Adv}^{\le D}(f_{\mathbf{K}_{1,t}})=\omega(1).$ This is the key step that the importance of star counts arises. 

The second now step is to prove that the star subgraph with exploding advantage achieves strong separation. The exploding advantage already implies by definition part of the strong separation result, that is
\[\sqrt{\Var_{\mathbb{Q}}[f_{\mathbf{K}_{1,t}}]} = o\left(|\mathbb{E}_{\mathbb{P}}[f_{\mathbf{K}_{1,t}}] - \mathbb{E}_{\mathbb{Q}}[f_{\mathbf{K}_{1,t}}]|\right).\]
To show the other side of the strong separation,
\[\sqrt{\Var_{\mathbb{P}}[f_{\mathbf{K}_{1,t}}]} = o\left(|\mathbb{E}_{\mathbb{P}}[f_{\mathbf{K}_{1,t}}] - \mathbb{E}_{\mathbb{Q}}[f_{\mathbf{K}_{1,t}}]|\right).\]
we need to show equivalently 
\begin{align*}
        \frac{\mathbb{E}_{\mathbb{P}}[f_{\mathbf{K}_{1,t}}^2]}{\mathbb{E}_{\mathbb{P}}[f_{\mathbf{K}_{1,t}}^2]} \le 1 + o(1),
    \end{align*}which is a second moment method comprising the technical bulk of the proof (recall $\mathbb{P}$ is \textbf{not} a product distribution and depends on the planted arbitrarily structured $H$). Careful again calculations along with combinatorial arguments allow us to conclude that the second moment condition indeed holds for all planted $H.$  We highlight that this is a key technical part of the proof and the simplicity of the star shape appears essential. For example, one way this manifests itself is that there are only a few different cases that two star graphs can be intersecting (see e.g., Figures \ref{fig:intersecting-shape-distinct-root}, \ref{fig:vertex-partitioning-of-fig1} and \ref{fig:intersecting-shape-sharing-root} below) which greatly simplifies the second moment expansion (see Proposition \ref{prop:intersection-ratio-bound} for the key step leading to the second moment method).

\subsection{Key Lemmas}

Recall the notation $M_{\mathbf{S},G}$ in \eqref{dfn:copies}. The first key proposition proves that an exploding advantage implies that for some star shape $\mathbf{S}$ the quantity $ \frac{M_{\mathbf{S}, H }^2}{M_{\mathbf{S} }} \left(\frac{1-p}{p}\right)^{|\mathbf{S} |}$ is exploding. This quantity in simple the rescaled squared advantage of star count $f_{\mathbf{S}}$. 
\begin{proposition}\label{prop:star-shape}
    Suppose $D = O(1)$ and $p = \Omega(1)$. If $\text{Adv}^{\le D} \to \infty$, then among star graphs with at most $D$ edges \[\max_{\substack{\mathbf{S} \cong \mathbf{K}_{1,t}:\\ t \le D} } \frac{M_{\mathbf{S}, H }^2}{M_{\mathbf{S} }} \left(\frac{1-p}{p}\right)^{|\mathbf{S} |} \to \infty.\]
\end{proposition}

The next is an important and quite technical proposition that reveals some structure between overlapping copies of the star shape $\mathbf{S}$ with exploding advantage from Proposition \ref{prop:star-shape}. This proposition is the key step in a second moment calculation used in the proof of Theorem \ref{thm:main}.

\begin{proposition} \label{prop:intersection-ratio-bound} 
    Suppose $D = O(1)$ and $p = \Omega(1)$. Let $\mathbf{S} \cong \mathbf{K}_{1,t}$ be a star shape that maximizes $\frac{M_{\mathbf{S} ,H}^2}{M_{\mathbf{S} }} \left(\frac{1-p}{p}\right)^{|\mathbf{S}|}$ among star shapes with at most $D$ edges, and suppose $\frac{M_{\mathbf{S} ,H}^2}{M_{\mathbf{S} }} \left(\frac{1-p}{p}\right)^{|\mathbf{S}|} \to \infty$. Let $i(S_1, S_2) \in \text{Inter}(\mathbf{S}, \mathbf{S})$ be an intersecting pattern such that $S_1$ and $S_2$ have non-empty intersection. Then, 
    \begin{align*}
        \frac{n^{|V(S_1 \cup S_2)| - |V(S_1 \triangle S_2)| } M_{S_1 \triangle S_2, H}\left(\frac{1-p}{p}\right)^{|S_1 \triangle S_2|/2} }{M_{\mathbf{S} ,H}^2 \left(\frac{1-p}{p}\right)^{|\mathbf{S}|}} = o(1).
    \end{align*}
\end{proposition}

\subsection{Auxillary Lemmas}
In this section, we have a series of auxiliary lemmas needed for the proof of Theorem \ref{thm:main}.

We start with a standard result, which we state as a fact.
\begin{fact}
    The Walsh-Fourier basis $(\chi_{S})_{S \subseteq \binom{[n]}{2}}$ (the degree-$D$ Walsh-Fourier basis $(\chi_{S})_{S \subseteq \binom{[n]}{2}: |S| \le D}$ resp.) with respect to the Erd\H{o}s-R\'{e}nyi distribution $\mathbb{Q} := G(n,p)$ form an orthonormal basis for $\mathbb{R}[X]$ ($\mathbb{R}[X]^{\le D}$ resp.) with respect to the inner product $\langle \cdot , \cdot \rangle_{\mathbb{Q}}$ defined by $\langle f, g \rangle_{\mathbb{Q}} = \mathbb{E}_{\mathbb{Q}}[fg]$.
\end{fact}

The following lemma is taken from \cite{mossel2022second}, and will be useful throughout the calculations in this paper. Recall the definition of $M_{\mathbf{S},H}$ in \eqref{dfn:copies}.
\begin{lemma}\label{lemma:inclusion-prob}
    For a fixed shape $S$, and $(G, \bold{H})$ drawn from $\mathbb{P}$, 
    \begin{align*}
        \mathbb{P}(S \subseteq \bold{H}) = 
\frac{M_{S,H}}{M_S}
    \end{align*}
\end{lemma}

\begin{proof}[Proof of Lemma \ref{lemma:inclusion-prob}]
    \begin{align*}
        M_{S,H} = \mathbb{E}[\# \{\text{copies of } S \text{ inside } \bold{H}\} ] = M_S \cdot \mathbb{P}(S \subseteq \bold{H}).
    \end{align*}
\end{proof}

The following proposition calculates the advantage for all planted subgraph detection tasks.
\begin{proposition}\label{prop:Adv}
    For the planted subgraph detection task, the square of the degree-$D$ advantage for testing distribution $\mathbb{P}$ against distribution $\mathbb{Q}$ is
    \begin{align*}
        \left(\text{Adv}^{\le D}\right)^2 = \sum_{\mathbf{S} : |\mathbf{S}| \le D } \frac{M_{\mathbf{S}, H}^2}{M_\mathbf{S} \cdot |\text{Aut}(\mathbf{S})|} \left(\frac{1-p}{p}\right)^{|\mathbf{S}|},
    \end{align*}
    where the sum is over (the isomorphism classes of) the shapes $\mathbf{S}$ with at most $D$ edges.
\end{proposition}

The following proposition gives a way of double counting the number of pairs of copies of two shapes in a graph in terms of their intersecting patterns.
\begin{proposition}\label{prop:total-intersecting}
    Let $\mathbf{S}_1$ and $\mathbf{S}_2$ be shapes, and $G$ be a graph. Then,
    \begin{align*}
        M_{\mathbf{
S}_1, G} M_{\mathbf{S}_2, G} = \sum_{i(S_1, S_2) \in \text{Inter}(\mathbf{S}_1, \mathbf{S}_2)} M_{S_1 \cup S_2, G}.
    \end{align*}
\end{proposition}

\begin{proof}[Proof of Proposition~\ref{prop:total-intersecting}]
    The left hand side is the number of pairs of copies of $\mathbf{S}_1$ and $\mathbf{S}_2$. The right hand side counts the same number by enumerating over intersecting patterns of $\mathbf{S}_1$ and $\mathbf{S}_2$, and then counting the number of pairs isomorphic to a specific intersecting pattern.
\end{proof}

First, we state a lemma that expresses the number of copies of a star shape in any graph in terms of its degree sequence.
\begin{lemma} \label{lemma:star-shape-copy}
    Let $H$ be a graph, and $\mathbf{S} \cong  \mathbf{K}_{1,t}$ be a star shape. Then,
    \begin{align*}
        M_{\mathbf{S}, H} = \sum_{i \in V(H)} (d_i)_{(t)},
    \end{align*}
    where $d_i = \text{deg}_H(i)$.
\end{lemma} 

The following lemmas will be useful when dealing with sums of falling factorials that arise from Lemma \ref{lemma:star-shape-copy}.
\begin{lemma} \label{lemma: falling-factorial}
    Let $d_1, \dots, d_k$ be a sequence of natural numbers taking values at most $k$ for some function $k = k(n)$, and $t \in \mathbb{N}$ be a constant. If $\sum_{i\in [k]} d_i^t = \omega(k)$ then \[\sum_{i \in [k]} (d_i)_{(t)}=(1-o(1))\left(\sum_{i \in [k]} d_i^t\right).\]
\end{lemma} 

\begin{lemma} \label{lemma: sum-di-ub}
    Suppose $d_i \geq 0$ is a sequence of nonnegative real numbers. Let $\Delta = \max_{i} d_i$. Let $p \ge q$ be two nonnegative integers. Then
    \begin{align*}
        \sum_{i} d_i^p \le \left(\sum_{i} d_i^q\right) \Delta^{p-q}.
    \end{align*}
\end{lemma}

\subsection{Proof of Theorem \ref{thm:main}: Putting it all together} 

For the signed count polynomial $f_{\mathbf{S}}$ of some shape $\mathbf{S}$, the following proposition expresses its first, second moments under $\mathbb{Q}$ and the first moment under $\mathbb{P}$ in simple formulas, whose proof is deferred to Section \ref{subsection:auxillary-proof} in the appendix.

\begin{proposition}\label{prop:simple-moments}
    Let $\mathbf{S}$ be a shape. Then the following identities holds:
    \begin{align*}
        \mathbb{E}_{\mathbb{Q} }[f_{\mathbf{S} } ] &= 0,\\
        \mathbb{E}_{\mathbb{Q} }[f_{\mathbf{S} }^2 ] &= \frac{M_{\mathbf{S} }}{|\text{Aut}(\mathbf{S})|},\\
        \mathbb{E}_{\mathbb{P} }[f_{\mathbf{S} } ] &=  \frac{M_{\mathbf{S}, H }}{|\text{Aut}(\mathbf{S})|} \left(\frac{1-p}{p}\right)^{|\mathbf{S} |/2}.
    \end{align*}
\end{proposition}

Now we are ready to present the full proof of our main theorem.

\begin{proof}[Proof of Theorem \ref{thm:main}]
    If $\limsup_{n \to \infty} 
\text{Adv}^{\le D} < \infty$, by \cite[Lemma 7.3]{coja2022statistical} we conclude no degree-$D$ polynomial $f$ achieves strong separation between $\mathbb{P}$ and $\mathbb{Q}$.

    From now on we consider the situation $\lim_{n \to \infty} \text{Adv}^{\le D} = \infty$. Let $\mathbf{S}$ be the star shape that maximizes $\frac{M_{\mathbf{S}, H}^2}{M_{\mathbf{S}}} \left(\frac{1-p}{p}\right)^{|\mathbf{S}|}$ among star shapes with at most $D$ edges. By Proposition \ref{prop:star-shape},
    \begin{align}
        \frac{M_{\mathbf{S}, H}^2}{M_{\mathbf{S}}} \left(\frac{1-p}{p}\right)^{|\mathbf{S}|} \to \infty. \label{ineq:advantage-condition}
    \end{align}

    We now aim to show that $f_{\mathbf{S}}$, the signed count of $\mathbf{S}$, achieves strong separation between $\mathbb{P}$ and $\mathbb{Q}$. From Proposition \ref{prop:simple-moments}, we have
    \begin{align*}
        \mathbb{E}_{\mathbb{Q} }[f_{\mathbf{S} } ] &= 0,\\
        \mathbb{E}_{\mathbb{Q} }[f_{\mathbf{S} }^2 ] &= \frac{M_{\mathbf{S} }}{|\text{Aut}(\mathbf{S})|},\\
        \mathbb{E}_{\mathbb{P} }[f_{\mathbf{S} } ] &=  \frac{M_{\mathbf{S}, H }}{|\text{Aut}(\mathbf{S})|} \left(\frac{1-p}{p}\right)^{|\mathbf{S} |/2}.
    \end{align*}
    Note that $\Var_{\mathbb{Q}}[f_{\mathbf{S}}] = \mathbb{E}_{\mathbb{Q} }[f_{\mathbf{S} }^2 ] - \mathbb{E}_{\mathbb{Q} }[f_{\mathbf{S} } ]^2 = \frac{M_{\mathbf{S}}}{|\text{Aut}(\mathbf{S})|}$ and $|\mathbb{E}_{\mathbb{P} }[f_{\mathbf{S} } ] - \mathbb{E}_{\mathbb{Q} }[f_{\mathbf{S} } ]| = \frac{M_{\mathbf{S}, H }}{|\text{Aut}(\mathbf{S})|} \left(\frac{1-p}{p}\right)^{|\mathbf{S} |/2}$. The condition in \eqref{ineq:advantage-condition} implies one side of the strong separation
    \begin{align}
        \sqrt{\Var_{\mathbb{Q}}[f_{\mathbf{S}}]} = o
\left(|\mathbb{E}_{\mathbb{P}}[f_{\mathbf{S}}] - \mathbb{E}_{\mathbb{Q}}[f_{\mathbf{S}}]|\right), \label{ineq:strong-separation-1}
    \end{align}
    where we use that $|\text{Aut}(\mathbf{S})| = O(1)$ as $\mathbf{S}$ has a constant number of edges.

    It remains to show $\sqrt{\Var_{\mathbb{P}}[f_{\mathbf{S}}]} = o\left(|\mathbb{E}_{\mathbb{P}}[f_{\mathbf{S}}] - \mathbb{E}_{\mathbb{Q}}[f_{\mathbf{S}}]|\right)$. Since $\mathbb{E}_{\mathbb{P}}[f_{\mathbf{S}}] - \mathbb{E}_{\mathbb{Q}}[f_{\mathbf{S}}] = \mathbb{E}_{\mathbb{P}}[f_{\mathbf{S}}]$ and $\Var_{\mathbb{P}}[f_{\mathbf{S}}] = \mathbb{E}_{\mathbb{P}}[f_{\mathbf{S}}^2] - \mathbb{E}_{\mathbb{P}}[f_{\mathbf{S}}]^2$, to show $\Var_{\mathbb{P}}[f_{\mathbf{S}}] = o\left(|\mathbb{E}_{\mathbb{P}}[f_{\mathbf{S}}] - \mathbb{E}_{\mathbb{Q}}[f_{\mathbf{S}}]|^2\right)$, it is equivalent to proving
    \begin{align*}
        \frac{\mathbb{E}_{\mathbb{P}}[f_{\mathbf{S}}^2]}{\mathbb{E}_{\mathbb{P}}[f_{\mathbf{S}}^2]} \le 1 + o(1).
    \end{align*}

    Recall $f_{\mathbf{S}} = \sum_{S \subseteq \binom{V}{2}: S \cong \mathbf{S}} \chi_S$. We now examine this ratio
    \begin{align}
        \frac{\mathbb{E}_{\mathbb{P} }[f_{\mathbf{S} }^2]}{\mathbb{E}_{\mathbb{P} }[f_{\mathbf{S} }]^2} &= \frac{\sum_{S,S' \subseteq \binom{V}{2}: S, S' \cong \mathbf{S} } \mathbb{E}_{\mathbb{P} }[\chi_S \chi_{S'}] }{ \frac{M_{\mathbf{S}, H}^2}{|\text{Aut}(\mathbf{S})|^2} \left(\frac{1-p}{p}\right)^{|\mathbf{S} |}}. \label{eq:ratio-main}
    \end{align}

    Let us compute
    \begin{align}
        \mathbb{E}_{\mathbb{P} }[\chi_S \chi_{S'}] &= \mathbb{E}_{\mathbb{P} }\left[\prod_{\{i,j\} \in S } \frac{G_{i,j} - p}{\sqrt{p(1-p)}} \prod_{\{i',j'\} \in S' } \frac{G_{i',j'} - p}{\sqrt{p(1-p)}}\right]\\
        &= \mathbb{E}_{\mathbb{P} }\left[\prod_{\{i,j\} \in S \triangle S' } \frac{G_{i,j} - p}{\sqrt{p(1-p)}} \prod_{\{i',j'\} \in S \cap S' } \left(\frac{G_{i',j'} - p}{\sqrt{p(1-p)}}\right)^2\right]\\
        &=\mathbb{E}_{\mathbb{P} }\left[\chi_{S \triangle S'} \prod_{\{i,j\} \in S \cap S' } \left(\frac{G_{i,j} - p}{\sqrt{p(1-p)}}\right)^2 \right]\\
        &= \mathbb{E}_{\bold{H} } \mathbb{E}_{\mathbb{P} } \left[\chi_{S \triangle S'} \prod_{\{i,j\} \in S \cap S' } \left(\frac{G_{i,j} - p}{\sqrt{p(1-p)}}\right)^2 \Bigg| \bold{H} \right]. \label{eq:fourier-expectation}
    \end{align}
    Observe that the conditional expectation above evaluates to $0$ whenever $S \triangle S'$ is not fully contained inside $\bold{H}$. For a fixed embedding $\bold{H}$ of $H$, if $S \triangle S' \subseteq \bold{H}$, then
    \begin{align}
        &\quad \mathbb{E}_{\mathbb{P} } \left[\chi_{S \triangle S'} \prod_{\{i,j\} \in S \cap S' } \left(\frac{G_{i,j} - p}{\sqrt{p(1-p)}}\right)^2 \Bigg| \bold{H} \right]\\
        &= \left(\frac{1-p}{p}\right)^{|S \triangle S'|/2} \mathbb{E}_{\mathbb{P} } \left[ \prod_{\{i,j\} \in S \cap S' } \left(\frac{G_{i,j} - p}{\sqrt{p(1-p)}}\right)^2 \Bigg| \bold{H} \right]\\
        &= \left(\frac{1-p}{p}\right)^{|S \triangle S'|/2} \left(\frac{1-p}{p}\right)^{|(S \cap S') \cap E(\bold{H})|}, \label{eq:conditonal-expectation}
    \end{align}
    where the last equality follows from $\mathbb{E}\left[\frac{(G_{i,j} - p)^2}{p(1-p)} \Bigg| \bold{H}\right] = \begin{cases}
        \frac{1-p}{p} & \quad \text{ if } \{i,j\} \in E(\bold{H}),\\
        1 & \quad \text{ otherwise}.
    \end{cases}$ Moreover, since $D = O(1)$ and $p = \Omega(1)$, we have
    \begin{align}
        \left(\frac{1-p}{p}\right)^{|(S \cap S') \cap E(\bold{H})|} &\le \left(\frac{1-p}{p}\right)^{O(1)} \le O(1). \label{ineq:(1-p)/p-ub}
    \end{align}
    Inserting \eqref{eq:conditonal-expectation} and \eqref{ineq:(1-p)/p-ub} back to \eqref{eq:fourier-expectation}, we get
    \begin{align}
        \mathbb{E}_{\mathbb{P} }[\chi_S \chi_{S'}] &= \mathbb{E}_{\bold{H} } \mathbb{E}_{\mathbb{P} } \left[\chi_{S \triangle S'} \prod_{\{i,j\} \in E(S \cap S') } \left(\frac{G_{i,j} - p}{\sqrt{p(1-p)}}\right)^2 \Bigg| \bold{H} \right]\\
        &\lesssim \mathbb{E}_{\bold{H} } \left[\boldsymbol{1}\{S \triangle S' \subseteq \bold{H}\} \left(\frac{1-p}{p}\right)^{|S \triangle S'|/2}\right]\\
        &= \left(\frac{1-p}{p}\right)^{|S \triangle S'|/2} \mathbb{P}(S \triangle S' \subseteq \bold{H} )\\
        &=  \left(\frac{1-p}{p}\right)^{|S \triangle S'|/2} \frac{M_{S\triangle S', H}}{M_{S\triangle S'}}, \label{ineq:fourier-expecation-bound}
    \end{align}
    where in the last line we use Lemma \ref{lemma:inclusion-prob}.

    Substituting the bound \eqref{ineq:fourier-expecation-bound} back to our ratio \eqref{eq:ratio-main}, we get
    \begin{align}
        \frac{\mathbb{E}_{\mathbb{P} }[f_{\mathbf{S} }^2]}{\mathbb{E}_{\mathbb{P} }[f_{\mathbf{S} }]^2} &\lesssim \frac{\sum_{S,S' \subseteq \binom{V}{2}: S, S' \cong \mathbf{S} }  \frac{M_{S\triangle S', H}}{M_{S\triangle S'}} \left(\frac{1-p}{p}\right)^{|S \triangle S'|/2} }{ \frac{M_{\mathbf{S}, H}^2}{|\text{Aut}(\mathbf{S})|^2} \left(\frac{1-p}{p}\right)^{|\mathbf{S} |}}\\
        &= \frac{1}{M_{\mathbf{S}, H}^2 \left(\frac{1-p}{p}\right)^{|\mathbf{S} |}} \cdot \sum_{\substack{((S, \gamma), (S', \gamma')):\\ S, S' \cong \mathbf{S}}} \frac{M_{S \triangle S', H}}{M_{S \triangle S'}} \left(\frac{1-p}{p}\right)^{|S \triangle S'|/2}\\
        &= \frac{1}{M_{\mathbf{S}, H}^2 \left(\frac{1-p}{p}\right)^{|\mathbf{S} |}} \cdot \sum_{\substack{i(S_1, S_2)\\ \in \text{Inter}(\mathbf{S}, \mathbf{S})}} \sum_{\substack{(S, S'):\\ (S, S') \\\cong i(S_1, S_2)}} \frac{M_{S \triangle S', H}}{M_{S \triangle S'}} \left(\frac{1-p}{p}\right)^{|S \triangle S'|/2}
    \end{align}
    \begin{align}
        &= \frac{1}{M_{\mathbf{S}, H}^2 \left(\frac{1-p}{p}\right)^{|\mathbf{S} |}} \cdot \Bigg(\sum_{\substack{i(S_1, S_2)\\ \in \text{Inter}(\mathbf{S}, \mathbf{S}):\\
        S_1 \cap S_2 =\emptyset } } \sum_{\substack{(S, S'):\\ (S, S') \\\cong i(S_1, S_2)}} \frac{M_{S_1 \triangle S_2, H}}{M_{S_1 \triangle S_2}} \left(\frac{1-p}{p}\right)^{|S_1 \triangle S_2|/2}\\
        &+ \sum_{\substack{i(S_1, S_2)\\ \in \text{Inter}(\mathbf{S}, \mathbf{S}):\\ S_1 \cap S_2 \ne \emptyset } } \sum_{\substack{(S, S'):\\ (S, S') \\\cong i(S_1, S_2)}} \frac{M_{S_1 \triangle S_2, H}}{M_{S_1 \triangle S_2}} \left(\frac{1-p}{p}\right)^{|S_1 \triangle S_2|/2}\Bigg),
    \end{align}
    where in the second line the summation is over pairs of copies of $\mathbf{S}$, rather than sets (unlabelled copies) isomorphic to $\mathbf{S}$, which cancels out the $|\text{Aut}(\mathbf{S})|^2$ from the first line.

    Let us examine the first term, corresponding to pairs with empty intersection:
    \begin{align*}
        \sum_{\substack{i(S_1, S_2)\\ \in \text{Inter}(\mathbf{S}, \mathbf{S}):\\
        S_1 \cap S_2 =\emptyset } } \sum_{\substack{(S, S'):\\ (S, S') \\\cong i(S_1, S_2)}} \frac{M_{S_1 \triangle S_2, H}}{M_{S_1 \triangle S_2}} &= \sum_{\substack{i(S_1, S_2)\\ \in \text{Inter}(\mathbf{S}, \mathbf{S}):\\
        S_1 \cap S_2 =\emptyset } } \sum_{\substack{(S, S'):\\ (S, S') \\\cong i(S_1, S_2)}} \frac{M_{S_1 \cup S_2, H}}{M_{S_1 \cup S_2}}\\
        &\le \sum_{\substack{i(S_1, S_2)\\ \in \text{Inter}(\mathbf{S}, \mathbf{S})}} \sum_{\substack{(S, S'):\\ (S, S') \\\cong i(S_1, S_2)}} \frac{M_{S_1 \cup S_2, H}}{M_{S_1 \cup S_2}}\\
        &= \sum_{\substack{i(S_1, S_2)\\ \in \text{Inter}(\mathbf{S}, \mathbf{S})} } M_{S_1 \cup S_2} \frac{M_{S_1 \cup S_2, H}}{M_{S_1 \cup S_2}}\\
        &= \sum_{\substack{i(S_1, S_2)\\ \in \text{Inter}(\mathbf{S}, \mathbf{S})} } M_{S_1 \cup S_2, H}\\
        &= M_{S_1, H} M_{S_2, H}\\
        &= M_{\mathbf{S}, H }^2,
    \end{align*}
    where in the second-to-last line we use Proposition \ref{prop:total-intersecting}. As a result, the first term can be bounded by
    \begin{align*}
        &\quad \frac{1}{M_{\mathbf{S}, H}^2 \left(\frac{1-p}{p}\right)^{|\mathbf{S} |}} \cdot \sum_{\substack{i(S_1, S_2)\\ \in \text{Inter}(\mathbf{S}, \mathbf{S}):\\
        S_1 \cap S_2 =\emptyset } } \sum_{\substack{(S, S'):\\ (S, S') \\\cong i(S_1, S_2)}} \frac{M_{S_1 \triangle S_2, H}}{M_{S_1 \triangle S_2}} \left(\frac{1-p}{p}\right)^{|S_1 \triangle S_2|/2} \\
        &\le \frac{1}{M_{\mathbf{S}, H}^2 \left(\frac{1-p}{p}\right)^{|\mathbf{S} |}} \cdot M^2_{\mathbf{S}, H} \left(\frac{1-p}{p}\right)^{|\mathbf{S}|} = 1
    \end{align*}

    Thus, we obtain a bound
    \begin{align}
        \frac{\mathbb{E}_{\mathbb{P} }[f_{\mathbf{S} }^2]}{\mathbb{E}_{\mathbb{P} }[f_{\mathbf{S} }]^2} &\le 1 + \frac{1}{M_{\mathbf{S}, H}^2 \left(\frac{1-p}{p}\right)^{|\mathbf{S}|}} \cdot \sum_{\substack{i(S_1, S_2)\\ \in \text{Inter}(\mathbf{S}, \mathbf{S}):\\ S_1 \cap S_2 \ne \emptyset } } \sum_{\substack{(S, S'):\\ (S, S') \\\cong i(S_1, S_2)}} \frac{M_{S_1 \triangle S_2, H}}{M_{S_1 \triangle S_2}} \left(\frac{1-p}{p}\right)^{|S_1 \triangle S_2|/2}. \label{ineq:ratio-main-bound}
    \end{align}

    Next let us examine the second term, corresponding to pairs with nonempty intersection:
    \begin{align*}
        &\quad \sum_{\substack{i(S_1, S_2)\\ \in \text{Inter}(\mathbf{S}, \mathbf{S}):\\ S_1 \cap S_2 \ne \emptyset } } \sum_{\substack{(S, S'):\\ (S, S') \\\cong i(S_1, S_2)}} \frac{M_{S_1 \triangle S_2, H}}{M_{S_1 \triangle S_2}} \left(\frac{1-p}{p}\right)^{|S_1 \triangle S_2|/2}
    \end{align*}
    \begin{align*}
        &= \sum_{\substack{i(S_1, S_2)\\ \in \text{Inter}(\mathbf{S}, \mathbf{S}):\\ S_1 \cap S_2 \ne \emptyset } } M_{S_1 \cup S_2} \cdot \frac{M_{S_1 \triangle S_2, H}}{M_{S_1 \triangle S_2}} \left(\frac{1-p}{p}\right)^{|S_1 \triangle S_2|/2}\\
        &\le \sum_{\substack{i(S_1, S_2)\\ \in \text{Inter}(\mathbf{S}, \mathbf{S}):\\ S_1 \cap S_2 \ne \emptyset } } n^{|V(S_1 \cup S_2)| - |V(S_1 \triangle S_2)|} M_{S_1 \triangle S_2, H} \left(\frac{1-p}{p}\right)^{|S_1 \triangle S_2|/2},
    \end{align*}
    which is a sum over a constant number (depending on the constant $D$) of terms. Moreover, by Proposition \ref{prop:intersection-ratio-bound}, each term is $o\left(M_{\mathbf{S} ,H}^2 \left(\frac{1-p}{p}\right)^{|\mathbf{S}|}\right)$, and thus the second term is bounded by
    \begin{align*}
        &\quad \frac{1}{M_{\mathbf{S}, H}^2 \left(\frac{1-p}{p}\right)^{|\mathbf{S}|}} \cdot \sum_{\substack{i(S_1, S_2)\\ \in \text{Inter}(\mathbf{S}, \mathbf{S}):\\ S_1 \cap S_2 \ne \emptyset } } \sum_{\substack{(S, S'):\\ (S, S') \\\cong i(S_1, S_2)}} \frac{M_{S_1 \triangle S_2, H}}{M_{S_1 \triangle S_2}} \left(\frac{1-p}{p}\right)^{|S_1 \triangle S_2|/2}\\
        &\le \frac{1}{M_{\mathbf{S}, H}^2 \left(\frac{1-p}{p}\right)^{|\mathbf{S}|}} \cdot \sum_{\substack{i(S_1, S_2)\\ \in \text{Inter}(\mathbf{S}, \mathbf{S}):\\ S_1 \cap S_2 \ne \emptyset } } n^{|V(S_1 \cup S_2)| - |V(S_1 \triangle S_2)|} M_{S_1 \triangle S_2, H} \left(\frac{1-p}{p}\right)^{|S_1 \triangle S_2|/2}\\
        &= o(1).
    \end{align*}
    Plugging it back to \eqref{ineq:ratio-main-bound}, we get the desired bound
    \begin{align*}
        \frac{\mathbb{E}_{\mathbb{P} }[f_{\mathbf{S} }^2]}{\mathbb{E}_{\mathbb{P} }[f_{\mathbf{S} }]^2} &\le 1 + o(1),
    \end{align*}
    which, as discussed at the beginning of the proof, implies the other side of the strong separation \begin{align}
        \sqrt{\Var_{\mathbb{P}}[f_{\mathbf{S}}]} = o\left(|\mathbb{E}_{\mathbb{P}}[f_{\mathbf{S}}] - \mathbb{E}_{\mathbb{Q}}[f_{\mathbf{S}}]|\right). \label{ineq:strong-separation-2}
    \end{align}

    Since both conditions \eqref{ineq:strong-separation-1} and \eqref{ineq:strong-separation-2} hold, we conclude that $f_{\mathbf{S}}$, the signed count of the star shape $\mathbf{S}$, achieves strong separation between $\mathbb{P}$ and $\mathbb{Q}$.
\end{proof}

\subsection{Proof of Key Lemmas}

We will need the following two claims about some combinatorial properties of the quantities $M_{\mathbf{S}, H}$. This will be directly useful for proving the intuition that focusing on stars is all we need, leading to Proposition \ref{prop:star-shape}.
\begin{claim}\label{claim:spanning}
    Let $\mathbf{S}$ be a shape, and $\mathbf{S}'$ be a spanning sub-shape of $\mathbf{S}$. Then, $M_{\mathbf{S}', H }  \ge  M_{\mathbf{S}, H }$ for any graph $H$.
\end{claim}

\begin{claim}\label{claim:disjoint-union}
    Let $\mathbf{S}$ be a shape that is the disjoint union of two shapes $\mathbf{S}_1$ and $\mathbf{S}_2$ that are vertex-disjoint inside $\mathbf{S}$. Then, $ M_{\mathbf{S}_1, H }M_{\mathbf{S}_2, H } \ge M_{\mathbf{S}, H }$ for any graph $H$.

    If moreover $\mathbf{S}$ has a constant number of edges, then $M_{\mathbf{S}_1}M_{\mathbf{S}_2} \le (1 + o(1))M_{\mathbf{S}}$ as $n \to \infty$.
\end{claim}

\begin{proof}[Proof of \ref{claim:spanning}]
    Let $S \subseteq H$ together with a labelling $\gamma: V(S) \to V(\mathbf{S})$ be a copy of $\mathbf{S}$ in $H$. Since $\mathbf{S}'$ is a spanning sub-shape of $\mathbf{S}$, there is a spanning subgraph $S' \subseteq S$ that is isomorphic to $\mathbf{S}'$ and moreover inherits $\gamma$ as the isomorphism mapping. It is not hard to see from the argument above for every copy of $\mathbf{S}$, we find a distinct copy of $\mathbf{S}'$, and $M_{\mathbf{S}', H } \ge M_{\mathbf{S}, H}$.
\end{proof}

\begin{proof}[Proof of \ref{claim:disjoint-union}]
    Let $S \subseteq H$ with a labelling $\gamma: V(S) \to V(\mathbf{S})$ be a copy of $\mathbf{S}$ inside $H$. As $\mathbf{S}$ is the vertex-disjoint union of two shapes $\mathbf{S}_1$ and $\mathbf{S}_2$, $S$ consists of two vertex-disjoint subgraphs $\gamma^{-1}(\mathbf{S}_1)$ and $\gamma^{-1}(\mathbf{S}_2)$, and $\gamma$ induces two labellings $\gamma_1: \gamma^{-1}(\mathbf{S}_1) \to \mathbf{S}_1$ and $\gamma_2: \gamma^{-1}(\mathbf{S}_2) \to \mathbf{S}_2$. Thus, $(\gamma^{-1}(\mathbf{S}_i), \gamma_i)$ is a copy of $\mathbf{S}_i$ in $H$ for $i \in \{1,2\}$. It is not hard to see that for every copy of $\mathbf{S}$, we find a distinct pair of copies of $\mathbf{S}_1$ and $\mathbf{S}_2$, and $ M_{\mathbf{S}_1, H }M_{\mathbf{S}_2, H } \ge M_{\mathbf{S}, H }$.
    
    On the other hand, if $\mathbf{S}$ has only a constant number of edges, then in the complete graph $K_n$,
    \begin{align*}
        M_{\mathbf{S}_1} M_{\mathbf{S}_2} = (n)_{(|V(\mathbf{S}_1)|)} (n)_{(|V(\mathbf{S}_2)|)} \le n^{|V(\mathbf{S}_1)|+|V(\mathbf{S}_2)|} \le (1 + o(1)) (n)_{|V(\mathbf{S})|} = (1 + o(1)) M_{\mathbf{S}}.
    \end{align*}
\end{proof}

Immediately, the two claims above yield the following two simple corollaries, which will be used to prove Proposition \ref{prop:star-shape}.
\begin{corollary}\label{cor:spanning}

    Let $\mathbf{S}$ be a shape, and $\mathbf{S}'$ be a spanning sub-shape of $\mathbf{S}$. Then, $\frac{M_{\mathbf{S}', H }^2}{M_{\mathbf{S}'} }  \ge  \frac{M_{\mathbf{S}, H }^2}{M_{\mathbf{S}} } $.
\end{corollary}

\begin{corollary}\label{cor:disjoint-union}
    Let $\mathbf{S}$ be a shape with a constant number of edges that is the disjoint union of two shapes $\mathbf{S}_1$ and $\mathbf{S}_2$ that are vertex-disjoint inside $\mathbf{S}$. Then, $ \left(\frac{M_{
\mathbf{S}_1, H }^2}{M_{\mathbf{S}_1} } \left(\frac{1-p}{p}\right)^{|\mathbf{S}_1|}\right) \cdot \left(\frac{M_{\mathbf{S}_2, H }^2}{M_{\mathbf{S}_2} } \left(\frac{1-p}{p}\right)^{|\mathbf{S}_2|}\right) \ge (1 - o(1)) \frac{M_{\mathbf{S}, H }^2}{M_{\mathbf{S}} } \left(\frac{1-p}{p}\right)^{|\mathbf{S}|}$.
\end{corollary}

\begin{proof}[Proof of Proposition~\ref{prop:star-shape}]
    Since $\left(\text{Adv}^{\le D} \right)^2 = \sum_{\mathbf{S}: |\mathbf{S}| \le D } \frac{M_{\mathbf{S} , H}^2}{M_{\mathbf{S}} \cdot |\text{Aut}(\mathbf{S})|} \left(\frac{1-p}{p}\right)^{|\mathbf{S} |} \to \infty $ by Proposition \ref{prop:Adv}, and there are at most a constant number of shapes depending on the constant $D$, each with an automorphism group of constant size, we have
    \begin{align*}
        \max_{\mathbf{S}: |\mathbf{S}| \le D } \frac{M_{\mathbf{S} , H}^2}{M_{\mathbf{S}}} \left(\frac{1-p}{p}\right)^{|\mathbf{S} |} \to \infty
    \end{align*}
    
    Now let $\mathbf{S}$ be a shape with at most $D$ edges that maximizes $\frac{M_{\mathbf{S}, H }^2}{M_{\mathbf{S} }} \left(\frac{1-p}{p}\right)^{|\mathbf{S}|}$. With Corollary \ref{cor:spanning} and Corollary \ref{cor:disjoint-union} in hand, we will show a properly chosen sub-shape of $
\mathbf{S}$ satisfies the condition of the corollary. 
    
    If $\mathbf{S}$ is already star graph, then we are done. If $\mathbf{S}$ is not a connected shape, then we may recurse on one of the connected components $\mathbf{S}'$ of $\mathbf{S}$ while ensuring $\frac{M_{\mathbf{S}', H }^2}{M_{\mathbf{S}' }} \left(\frac{1-p}{p}\right)^{|\mathbf{S}'|} \to \infty$ using Corollary \ref{cor:disjoint-union}. So now let us assume $\mathbf{S}$ is connected. Let $\mathbf{T}$ be a spanning tree of the shape $\mathbf{S}$. By Corollary \ref{cor:spanning}, 
    \begin{align*}
        \frac{M_{\mathbf{T}, H }^2}{M_{\mathbf{T} }} \left(\frac{1-p}{p}\right)^{|\mathbf{T}|} &\ge  \frac{M_{\mathbf{S}, H }^2}{M_{\mathbf{S} }} \left(\frac{1-p}{p}\right)^{|\mathbf{T}|}\\
        &\gtrsim \frac{M_{\mathbf{S}, H }^2}{M_{\mathbf{S} }} \left(\frac{1-p}{p}\right)^{|\mathbf{S}|} \to \infty,
    \end{align*}
    where in the last inequality we use that $\left(\frac{1-p}{p}\right)^{O(1)} = O(1)$ as $p = \Omega(1)$. If the diameter of $\mathbf{T}$ is at least $3$, let us consider a path $a-b-c-d$ of length $3$ in $\mathbf{T}$. Note that after the edge $\{b, c\}$ is deleted from $\mathbf{T}$, what remains is still a spanning sub-shape $\mathbf{T} - \{b,c\}$, and we may recurse on it by \ref{claim:spanning}.

    Repeating the process above, we will end up with a shape $\mathbf{S}$ that satisfies $\frac{M_{\mathbf{S}, H }^2}{M_{\mathbf{S} }} \left(\frac{1-p}{p}\right)^{|\mathbf{S}|} \to \infty$, and moreover is either a star graph or a tree of diameter at most $2$. Note that star graphs are precisely trees with diameters at most $2$. This finishes the proof that \[\max_{\substack{\mathbf{S} \cong \mathbf{K}_{1,t}:\\ t \le D} } \frac{M_{\mathbf{S}, H }^2}{M_{\mathbf{S} }} \left(\frac{1-p}{p}\right)^{|\mathbf{S} |} \to \infty.\]
\end{proof}

\begin{proof}[Proof of Proposition \ref{prop:intersection-ratio-bound}]
    Let $d_i = \deg_{H}(i)$ for $i \in V(H)$. By Lemma \ref{lemma:star-shape-copy},
    \begin{align*}
        M_{\mathbf{S},H} = \sum_{i\in V(H)}(d_{i})_{(t)}.
    \end{align*}

    First, let us notice that if $S_1 \triangle S_2 = \emptyset$, the desired bound easily follows from the assumptions, as in this situation the ratio of interest is bounded by
    \begin{align*}
        \frac{n^{|V(S_1 \cup S_2)| } M_{\emptyset, H} }{M_{\mathbf{S} ,H}^2 \left(\frac{1-p}{p}\right)^{|\mathbf{S}|}} \le (1 + o(1)) \frac{M_{\mathbf{S} }}{M_{\mathbf{S}, H}^2 \left(\frac{1-p}{p}\right)^{|\mathbf{S}|} } = o(1).
    \end{align*}
    So from now on let us assume $S_1 \triangle S_2 \ne \emptyset$. If $S_1, S_2 \cong \mathbf{K}_{1,t}$ have non-empty intersection and $S_1 \triangle S_2 \ne \emptyset$, there are two cases.
        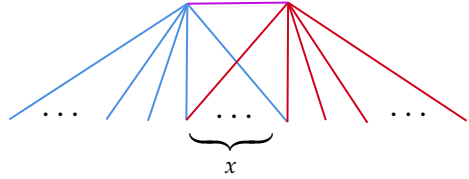
\begin{figure}[h]
        \centering

\tikzset{every picture/.style={line width=0.75pt}} 

\begin{tikzpicture}[x=0.75pt,y=0.75pt,yscale=-1,xscale=1]

\draw [color={rgb, 255:red, 74; green, 144; blue, 226 }  ,draw opacity=1 ]   (180.5,50.75) -- (90.67,109.17) ;
\draw [color={rgb, 255:red, 74; green, 144; blue, 226 }  ,draw opacity=1 ]   (180.5,50.75) -- (180,109.5) ;
\draw [color={rgb, 255:red, 74; green, 144; blue, 226 }  ,draw opacity=1 ]   (180.5,50.75) -- (231,110.5) ;
\draw [color={rgb, 255:red, 208; green, 2; blue, 27 }  ,draw opacity=1 ]   (180,109.5) -- (231.33,50.17) ;
\draw [color={rgb, 255:red, 208; green, 2; blue, 27 }  ,draw opacity=1 ]   (231.33,50.17) -- (231,110.5) ;
\draw [color={rgb, 255:red, 208; green, 2; blue, 27 }  ,draw opacity=1 ]   (231.33,50.17) -- (271.33,110.83) ;
\draw [color={rgb, 255:red, 208; green, 2; blue, 27 }  ,draw opacity=1 ]   (231.33,50.17) -- (321.33,109.83) ;
\draw [color={rgb, 255:red, 189; green, 16; blue, 224 }  ,draw opacity=1 ]   (180.5,50.75) -- (231.33,50.17) ;
\draw [color={rgb, 255:red, 74; green, 144; blue, 226 }  ,draw opacity=1 ]   (180.5,50.75) -- (139.67,109.17) ;
\draw [color={rgb, 255:red, 74; green, 144; blue, 226 }  ,draw opacity=1 ]   (180.5,50.75) -- (160.67,109.83) ;
\draw [color={rgb, 255:red, 208; green, 2; blue, 27 }  ,draw opacity=1 ]   (231.33,50.17) -- (250.33,109.5) ;

\draw (105.37,104.25) node [anchor=north west][inner sep=0.75pt]    {$\dotsc $};
\draw (192.92,105.1) node [anchor=north west][inner sep=0.75pt]    {$\dotsc $};
\draw (280.75,103.95) node [anchor=north west][inner sep=0.75pt]    {$\dotsc $};
\draw (179,113.47) node [anchor=north west][inner sep=0.75pt]    {$\underbrace{\ \ \ \ \ \ \ \ \ \ \ }_{x}$};

\end{tikzpicture}
        \caption{Case $1$ for the intersecting pattern $i(S_1, S_2)$ of two $t$-stars with non-empty intersection. The edges contained solely in $S_1$, solely in $S_2$, and in the intersection of $S_1$ and $S_2$, are marked by blue, red, and purple respectively.}
        \label{fig:intersecting-shape-distinct-root}
    \end{figure}
    
    Case 1: Under this case, $S_1$ and $S_2$ do not share the same root. Note that $|V(S_1\cup S_2)| = |V(S_1 \triangle S_2)|$ and $|S_1 \triangle S_2| = 2|\mathbf{S}| - 2$ in this case. Consider the shapes $S_1' = S_1 \setminus S_2$, and $S_2' = S_2[V(S_2) \setminus V(S_1')]$, as illustrated in Figure \ref{fig:vertex-partitioning-of-fig1}.

    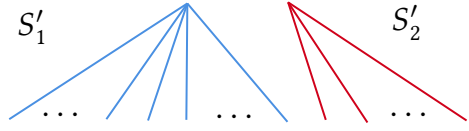
\begin{figure}[h]
        \centering

\tikzset{every picture/.style={line width=0.75pt}} 

\begin{tikzpicture}[x=0.75pt,y=0.75pt,yscale=-1,xscale=1]

\draw [color={rgb, 255:red, 74; green, 144; blue, 226 }  ,draw opacity=1 ]   (180.5,50.75) -- (90.67,109.17) ;
\draw [color={rgb, 255:red, 74; green, 144; blue, 226 }  ,draw opacity=1 ]   (180.5,50.75) -- (180,109.5) ;
\draw [color={rgb, 255:red, 74; green, 144; blue, 226 }  ,draw opacity=1 ]   (180.5,50.75) -- (231,110.5) ;
\draw [color={rgb, 255:red, 208; green, 2; blue, 27 }  ,draw opacity=1 ]   (231.33,50.17) -- (271.33,110.83) ;
\draw [color={rgb, 255:red, 208; green, 2; blue, 27 }  ,draw opacity=1 ]   (231.33,50.17) -- (321.33,109.83) ;
\draw [color={rgb, 255:red, 74; green, 144; blue, 226 }  ,draw opacity=1 ]   (180.5,50.75) -- (139.67,109.17) ;
\draw [color={rgb, 255:red, 74; green, 144; blue, 226 }  ,draw opacity=1 ]   (180.5,50.75) -- (160.67,109.83) ;
\draw [color={rgb, 255:red, 208; green, 2; blue, 27 }  ,draw opacity=1 ]   (231.33,50.17) -- (250.33,109.5) ;

\draw (105.37,104.25) node [anchor=north west][inner sep=0.75pt]    {$\dotsc $};
\draw (192.92,105.1) node [anchor=north west][inner sep=0.75pt]    {$\dotsc $};
\draw (280.75,103.95) node [anchor=north west][inner sep=0.75pt]    {$\dotsc $};
\draw (93.2,52.4) node [anchor=north west][inner sep=0.75pt]    {$S_{1} '$};
\draw (282,50) node [anchor=north west][inner sep=0.75pt]    {$S_{2} '$};

\end{tikzpicture}
        \caption{Vertex disjoint $S_1'$ and $S_2'$ whose union is a spanning sub-shape of $S_1 \triangle S_2$ in Figure \ref{fig:intersecting-shape-distinct-root}}
        \label{fig:vertex-partitioning-of-fig1}
    \end{figure}

    Notice that $S_1'$ and $S_2'$ are vertex disjoint, and $S_1' \sqcup S_2'$ is a spanning sub-shape of $S_1 \triangle S_2$. By Claim \ref{claim:spanning} and Claim \ref{claim:disjoint-union}, we have
    \begin{align*}
        M_{S_1 \triangle S_2, H} \le M_{S_1' \sqcup S_2', H} \le M_{S_1', H} M_{S_2', H}.
    \end{align*}
    
    Then, using $|V(S_1\cup S_2)| = |V(S_1 \triangle S_2)|$, $|S_1 \triangle S_2| = 2|\mathbf{S}| - 2$, and $M_{S_1 \triangle S_2, H} \le M_{S_1', H} M_{S_2', H}$, 
    \begin{align}
        &\quad \frac{n^{|V(S_1 \cup S_2)| - |V(S_1 \triangle S_2)|} M_{S_1 \triangle S_2, H} \left(\frac{1-p}{p}\right)^{|S_1 \triangle S_2|/2} }{M_{\mathbf{S} ,H}^2 \left(\frac{1-p}{p}\right)^{|\mathbf{S}| }}\\
        &\le \frac{ M_{S_1', H}M_{S_2', H}  }{M_{\mathbf{S} ,H}^2 \left(\frac{1-p}{p}\right)}\\
        &\le (1+o(1))\frac{\sqrt{\frac{M_{S_1', H}^2}{M_{S_1'}} \left(\frac{1-p}{p}\right)^{|S_1'|} \frac{M_{S_2', H}^2}{M_{S_2'}} \left(\frac{1-p}{p}\right)^{|S_2'|} }}{\frac{M_{\mathbf{S}, H}^2}{M_{\mathbf{S}}} \left(\frac{1-p}{p}\right)^{|\mathbf{S}|}} \cdot \frac{\left(\frac{1-p}{p}\right)^{|\mathbf{S}| - (|S_1'| + |S_2'|)/2 - 1}}{n^{|V(\mathbf{S})| - (|V(S_1')| + |V(S_2')|)/2}}, \label{ineq:ratio-bound-case-1}
    \end{align}
    where the last line uses $M_{S} = (1 - o(1))n^{|V(S)|}$ for constant sized shapes $S$.

    Since $\mathbf{S}$ is the star shape that maximizes $\frac{M_{\mathbf{S}, H}^2}{M_{\mathbf{S}}}\left(\frac{1-p}{p}\right)^{|\mathbf{S}|}$ among star shapes with at most $D$ edges, and $S_1', S_2'$ are also star shapes with at most $|\mathbf{S}|$ number of edges, we have
    \begin{align*}
        \frac{\sqrt{\frac{M_{S_1', H}^2}{M_{S_1'}} \left(\frac{1-p}{p}\right)^{|S_1'|} \frac{M_{S_2', H}^2}{M_{S_2'}} \left(\frac{1-p}{p}\right)^{|S_2'|} }}{\frac{M_{\mathbf{S}, H}^2}{M_{\mathbf{S}}} \left(\frac{1-p}{p}\right)^{|\mathbf{S}|}} \le 1.
    \end{align*}
    Also observe that $|\mathbf{S}| - (|S_1'| + |S_2'|)/2 - 1 \ge 0$ and $|V(\mathbf{S})| - (|V(S_1')| + |V(S_2')|)/2 \ge 1$, we have
    \begin{align*}
        \frac{\left(\frac{1-p}{p}\right)^{|\mathbf{S}| - (|S_1'| + |S_2'|)/2 - 1}}{n^{|V(\mathbf{S})| - (|V(S_1')| + |V(S_2')|)/2}} \lesssim \frac{1}{n},
    \end{align*}
    where we use the fact that $\frac{1-p}{p} = O(1)$ as $p = \Omega(1)$, and $|\mathbf{S}| \le D = O(1)$. Plugging these bounds back to \eqref{ineq:ratio-bound-case-1}, we get
    \begin{align*}
        &\quad \frac{n^{|V(S_1 \cup S_2)| - |V(S_1 \triangle S_2)|} M_{S_1 \triangle S_2, H} \left(\frac{1-p}{p}\right)^{|S_1 \triangle S_2|/2} }{M_{\mathbf{S} ,H}^2 \left(\frac{1-p}{p}\right)^{|\mathbf{S}| }}
    \end{align*}
    \begin{align*}
        &\le (1+o(1))\frac{\sqrt{\frac{M_{S_1', H}^2}{M_{S_1'}} \left(\frac{1-p}{p}\right)^{|S_1'|} \frac{M_{S_2', H}^2}{M_{S_2'}} \left(\frac{1-p}{p}\right)^{|S_2'|} }}{\frac{M_{\mathbf{S}, H}^2}{M_{\mathbf{S}}} \left(\frac{1-p}{p}\right)^{|\mathbf{S}|}} \cdot \frac{\left(\frac{1-p}{p}\right)^{|\mathbf{S}| - (|S_1'| + |S_2'|)/2 - 1}}{n^{|V(\mathbf{S})| - (|V(S_1')| + |V(S_2')|)/2}}\\
        &\lesssim \frac{1}{n}\\
        &= o(1).
    \end{align*}

        \begin{figure}[h]
        \centering
\tikzset{every picture/.style={line width=0.75pt}} 

\begin{tikzpicture}[x=0.75pt,y=0.75pt,yscale=-1,xscale=1]

\draw [color={rgb, 255:red, 74; green, 144; blue, 226 }  ,draw opacity=1 ]   (170.67,40.17) -- (50.33,119.17) ;
\draw [color={rgb, 255:red, 74; green, 144; blue, 226 }  ,draw opacity=1 ]   (170.67,40.17) -- (120.6,119) ;
\draw [color={rgb, 255:red, 144; green, 19; blue, 254 }  ,draw opacity=1 ]   (170.67,40.17) -- (150.6,119.8) ;
\draw [color={rgb, 255:red, 144; green, 19; blue, 254 }  ,draw opacity=1 ]   (170.67,40.17) -- (190.6,119.8) ;
\draw [color={rgb, 255:red, 74; green, 144; blue, 226 }  ,draw opacity=1 ]   (170.67,40.17) -- (90.2,119) ;
\draw [color={rgb, 255:red, 208; green, 2; blue, 27 }  ,draw opacity=1 ]   (170.67,40.17) -- (220.6,119.8) ;
\draw [color={rgb, 255:red, 208; green, 2; blue, 27 }  ,draw opacity=1 ]   (170.67,40.17) -- (249.4,119) ;
\draw [color={rgb, 255:red, 208; green, 2; blue, 27 }  ,draw opacity=1 ]   (170.67,40.17) -- (290.6,120.6) ;

\draw (61.73,113.2) node [anchor=north west][inner sep=0.75pt]    {$\dotsc $};
\draw (157.13,113.67) node [anchor=north west][inner sep=0.75pt]    {$\dotsc $};
\draw (255.4,113.07) node [anchor=north west][inner sep=0.75pt]    {$\dotsc $};
\draw (149,124.4) node [anchor=north west][inner sep=0.75pt]    {$\underbrace{\ \ \ \ \ \ \ \ \ }_{x}$};

\end{tikzpicture}
        \caption{Case $2$ for the intersecting pattern $i(S_1, S_2)$ of two $t$-stars with non-empty intersection. The edges contained solely in $S_1$, solely in $S_2$, and in the intersection of $S_1$ and $S_2$, are marked by blue, red, and purple respectively.}
        \label{fig:intersecting-shape-sharing-root}
    \end{figure}
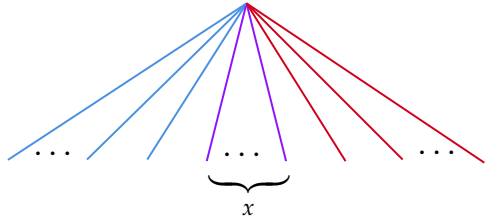
    
    Case 2: Under this case, star-shaped copies $S_1$ and $S_2$ share the same root vertex, as shown in the Figure \ref{fig:intersecting-shape-sharing-root}. Let $ x := |V(S_1 \cup S_2)| - |V(S_1 \triangle S_2)|$. Note that $0 < x < t$. Then, $S_1 \triangle S_2 \cong K_{1, 2t - 2x}$, and by Lemma \ref{lemma:star-shape-copy} we have
    \begin{align*}
        M_{S_1\triangle S_2,H} = M_{K_{1,2t-2x},H} = \sum_{i\in V(H)}(d_i)_{(2t-2x)}.
    \end{align*}

    If $x \ge t/2$, then $S_1 \triangle S_2 \cong K_{1,2t-2x}$ is a star graph with at most $t$ edges, and therefore by the assumption that $\mathbf{S} = \mathbf{K}_{1,t}$ is the star shape that maximizes $\frac{M_{\mathbf{S}, H}^2}{M_{\mathbf{S}}} \left(\frac{1-p}{p}\right)^{|\mathbf{S}|}$ among star shapes with at most $t$ edges, we have
    \begin{align*}
        \frac{M_{S_1 \triangle S_2, H}^2}{M_{S_1 \triangle S_2}} \left(\frac{1-p}{p}\right)^{|S_1 \triangle S_2|} &\le \frac{M_{\mathbf{S}, H}^2}{M_{\mathbf{S}}} \left(\frac{1-p}{p}\right)^{|\mathbf{S}|},\\
        M_{S_1 \triangle S_2, H} &\le M_{\mathbf{S}, H} \sqrt{\frac{M_{S_1 \triangle S_2}}{M_{\mathbf{S}}} \left(\frac{1-p}{p}\right)^{|\mathbf{S}| - |S_1 \triangle S_2|}}\\
        &\le (1+o(1))\cdot n^{(|V(S_1 \triangle S_2)| - |V(\mathbf{S} )|)/2} \left(\frac{1-p}{p}\right)^{(|\mathbf{S}| - |S_1 \triangle S_2| )/2} M_{\mathbf{S}, H}\\
        &= (1+o(1))\cdot n^{(t - 2x)/2} \left(\frac{1-p}{p}\right)^{(2x-t)/2} M_{\mathbf{S}, H},
    \end{align*}
    and
    \begin{align*}
        \frac{n^{|V(S_1 \cup S_2)| - |V(S_1 \triangle S_2)|} M_{S_1 \triangle S_2, H} \left(\frac{1-p}{p}\right)^{|S_1 \triangle S_2|/2} }{M_{\mathbf{S} ,H}^2 \left(\frac{1-p}{p}\right)^{|\mathbf{S}| }}
    \end{align*}
    \begin{align*}
        &\le (1 + o(1))\cdot \frac{\left(n \frac{p}{1-p}\right)^x \left(n \frac{p}{1-p}\right)^{t/2 - x} }{
        M_{\mathbf{S}, H}
        }\\ &\leq (1 + o(1))\cdot  \frac{\left(n \frac{p}{1-p}\right)^{t/2} }{
        M_{\mathbf{S}, H}}\\
        &\le o(1),
    \end{align*}
    as
    $\frac{M_{\mathbf{S}, H }^2}{\left(n \frac{p}{1-p}\right)^{t}} \ge (1- o(1)) \cdot n\cdot  \frac{M_{\mathbf{S}, H }^2}{M_{\mathbf{S}}} \left(\frac{1-p}{p}\right)^{|\mathbf{S}|} \to \infty$ by our assumption.

    Therefore, we now assume $x < t/2$. Since \[\sum_{i \in V(H)} (d_i)_{(2t-2x)} \ge \sum_{i\in V(H)} (d_i)_{(t)} = M_{\mathbf{S}, H} \ge \omega\left(n^{(1+t)/2} (\frac{p}{1-p})^{t/2}\right) = \omega(n),\] we may apply Lemma~\ref{lemma: falling-factorial} and calculate 
    \begin{align*}
        \frac{n^{|V(S_1 \cup S_2)| - |V(S_1 \triangle S_2)|} M_{S_1 \triangle S_2, H} \left(\frac{1-p}{p}\right)^{|S_1 \triangle S_2|/2} }{M_{\mathbf{S} ,H}^2 \left(\frac{1-p}{p}\right)^{|\mathbf{S}| }} &= \frac{\left(n \frac{p}{1-p}\right)^x\sum_{i\in V(H)}(d_{i})_{(2t-2x)} }{
        \left( \sum_{i\in V(H)}(d_{i})_{(t)} \right)^2
        }\\ &\leq (1 + o(1)) \cdot \frac{\left(n \frac{p}{1-p}\right)^x\sum_{i\in V(H)}d_{i}^{2t-2x} }{
        \Big( \sum_{i\in V(H)}d_{i}^{t} \Big)^2}
    \end{align*}
    Let $\Delta = \max_{i \in V(H)} d_i$. We insert Lemma~\ref{lemma: sum-di-ub} and obtain
    \begin{align} \label{ineq:Delta-bound}
        \frac{\left(n \frac{p}{1-p}\right)^x\sum_{i\in V(H)}d_{i}^{2t-2x} }{
        \Big( \sum_{i\in V(H)}d_{i}^{t} \Big)^2} \leq \frac{\left(n \frac{p}{1-p}\right)^x\Delta^{t-2x}}{\sum_{i\in V(H)}d_i^t}.
    \end{align}
    Recalling the assumption that the star-shape $\mathbf{S} = \mathbf{K}_{1,t}$ satisfy $\frac{M_{\mathbf{S}, H}^2}{M_{\mathbf{S}}} \left(\frac{1-p}{p}\right)^{|\mathbf{S}|} \to \infty$, and using Lemma \ref{lemma: falling-factorial}, we have
    \begin{align*}
        \omega(1) \le \frac{M_{\mathbf{S},H}^2}{M_{\mathbf{S}}} \left( \frac{1-p}{p}\right)^{|\mathbf{S}|} = \frac{\big(\sum_{i\in V(H)} (d_i)_{(t)}\big)^2}{(n)_{(t+1)}\left(\frac{p}{1-p}\right)^{t}} \le (1 + o(1)) \cdot \frac{\big(\sum_{i\in V(H)} d_i^t\big)^2}{ n^{t+1}\left( \frac{p}{1-p}\right)^{t}}.
    \end{align*}
    Thus, we obtain
    \begin{align*}
        \frac{1}{\sum_{i\in V(H)} d_i^t} \le o\left(\frac{1}{n^{(t+1)/2} \left(\frac{p}{1-p}\right)^{t/2}}\right),
    \end{align*}
    and plugging it back to \eqref{ineq:Delta-bound}, we get
    \begin{align} \label{ineq:ratio-bound-1}
        \frac{\left(n \frac{p}{1-p}\right)^x\sum_{i\in V(H)}d_{i}^{2t-2x} }{
        \Big( \sum_{i\in V(H)}d_{i}^{t} \Big)^2} \leq \frac{\left(n \frac{p}{1-p}\right)^x\Delta^{t-2x}}{\sum_{i\in V(H)}d_i^t} \le o\left(\frac{\Delta^{t-2x}}{n^{\frac{t+1}{2} - x} \left(\frac{p}{1-p}\right)^{\frac{t}{2} - x} } \right).
    \end{align}
    On the other hand, notice that $\sum_{i \in V(H)} d_i^t \ge \Delta^t$, and therefore
    \begin{align} \label{ineq:ratio-bound-2}
        \frac{\left(n \frac{p}{1-p}\right)^x\sum_{i\in V(H)}d_{i}^{2t-2x} }{
        \Big( \sum_{i\in V(H)}d_{i}^{t} \Big)^2} \leq \frac{\left(n \frac{p}{1-p}\right)^x\Delta^{t-2x}}{\sum_{i\in V(H)}d_i^t} \le \frac{\left(n \frac{p}{1-p}\right)^x}{\Delta^{2x}}.
    \end{align}
    Combining \eqref{ineq:ratio-bound-1} and \eqref{ineq:ratio-bound-2},
    \begin{align} \label{ineq:ratio-bound-combined}
        \frac{\left(n \frac{p}{1-p}\right)^x\sum_{i\in V(H)}d_{i}^{2t-2x} }{
        \Big( \sum_{i\in V(H)}d_{i}^{t} \Big)^2} \leq \min\left\{ o\left(\frac{\Delta^{t-2x}}{n^{\frac{t+1}{2} - x} \left(\frac{p}{1-p}\right)^{\frac{t}{2} - x} } \right), \frac{\left(n \frac{p}{1-p}\right)^x}{\Delta^{2x}}\right\}
    \end{align}
    If $\Delta \geq n^{1/2 + 1/(2t)} \left(\frac{p}{1-p}\right)^{1/2}$, then
    \begin{align*}
        \frac{\left(n \frac{p}{1-p}\right)^x}{\Delta^{2x}} &\le \frac{\left(n \frac{p}{1-p}\right)^x}{n^{x + x/t} \left(\frac{p}{1-p}\right)^x} = \frac{1}{n^{x/t}} = o(1).
    \end{align*}
    Otherwise, if $\Delta \leq n^{1/2 + 1/(2t)} \left(\frac{p}{1-p}\right)^{1/2}$, we can bound
    \begin{align*}
        o\left(\frac{\Delta^{t-2x}}{n^{\frac{t+1}{2} - x} \left(\frac{p}{1-p}\right)^{\frac{t}{2} - x} } \right) \le o\left(\frac{n^{\frac{t+1}{2} - x - \frac{x}{t}} \left(\frac{p}{1-p}\right)^{\frac{t}{2}-x}}{n^{\frac{t+1}{2} - x} \left(\frac{p}{1-p}\right)^{\frac{t}{2} - x} } \right) = o\left(\frac{1}{n^{x/t}}\right) = o(1).
    \end{align*}
    Since in both regimes of $\Delta$ the expression in \eqref{ineq:ratio-bound-combined} is $o(1)$, we conclude that:
    \begin{align*}
        &\quad \frac{n^{|V(S_1 \cup S_2)| - |V(S_1 \triangle S_2)|} M_{S_1 \triangle S_2, H} \left(\frac{1-p}{p}\right)^{|S_1 \triangle S_2|/2} }{M_{\mathbf{S} ,H}^2 \left(\frac{1-p}{p}\right)^{|\mathbf{S}| }}\\ &\leq (1+o(1)) \cdot \frac{\left(n \frac{p}{1-p}\right)^x\sum_{i\in V(H)}d_{i}^{2t-2x} }{
        \Big( \sum_{i\in V(H)}d_{i}^{t} \Big)^2}\\ &\leq o(1),
    \end{align*}
    which finishes the proof for Case $2$.

    Combining the case discussions, we obtain the result.
\end{proof}

\section{Proof of Main Corollary}

\begin{proof}[Proof of Corollary \ref{cor:main}]
    We prove both directions of the stated condition.
    \begin{itemize}
        \item If $\limsup_{n \to \infty} \frac{\sum_{v \in V(H)} d_v^t }{n^{\frac{1+t}{2}} \left(\frac{p}{1-p}\right)^{\frac{t}{2}}} < \infty$ for any $1\le t \le D$, then for any such star, by Lemma \ref{lemma:star-shape-copy}, we have
        \begin{align}
            &\quad \limsup_{n \to \infty} \frac{M_{\mathbf{K}_{1,t}, H}^2}{M_{\mathbf{K}_{1,t}}} \left(\frac{1-p}{p}\right)^{t} \\
            &= \limsup_{n \to \infty} \left(\frac{(1+o(1)) \sum_{v \in V(H)} (d_v)_{(t)} }{ n^{(1+t)/2} \left(\frac{p}{1-p}\right)^{t/2}}\right)^2\\
            &\le \limsup_{n \to \infty} \left(\frac{(1+o(1)) \sum_{v \in V(H)} d_v^t }{n^{(1+t)/2} \left(\frac{p}{1-p}\right)^{t/2}}\right)^2\\
            &< \infty.
        \end{align}
        By the contrapositive of Proposition \ref{prop:star-shape}, the degree-$D$ advantage in this case must be bounded away from infinity. By Theorem \ref{thm:main}, degree-$D$ polynomial tests do not achieve strong separation between $\mathbb{P}$ and $\mathbb{Q}$.

        \item If $\lim_{n \to \infty} \max_{1\le t\le D}\frac{\sum_{v \in V(H)} d_v^t }{n^{\frac{1+t}{2}} \left(\frac{p}{1-p}\right)^{\frac{t}{2}}} = \infty$, by Lemma \ref{lemma:star-shape-copy} and Lemma \ref{lemma: falling-factorial},
        \begin{align}
            &\quad \lim_{n \to \infty} \max_{1\le t\le D}\frac{M_{\mathbf{K}_{1,t}, H}^2}{M_{\mathbf{K}_{1,t}}} \left(\frac{1-p}{p}\right)^{t} \\
            &= \lim_{n \to \infty}\max_{1\le t\le D} \left(\frac{(1+o(1)) \sum_{v \in V(H)} (d_v)_{(t)} }{ n^{(1+t)/2} \left(\frac{p}{1-p}\right)^{t/2}}\right)^2\\
            &= \lim_{n \to \infty}\max_{1\le t\le D} \left(\frac{(1-o(1)) \sum_{v \in V(H)} d_v^t }{ n^{(1+t)/2} \left(\frac{p}{1-p}\right)^{t/2}}\right)^2\\
            &= \infty.
        \end{align}
        By Proposition \ref{prop:Adv}, the square of the degree-$D$ advantage can be written as
        \begin{align}
            \left(\text{Adv}^{\le D}\right)^2 = \sum_{\mathbf{S} : |\mathbf{S}| \le D } \frac{M_{\mathbf{S}, H}^2}{M_\mathbf{S} \cdot |\text{Aut}(\mathbf{S})|} \left(\frac{1-p}{p}\right)^{|\mathbf{S}|}, \label{eq:advantage-sum}
        \end{align}
        which clearly tends to infinity as $n \to \infty$, since $\mathbf{K}_{1,t}$ is one shape that contributes to the summation \eqref{eq:advantage-sum}, its automorphism group has constant size, and all the terms are nonnegative. By Theorem \ref{thm:main}, there exists a degree-$D$ test that achieves strong separation. Moreover, by the proof of Theorem \ref{thm:main} which follows the proof strategy explained in Section \ref{subsection:proof-strategy}, by choosing the star graph that maximizes the rescaled advantage 
        \begin{align}
            \frac{M_{\mathbf{K}_{1,t^*}, H}}{M_{\mathbf{K}_{1,t^*}}^{1/2} \left(\frac{p}{1-p}\right)^{t/2}} = \frac{(1-o(1)) \sum_{v \in V(H)} d_v^t }{ n^{(1+t)/2} \left(\frac{p}{1-p}\right)^{t/2}},    
        \end{align}
        we may use the polynomial $f_{\mathbf{K}_{1,t^*}}$ to strongly separate $\mathbb{P}$ and $\mathbb{Q}$.
    \end{itemize}
\end{proof}

\section{Proof of Characterization Theorem}
Here we prove our characterization theorem of optimal tests based on the maximum degree.
\begin{proof}[Proof of Theorem \ref{thm:degree-characterization}]
    By Theorem \ref{thm:main} and the proof strategy of Theorem \ref{thm:main} described in Subsection \ref{subsection:proof-strategy}, the optimal test is $f_{\mathbf{S}} = \sum_{S \subseteq \binom{V}{2}: S \cong \mathbf{S}} \chi_S$, where $S = \mathbf{K}_{1,t}$ is the star graph that maximizes $\frac{M_{\mathbf{S}, H}^2}{M_{\mathbf{S}}} \left(\frac{1-p}{p}\right)^{|\mathbf{S}|}$ among constant sized star graphs, and $\mathbb{P}$ and $\mathbb{Q}$ are strongly separated by $f_{\mathbf{S}}$ if and only if 
    \begin{equation}
        \frac{M_{\mathbf{S}, H}^2}{M_{\mathbf{S}}} \left(\frac{1-p}{p}\right)^{|\mathbf{S}|} \to \infty. \label{ineq:strong-separation-cond}
    \end{equation}

    Now let us consider for which value of $t$ the condition \eqref{ineq:strong-separation-cond} is achieved when strong separation using constant degree polynomial tests is possible.

    Reusing some computation in the proof of Proposition \ref{prop:intersection-ratio-bound}, we have
    \begin{equation}
        \frac{M_{\mathbf{S}, H}^2}{M_{\mathbf{S}}} \left(\frac{1-p}{p}\right)^{|\mathbf{S}|} \le  (1 + o(1)) \cdot \frac{\left(\sum_{i\in V(H)} d_i^t \right)^2}{n^{1+t}} \left(\frac{1-p}{p}\right)^t. \label{ineq:cond-upper-bound}
    \end{equation}

    If $\Delta \lesssim \left(n \frac{p}{1-p}\right)^{1/2}$, then the upper bound in \eqref{ineq:cond-upper-bound} is non-increasing asymptotically for constant values of $t$:
    \begin{align*}
        \frac{\left(\sum_{i\in V(H)} d_i^t \right)^2}{n^{1+t}} \left(\frac{1-p}{p}\right)^t &= \frac{1}{n} \left( \sum_{i \in V(H)} \left(\frac{d_i}{\left(n \frac{p}{1-p}\right)^{1/2}}\right)^t\right)^2,
    \end{align*}
    as $d_i \le \Delta \lesssim \left(n \frac{p}{1-p}\right)^{1/2}$. Moreover, the upper bound in \eqref{ineq:cond-upper-bound} is asymptotically tight for $t = 1$. Therefore, strong separation is possible only if the condition \eqref{ineq:strong-separation-cond} holds for $t = 1$, which means in this case counting edges is the optimal test.

    On the other hand, if $\Delta \ge \left(n \frac{p}{1-p}\right)^{1/2 + \varepsilon}$ for some constant $\varepsilon > 0$, we again reuse some computation in the proof of Proposition \ref{prop:intersection-ratio-bound} and have
    \begin{align*}
        \frac{M_{\mathbf{S}, H}^2}{M_{\mathbf{S}}} \left(\frac{1-p}{p}\right)^{|\mathbf{S}|} &\ge \frac{\left(\sum_{i \in V(H)} (d_i)_{(t)} \right)^2}{n^{1+t}} \left(\frac{1-p}{p}\right)^{t}
    \end{align*}
    \begin{align*}
        &\ge (1 - o(1)) \cdot \frac{\Delta^{2t}}{n^{1+t} \left(\frac{p}{1-p}\right)^t}\\
        &\ge (1 - o(1)) \cdot \frac{1}{n} \left(\frac{\Delta^2}{n \frac{p}{1-p}}\right)^t\\
        &\ge (1 - o(1)) \cdot \frac{1}{n} \left(n \frac{p}{1-p}\right)^{2\varepsilon t},
    \end{align*}
    which is clearly $\omega(1)$ for a large enough constant $t$. In particular, setting $t = \lceil \frac{3}{2\varepsilon} \rceil + 1$ achieves the condition \eqref{ineq:strong-separation-cond}, and it is easy to check that the quantity in \eqref{ineq:strong-separation-cond} evaluated at $\mathbf{K}_{1,t}$ dominates that evaluated at the edge graph. By the last claim of this theorem that the optimal test is either to count edges or large stars, which we are going to prove, counting signed copies of $\mathbf{K}_{1,t}$ achieves strong separation.

    Lastly, let us consider the square root of the quantity $\frac{M_{\mathbf{S}, H}^2}{M_{\mathbf{S}}} \left(\frac{1-p}{p}\right)^{|\mathbf{S}|}$ and prove that it is sandwiched between two convex functions. Using this strategy, we will show that the optimal test is either to count edges or large stars. By Lemma \ref{lemma:star-shape-copy}, for $\mathbf{S} = \mathbf{K}_{1,t}$, we have
    \begin{align*}
        \frac{M_{\mathbf{S}, H}}{M^{1/2}_{\mathbf{S}}} \left(\frac{1-p}{p}\right)^{|\mathbf{S}|/2} = \frac{\sum_{i\in V(H)} (d_i)_{(t)}}{n^{(1+t)/2}} \left(\frac{1-p}{p}\right)^{t/2}.
    \end{align*}
    We will prove that, for fixed constant $D$, there exists constants $C_1,C_2 > 0$ such that
    \begin{align}
       C_1 \cdot\frac{\sum_{i\in V(H)} d_i^t}{n^{(1+t)/2}} \left(\frac{1-p}{p}\right)^{t/2} - C_2 \leq \frac{\sum_{i\in V(H)} (d_i)_{(t)}}{n^{(1+t)/2}} \left(\frac{1-p}{p}\right)^{t/2} \leq \frac{\sum_{i\in V(H)} d_i^t}{n^{(1+t)/2}} \left(\frac{1-p}{p}\right)^{t/2} \label{ineq:squeeze-bound}
    \end{align}
    holds for any $t \in [D]$. The second inequality is trivial using the definition of falling factorial. We then focus on the first inequality. Let us divide all the vertices $i \in V(H)$ by comparing $d_i$ with $t$:
    \begin{align*}
         \sum_{i\in V(H)} (d_i)_{(t)} = & \sum_{\substack{ v: d_i \geq t}}(d_i)_{(t)} + \sum_{\substack{ v: d_i < t}}(d_i)_{(t)}.
    \end{align*}
    In particular, for any $v$ with $d_i \geq t$, Lemma~\ref{lemma: falling-factorial-lb} indicates that
    \begin{align*}
        (d_i)_{(t)} \geq d_i^t \cdot e^{-\tfrac{t^2}{2(d_i - t+1)}} \geq d_i^t \cdot e^{-\tfrac{t^2}{2}}  \geq  d_i^t \cdot e^{-D^2/2} 
    \end{align*}
    due to $t \leq D$ and $t \leq d_i$. For the parts with $d_i < t$, we compute   
    \begin{align*}
        \frac{\sum_{i\in V(H):d_i < t} d_i^t}{n^{(1+t)/2}}\left(\frac{1-p}{p}\right)^{t/2} \leq \frac{|V(H)| \cdot D^t}{n^{(1+t)/2}}\left(\frac{1-p}{p}\right)^{t/2} \leq \frac{|V(H)|D}{n} \cdot \left(\frac{D^2}{np}\right)^{\tfrac{t-1}{2}} = O(1)
    \end{align*}
    since $p = \Omega(1)$, $|V(H)| \leq n$ and $d_i < t \leq D = O(1)$. Combining the results, we attain
    \begin{align}
        \frac{\sum_{i\in V(H)} d_i^t}{n^{(1+t)/2}} \left(\frac{1-p}{p}\right)^{t/2} &= \frac{\sum_{i\in V(H):d_i<t} d_i^t}{n^{(1+t)/2}} \left(\frac{1-p}{p}\right)^{t/2} + \frac{\sum_{i\in V(H):d_i\geq t} d_i^t}{n^{(1+t)/2}} \left(\frac{1-p}{p}\right)^{t/2}
    \end{align}
    \begin{align}
        & \leq O(1) + e^{D^2/2} \cdot \frac{\sum_{i\in V(H):d_i\geq t} (d_i)_{(t)}}{n^{(1+t)/2}} \left(\frac{1-p}{p}\right)^{t/2} \\
        & = O(1) + e^{D^2/2} \cdot \frac{\sum_{i\in V(H)} (d_i)_{(t)}}{n^{(1+t)/2}} \left(\frac{1-p}{p}\right)^{t/2} \label{ineq:sum-bounds}
    \end{align}
    where the last equality holds because $(d_i)_{(t)} = 0$ for $d_i < t$. Therefore, by \eqref{ineq:sum-bounds}, given fixed $D$, there exists constants $C_1,C_2>0$ independent from $t \in [D]$ such that
    \begin{align*}
         C_1 \cdot\frac{\sum_{i\in V(H)} d_i^t}{n^{(1+t)/2}} \left(\frac{1-p}{p}\right)^{t/2} - C_2 \leq \frac{\sum_{i\in V(H)} (d_i)_{(t)}}{n^{(1+t)/2}} \left(\frac{1-p}{p}\right)^{t/2}, \quad \forall 1\le t \le D,
    \end{align*}
    which finishes the proof of \eqref{ineq:squeeze-bound}.

    According to Theorem \ref{thm:main} and the proof strategy described in Section \ref{subsection:proof-strategy}, counting a star graph $\mathbf{S}$ achieves strong separation as long as $\frac{M_{\mathbf{S},H}}{M^{1/2}_H} \left(\frac{1-p}{p}\right)^{|\mathbf{S}|/2} = \frac{\sum_{i\in V(H)} (d_i)_{(t)}}{n^{(1+t)/2}} \left(\frac{1-p}{p}\right)^{t/2} = \omega(1)$ and it maximizes this rescaled advantage (up to a multiplicative constant) among the set of stars with size at most $|\mathbf{S}|$. 
    Therefore, in order to establish that the optimal test tests is either to count edges or large stars, it suffices show that the quantity $\frac{\sum_{i \in V(H)} (d_i)_{(t)}}{n^{\tfrac{1+t}{2}}\left(\tfrac{p}{1-p}\right)^{t/2} }$ for $1\le t \le D$ is maximized (up to a multiplicative constant) at either $t = 1$ or $t = D$ whenever strong separation can be achieved. Note that whenever strong separation is achieved, the maximum rescaled advantage among stars of size at most $D$ satisfies $\max_{1\le t \le D } \frac{\sum_{i \in V(H)} (d_i)_{(t)}}{n^{\tfrac{1+t}{2}}\left(\tfrac{p}{1-p}\right)^{t/2} } = \omega(1)$ by Corollary \ref{cor:main}. Thus, when strong separation is achieved, we may as well show that $\frac{\sum_{i \in V(H)} d_i^t}{n^{\tfrac{1+t}{2}}\left(\tfrac{p}{1-p}\right)^{t/2} }$ is maximized at either $t = 1$ or $t = D$ with the help of the bounds from \eqref{ineq:squeeze-bound}. We notice that as a function of $t$, $\frac{\sum_{i \in V(H)} d_i^t}{n^{\tfrac{1+t}{2}}\left(\tfrac{p}{1-p}\right)^{t/2} }$ is convex. Thus, when taking values between $1$ and $D$, it is maximized at either $t=1$ or $t = D$. We thus conclude that the optimal test should be either to count edges or to count large stars.

    Finally, as a corollary of our characterization theorem, we identify the phase diagram of the planted subgraph detection problem, as shown in Figure \ref{fig:char}, parametrized by the maximum degree $\Delta$ and the total number of edges $m$. Since we know the optimal test is always to count edges or to count ``large'' stars, it is enough to check if any of the two tests achieve strong separation in each region of the Figure.
    \begin{itemize}
        \item  If $m = \omega\left(n \left(\frac{p}{1-p}\right)^{1/2}\right)$, let us analyze the rescaled advantage of counting signed edges:
        \begin{align}
            \frac{M_{\mathbf{K}_{1,1}, H}}{M_{\mathbf{K}_{1,1}}^{1/2} \left(\frac{p}{1-p}\right)^{1/2}} &= (1-o(1)) \frac{ \sum_{v \in V(H)} d_v }{ n \left(\frac{p}{1-p}\right)^{1/2}}\\
            &= (1-o(1)) \frac{ 2m }{ n \left(\frac{p}{1-p}\right)^{1/2}}, \label{eq:rescaled-adv}
        \end{align}
        where we use the hand-shaking lemma for the number of edges in the last equality. Clearly, $m = \omega\left(n \left(\frac{p}{1-p}\right)^{1/2}\right)$ implies the rescaled advantage in \eqref{eq:rescaled-adv} tends to infinity, and by Corollary \ref{cor:main} we conclude that counting signed edges achieves strong separation.
        \item If $\Delta \ge \left(n \frac{p}{1-p}\right)^{1/2 + \varepsilon}$ for some constant $\varepsilon > 0$, it follows from the characterization theorem that counting ``large'' stars achieves strong separation.
        \item Finally, we turn to the case $\Delta=O\left( \left(n \frac{p}{1-p}\right)^{1/2}\right)$ and $m = O\left(n \left(\frac{p}{1-p}\right)^{1/2}\right)$. By our characterization theorem, the optimal test among constant degree polynomials in this case is to count signed edges. Since $m = O\left(n \left(\frac{p}{1-p}\right)^{1/2}\right)$, the rescaled advantage in \eqref{eq:rescaled-adv} is bounded by a constant. By Corollary \ref{cor:main}, counting signed edges fails to achieve strong separation, which concludes the proof.
    \end{itemize}
\end{proof}

\begin{lemma} \label{lemma: falling-factorial-lb}
    Let $a,b$ be nonnegative integers with $a \geq b$. Then it holds that 
    \begin{align*}
        (a)_{(b)} = \frac{a!}{(a-b)!} \geq a^b \cdot \exp\left(-\tfrac{b^2}{2(a-b+1)}\right)
    \end{align*}
\end{lemma}
\begin{proof}
    We compute as
    \begin{align*}
        \log\left( \frac{a!}{(a-b)!}\right) & = b \log a + \sum_{i=0}^{b-1} \log(1-i/a) \\
        \intertext{using the fact that $\log(1-x) \geq -\frac{x}{1-x}$ holds for any $x \in (0,1)$,}
        & \geq b\log a - \sum_{i=0}^{b-1} \frac{i}{a-i} \\
        & \geq b\log a - \sum_{i=0}^{b-1} \frac{i}{a-b+1} \\
        & \geq b \log a - \frac{b^2}{2(a-b+1)}.
    \end{align*}
    We conclude by taking exponential of both sides.
\end{proof}

    

\section{Proof of Tightness of Main Theorem}

In this section, we prove the tightness of our main theorem: if either $p = \Omega(1)$ or $D = O(1)$ is not satisfied, counting stars could fail to strong separate $\mathbb{P}$ and $\mathbb{Q}$ in the planted subgraph detection task while some other degree-$D$ polynomial does so. Our proofs consider two natural planted subgraph detection settings:
\begin{itemize}
    \item a constant-sized clique planted in a sparse $G(n,p)$,
    \item a clique of size $k = \Theta(\sqrt{n})$ planted in $G(n, \frac{1}{2})$,
\end{itemize}
and show in each case that counting stars does not capture the strong separability of the problem with respect to degree-$D$ polynomials, with $D = O(1)$ in the first setting and $D = O(\log n)$ in the second setting. The main effort in these proofs lies in careful second moment analysis which confirms strong separation is achieved by some natural degree-$D$ polynomial (which of course is not achieved by counting stars) in each setting.

\begin{proof}[Proof of Lemma~\ref{lemma:small-p}]
    For $p = n^{-\gamma}$ where $\gamma \in (0, 1)$ is a constant, let $k = 4/\gamma$. Consider the planted subgraph detection task with a clique of size $k$ planted in $G(n,p)$. We want to show under this setting, it holds that (1) counting stars fails to achieve strong separation and (2) there exists a constant degree polynomial test that strongly separates the two hypotheses. 
    
    For (1), we calculate the advantage of signed count of a star shape $\mathbf{K}_{1,t}$ using Proposition \ref{prop:simple-moments} and Lemma \ref{lemma:star-shape-copy}:
    \begin{align}
        \left(\text{Adv}^{\leq D}\right)^2 &= \frac{\left(\sum_{i\in V(H)} (d_i)_{(t)}\right)^2}{|\text{Aut}(\mathbf{K}_{1,t})|\cdot n^{1+t}} \left( \frac{1-p}{p}\right)^t\\ 
        &\le(1+o(1))\cdot \frac{\left(\sum_{i\in V(H)} d_i^t\right)^2}{n^{1+t}} \left( \frac{1-p}{p}\right)^t\\
        &\leq (1+o(1)) \cdot \frac{k^2}{n} \cdot  \left(\frac{ k^{2}}{n^{1-\gamma} } \right)^t, \label{ineq:small-p-star-advantage}
    \end{align}
    where we use Lemma \ref{lemma: falling-factorial} in the second line. In particular, we plug in $k = 4/\gamma$ and can easily verify that $k^2 \leq n^{1-\gamma} \le n$ holds. Therefore, the advantage \eqref{ineq:small-p-star-advantage} of counting a star shape $\mathbf{K}_{1,t}$ is $O(1)$ for any $t$, which implies that counting star fails to achieve strong separation. 

    For (2), consider the simple polynomial test $f$ which counts the \emph{unsigned} number of \emph{unlabelled} copies of $k$-cliques in $G$, which can be expressed as
    \begin{align}
        f(G) = \sum_{U \subseteq V: |U| = k} \boldsymbol{1}\{G[U] \text{ is a clique}\},
    \end{align}
    and corresponds to a degree-$\binom{k}{2}$ polynomial. We will show that $f$ achieves strong separation for detecting a planted clique of size $k$ in $G(n,p)$. To this end, we need to compute the first and the second moments of $f$ under $\mathbb{P}$ and $\mathbb{Q}$.
    We observe that we always have $f \ge 1$ under $G \sim \mathbb{P}$, so $\mathbb{E}_{\mathbb{P} }[f] \ge 1$. Recall that $p = n^{-\gamma} = o\left(n^{-\frac{2}{k-1}}\right)$ for our choice of $k$. Under $\mathbb{Q}$, we have
    \begin{align}
        \mathbb{E}_{\mathbb{Q} }[f] &= \sum_{U \subseteq V: |U| = k} \mathbb{E}_{\mathbb{Q} }[\boldsymbol{1}\{G[U] \text{ is a clique}\}]\\
        &= \sum_{U \subseteq V: |U| = k} p^{\binom{k}{2}}\\
        &\le n^k p^{\binom{k}{2}}\\
        &= o(1),
    \end{align}
    \begin{align}
        \mathbb{E}_{\mathbb{Q} }[f^2] &= \sum_{U, U'\subseteq V: |U|=|U'|=k} \mathbb{E}_{\mathbb{Q}}[\boldsymbol{1}\{G[U] \text{ is clique}\} \boldsymbol{1}\{G[U'] \text{ is clique}\} ]\\
        &= \sum_{U, U'\subseteq V: |U|=|U'|=k} p^{2\binom{k}{2} - \binom{|U \cap U'|}{2}}\\
        &= \sum_{i=0}^k \sum_{\substack{U, U'\subseteq V:\\ |U|=|U'|=k,\\ |U \cap U'| = i}} p^{2\binom{k}{2} - \binom{i}{2}}\\
        &\le \sum_{i = 0}^k  \binom{n}{i} \binom{n}{k-i} \binom{n}{k-i} p^{k(k-1) - \frac{i(i-1)}{2}}\\
        &\le \sum_{i = 0}^k n^{2k - i} p^{k(k-1) - \frac{i(i-1)}{2}}\\
        &= \sum_{i = 0}^k \left(n^{k} p^{\binom{k}{2} }\right)^{2 - \frac{i}{k}} p^{\frac{i(k-i)}{2}}\\
        &= o(1),
    \end{align}
    and we get $\Var_{\mathbb{Q} }[f] = o(1)$. Now we turn to the second moment of $f$ under $\mathbb{P}$. We first notice that for any realization of the planted $\bold{H}$, we have
    \begin{align}
        \mathbb{E}_{\mathbb{P} }[f^2] = \mathbb{E}_{\mathbb{P} }[f^2 \vert \bold{H}].
    \end{align}
    Thus, we may equivalently consider another distribution $\mathbb{P}'$ where we fix the planted $k$-clique to be on the first $k$ vertices $\{1, \dots, k\} \subseteq V := [n]$. Let us denote this \emph{fixed} set of vertices where the $k$-clique under $\mathbb{P}'$ is planted on as $W$. We may now compute
    \begin{align}
        &\quad \mathbb{E}_{\mathbb{P} }[f^2]\\
        &= \mathbb{E}_{\mathbb{P}' }[f^2]\\
        &= \sum_{\substack{U, U' \subseteq V:\\ |U| = |U'| = k} } \mathbb{E}_{\mathbb{P}' }[\boldsymbol{1}\{G[U] \text{ is clique}\} \boldsymbol{1}\{G[U'] \text{ is clique}\} ]\\
        &= \sum_{x=0}^k \sum_{y=0}^k \sum_{\substack{i \le \min\{k-x, k-y\},\\j \le \min\{x,y\} }} \sum_{\substack{U, U' \subseteq V:\\ |U| = |U'| = k,\\ |U \cap W| = x,\\ |U' \cap W| = y, \\ |U \cap U' \setminus W| = i,\\ |U \cap U' \cap W| = j} } \mathbb{E}_{\mathbb{P}' }[\boldsymbol{1}\{G[U] \text{ is clique}\} \boldsymbol{1}\{G[U'] \text{ is clique}\} ]
    \end{align}
    \begin{align}
        &= 1 + 2\sum_{0 \le x < k} \sum_{\substack{U \subseteq V:\\ |U| = k,\\ |U \cap W| = x} } \mathbb{E}_{\mathbb{P}' }[\boldsymbol{1}\{G[U] \text{ is clique}\} ] \\
        &+ \sum_{0 \le x < k} \sum_{0 \le y < k} \sum_{\substack{i \le \min\{k-x, k-y\},\\j \le \min\{x,y\} }} \sum_{\substack{U, U' \subseteq V:\\ |U| = |U'| = k,\\ |U \cap W| = x,\\ |U' \cap W| = y, \\ |U \cap U' \setminus W| = i,\\ |U \cap U' \cap W| = j} } \mathbb{E}_{\mathbb{P}' }[\boldsymbol{1}\{G[U] \text{ is clique}\} \boldsymbol{1}\{G[U'] \text{ is clique}\} ]\\
        &= 1 + 2\sum_{0 \le x < k} \sum_{\substack{U \subseteq V:\\ |U| = k,\\ |U \cap W| = x} } p^{\binom{k}{2} - \binom{x}{2}} \\
        &+ \sum_{0 \le x < k} \sum_{0 \le y < k} \sum_{\substack{i \le \min\{k-x, k-y\},\\j \le \min\{x,y\} }} \sum_{\substack{U, U' \subseteq V:\\ |U| = |U'| = k,\\ |U \cap W| = x,\\ |U' \cap W| = y, \\ |U \cap U' \setminus W| = i,\\ |U \cap U' \cap W| = j} } p^{2\binom{k}{2} - \binom{x}{2} - \binom{y}{2} - \binom{i}{2} - ij}\\
        &\le 1 + 2\sum_{0 \le x < k} k^x n^{k-x} p^{\binom{k}{2} - \binom{x}{2}} \\
        &+ \sum_{0 \le x < k} \sum_{0 \le y < k} \sum_{\substack{0\le i \le \min\{k-x, k-y\},\\0 \le j \le \min\{x,y\} }} k^{x+y-j} n^i n^{k-x-i}n^{k-y-i} p^{2\binom{k}{2} - \binom{x}{2} - \binom{y}{2} - \binom{i}{2} - ij}\\
        &\le 1 + 2\sum_{0 \le x < k} k^x n^{k-x} p^{\binom{k}{2} - \binom{x}{2}} \\
        &+ 2\sum_{0 \le x \le y < k} \sum_{\substack{0 \le i \le k-y,\\0 \le j \le x }} k^{x+y-j} n^{2k-x-y-i} p^{2\binom{k}{2} - \binom{x}{2} - \binom{y}{2} - \binom{i}{2} - ij}\\
        &\le 1 + C_k \max_{0 \le x < k} n^{k-x} p^{\binom{k}{2} - \binom{x}{2}} \label{ineq:small-p-second-moment-1} \\
        &+ C_k \max_{\substack{0 \le x \le y < k, \\0 \le i \le k-y,\\0 \le j \le x }} n^{2k-x-y-i} p^{2\binom{k}{2} - \binom{x}{2} - \binom{y}{2} - \binom{i}{2} - ij}, \label{ineq:small-p-second-moment-2}
    \end{align}
    where $C_k$ in the last inequality is a constant depending on the constant $k$. Last us separately examine the terms $n^{k-x} p^{\binom{k}{2} - \binom{x}{2}}$ and $n^{2k-x-y-i} p^{2\binom{k}{2} - \binom{x}{2} - \binom{y}{2} - \binom{i}{2} - ij}$ appearing in \eqref{ineq:small-p-second-moment-1} and \eqref{ineq:small-p-second-moment-2}. Recall $p = n^{-\gamma} = o\left(n^{-\frac{2}{k-1}}\right)$. 

    For any $0 \le x < k$, we have
    \begin{align}
        n^{k-x} p^{\binom{k}{2} - \binom{x}{2}} &= \left(n^k p^{\binom{k}{2}} \right)^{1 - \frac{x}{k}} p^{\frac{x(k-x)}{2}}\\
        &= o(1).
    \end{align}
    For any $0 \le x \le y < k$, $0 \le i \le k - y$, and $0 \le j \le x$, we have
    \begin{align}
        &\quad n^{2k-x-y-i} p^{2\binom{k}{2} - \binom{x}{2} - \binom{y}{2} - \binom{i}{2} - ij}\\
        &= \left(n^k p^{\binom{k}{2}}\right)^{2 - \frac{x+y+i}{k}} p^{\frac{x(k-x)}{2} +\frac{y(k-y)}{2}+\frac{i(k-i - 2j)}{2}}. \label{ineq:small-p-second-moment-2-bound}
    \end{align}
    Note that the first exponent $2 - \frac{x + y + i}{k}$ is strictly positive as $x + y + i \le x + k < 2k$, and the second exponent satisfies
    \begin{align}
        &\quad \frac{x(k-x)}{2} +\frac{y(k-y)}{2}+\frac{i(k-i - 2j)}{2}\\
        &\ge \frac{x(k-x)}{2} +\frac{y(k-y)}{2}+\frac{i(k-i - 2x)}{2}\\
        &\ge \frac{x(k-x)}{2} +\frac{y(k-y)}{2}+\frac{i(k-(k-y) - 2x)}{2}\\
        &\ge \frac{x(k-x)}{2} +\frac{y(k-y)}{2}+\frac{-ix}{2}\\
        &\ge \frac{x(k-x)}{2} +\frac{y(k-y)}{2}+\frac{-(k-y)x}{2}\\
        &\ge \frac{x(k-x)}{2} +\frac{(y-x)(k-y)}{2}\\
        &\ge 0,
    \end{align}
    where we repeatedly apply $0\le i \le k - y$, $0\le j \le x$, and $0 \le x \le y < k$. Therefore, we conclude that the expression in \eqref{ineq:small-p-second-moment-2-bound} is $o(1)$. Now that we know both terms in \eqref{ineq:small-p-second-moment-1} and \eqref{ineq:small-p-second-moment-2} are $o(1)$, we use these bounds and conclude that
    \begin{align}
        \mathbb{E}_{\mathbb{P} }[f^2] &\le 1 + C_k \max_{0 \le x < k} n^{k-x} p^{\binom{k}{2} - \binom{x}{2}} + C_k \max_{\substack{0 \le x \le y < k, \\0 \le i \le k-y,\\0 \le j \le x }} n^{2k-x-y-i} p^{2\binom{k}{2} - \binom{x}{2} - \binom{y}{2} - \binom{i}{2} - ij}\\
        &\le 1 + o(1).
    \end{align}
    Since moreover $\mathbb{E}_{\mathbb{P}}[f] \ge 1$, we know $\Var_{\mathbb{P}}[f] = o(1)$.
    As $\Var_{\mathbb{Q}}[f], \Var_{\mathbb{P}}[f] = o(1)$ and $|\mathbb{E}_{\mathbb{P} }[f] - \mathbb{E}_{\mathbb{Q} }[f] | \ge 1 - o(1)$, we conclude that $f$ achieves strong separation.
\end{proof}

\begin{proof}[Proof of Lemma~\ref{lemma:large-D}]
    Let $k = C\sqrt{n}$ where $C>0$ is a constant. Consider the planted subgraph detection task with a clique of size $k$ planted in $G(n,\frac{1}{2})$. We want to show under this setting, it holds that (1) counting stars fails to achieve strong separation and (2) there exists an $O(\log n)$-degree polynomial test that strongly separates the two hypotheses. 

    For (1), let us consider the signed count polynomial $f_{\mathbf{S}}$ where $\mathbf{S}$ is a star shape. When $\mathbf{S}$ is a $t$-star for any $t \ge 2$, $|\text{Aut}(\mathbf{S})| = t!$. By Proposition \ref{prop:simple-moments}, we may calculate the advantage of counting $t$-stars for detecting a planted clique of size $k$ in $G(n, \frac{1}{2})$ as
\begin{align*}
    \left(\text{Adv}(f_{\mathbf{S} })\right)^2 = \frac{M^2_{\mathbf{S},K_{k}}}{|\text{Aut}(\mathbf{S})| \cdot M_{\mathbf{S}}} & = (1+o(1)) \cdot \frac{\left(\sum_{i\in V(K_{k})} (d_i)_{(t)}\right)^2}{n^{1+t} \cdot t!} \\
    & \leq  (1+o(1)) \cdot \frac{(k \cdot k^t)^2 \cdot e^t}{n^{1+t} \cdot t^t}\\
    &\leq \frac{k^2}{n}\left( \frac{ e \cdot k^2}{t\cdot n} \right)^t,
\end{align*}
which is bounded by $O(1)$ for any $t \ge 2$ when $k = C\sqrt{n}$ for a constant $C$. It is also easy to check that counting edges does not achieve strong separation in this case. We thus conclude that counting stars fails to strongly separate $\mathbb{P}$ under this setting for any $t$. 

Now we turn to prove (2). The seminal work of \cite{alon1998finding} proved that a spectral method successfully detects a planted $k$-clique with high probability when $k = C\sqrt{n}$ for a large enough constant $C > 0$, and it is known that such spectral method can be well approximated by $O(\log n)$-degree polynomials (see \cite{kunisky2019notes}, \cite{gamarnik2020low}). However, these results do not address the aspect of strong separation. Here, we will show that some $O(\log n)$-degree polynomial does achieve strong separation when $k = C\sqrt{n}$ for a large enough constant $C > 0$, following a slight variant of the polynomial that approximates the trace method.

For convenience, in the following discussion we will use $M$ to denote the $\{+1, -1\}$ adjacency matrix of $G$ drawn from the null distribution $\mathbb{Q}$ or the planted distribution $\mathbb{P}$, where an entry $(i,j)$ is $1$ if $\{i,j\}$ is an edge in $G$, $-1$ if $\{i,j\}$ is not an edge, and $0$ on the diagonal $i = j$. Let $l = B\log n$, where $B$ is a large enough constant. Consider the following polynomial \[f(M) = \sum_{\substack{(i_1, \dots, i_l):\\ i_t \in [n],\text{ all distinct}}} M_{i_1, i_2} M_{i_2, i_3} \dots M_{i_{l-1}, i_l} M_{i_l, i_1}\] of $M$, which is a degree-$l$ polynomial in the entries of $M$. Moreover, we will call $\overline{i} = (i_1, \dots, i_l)$ a (simple) closed path (of length $l$ on $l$ vertices), and use $M^{\overline{i}}:=M_{i_1, i_2} M_{i_2, i_3} \dots M_{i_{l-1}, i_l} M_{i_l, i_1}$ to denote this product that corresponds to the closed path $\overline{i}$.

Under the null distribution $\mathbb{Q}$, we have
\begin{align}
    \mathbb{E}_{\mathbb{Q} }[f(M)] &= 0,\\
    \mathbb{E}_{\mathbb{Q}}[f(M)^2] &= \sum_{\text{closed paths } \overline{i}, \overline{j}} \mathbb{E}_{\mathbb{Q}}\left[M^{\overline{i}} M^{\overline{j}}\right]\\
    &= \sum_{\text{closed paths } \overline{i}, \overline{j}} \boldsymbol{1}\{E(\overline{i}) = E(\overline{j})\}\\
    &= n(n-1)\dots(n-l+1)\cdot 2l\\
    &= (1 + o(1))2ln^l.
\end{align}

Let us now consider the planted distribution $\mathbb{P}$. Since the planted clique can be specified by a subset of vertices the clique is planted on, we will denote this subset of vertices as $\bold{W}$, where we use bold letter to emphasize it is uniformly random among all subsets of size $k$. Under $\mathbb{P}$, we have
\begin{align}
    \mathbb{E}_{\mathbb{P} }[f(M)] &= \sum_{\text{closed path } \overline{i}} \mathbb{E}_{\mathbb{P} }\left[M^{\overline{i}}\right]\\
    &= \sum_{\text{closed path } \overline{i}} \mathbb{E}_{\mathbb{P} }\left[\boldsymbol{1}\{i_1, \dots, i_l \in \bold{W}\}\right],
\end{align}
since conditioning on the planted clique, $\mathbb{E}_{\mathbb{P} }\left[M^{\overline{i}} \Big\vert \bold{W}\right] = 0$ whenever $\overline{i}$ is not fully contained in the planted clique, and $1$ otherwise. We further observe that the probability that any fixed closed path $\overline{i}$ of length $l$ is contained in $\bold{W}$ is simply $\frac{k(k-1)\dots (k-l+1)}{n(n-1)\dots (n-l+1)} = (1 + o(1))\cdot \left(\frac{k}{n}\right)^l$, which can also be verified using Lemma \ref{lemma:inclusion-prob}. As a result,
\begin{align}
    \mathbb{E}_{\mathbb{P} }[f(M)] &= \sum_{\text{closed path } \overline{i}} \mathbb{E}_{\mathbb{P} }\left[\boldsymbol{1}\{i_1, \dots, i_l \in \bold{W}\}\right]\\
    &= \sum_{\text{closed path } \overline{i}} \mathbb{P} (i_1, \dots, i_l \in \bold{W})\\
    &= (1 + o(1)) n(n-1)\dots(n-l+1) \cdot \left(\frac{k}{n}\right)^l\\
    &= (1 + o(1)) k^l.
\end{align}

Finally, we turn to the second moment of $f(M)$ under $\mathbb{P}$. We have
\begin{align}
    \mathbb{E}_{\mathbb{P}}[f(M)^2] &= \sum_{\text{closed paths } \overline{i}, \overline{j}} \mathbb{E}_{\mathbb{P}}\left[M^{\overline{i}} M^{\overline{j}}\right]\\
    &= \sum_{\text{closed paths } \overline{i}, \overline{j}} \mathbb{E}_{\mathbb{P} } \left[\boldsymbol{1}\left\{V\left(\overline{i} \triangle \overline{j} \right) \subseteq \bold{W}\right\}\right]
\end{align}
\begin{align}
    &= (1 + o(1))\sum_{\text{closed paths } \overline{i}, \overline{j}} \left(\frac{k}{n}\right)^{\left|V\left(\overline{i} \triangle \overline{j} \right)\right|},
\end{align}
following similar reasoning as before, where $\overline{i} \triangle \overline{j}$ denotes the symmetric difference shape of the two closed paths $\overline{i}$ and $\overline{j}$, which in particular, as we recall from Definition \ref{dfn:sym-diff-shape}, does not contain any isolated vertex. Note that for the case of vertex-disjoint $\overline{i}$ and $\overline{j}$, there are approximately $n^{2l}$ many such terms in the sum, each contributing approximately $\left(\frac{k}{n}\right)^{2l}$, and thus the total contribution is close to $k^{2l}$ which matches the square of the first moment of $f(M)$ under $\mathbb{P}$. If we can show that the rest of the contribution from the other terms corresponding to pairs of $\overline{i}$ and $\overline{j}$ that are not vertex-disjoint\footnote{In fact, one can even follow the argument used in the proof of our main theorem to only focus on the pairs of $\overline{i}$ and $\overline{j}$ that are not edge-disjoint. Nevertheless, the argument we present here suffices for the proof.} is much smaller than $k^{2l}$, then we would be done.

Let us further break down the situation when $\overline{i}$ and $\overline{j}$ are not vertex-disjoint into cases. In the following definitions, the indices in the constraints are cyclic modulo $l$ (e.g., $i_{t'-1}$ refers to $i_l$ if $t' = 1$). For a pair of closed paths $\overline{i}$ and $\overline{j}$, define
\begin{align}
    A &:= \{t \in [l]: \forall t'\in [l], j_t \ne i_{t'}\},\\
    B &:= \{t \in [l]: \exists t' \in [l], \text{s.t. } j_t = i_{t'}, j_{t-1} \ne i_{t' - 1}, j_{t + 1} \ne i_{t' + 1}\},\\
    S &:= \{\{j_t, j_{t+1}\}, t \in [l]: \exists t' \in [l], \text{s.t. } \{j_t, j_{t+1}\} = \{i_{t'}, i_{t'+1}\} \},\\
    a &:= |A|,\\
    b &:= |B|,\\
    s &:= |S|,
\end{align}
where $A$ is the index set of vertices in $\overline{j}$ not shared with $\overline{i}$, $B$ is the index set of vertices in $\overline{j}$ shared with $\overline{i}$ that do not participate in any shared edges, and $S$ is the set of edges in $\overline{j}$ shared with $\overline{i}$. We moreover observe that if we restrict the attention to the set $S$ of shared edges, they form a number of connected components, and we denote this number as $c$. Formally,
\begin{align}
    C &:= \{t \in [l]: \{j_t, j_{t+1}\} \in S, \{j_{t-1}, j_t\} \not\in S \},\\
    c &:= |C|,\\
    CC &:= \{(j_t, j_{t+1}, \dots, j_{t+s_t}): t \in C, s_t \text{ is the maximum such that }\\
    &\quad\quad \forall t \le r \le t+s_t-1, \{j_{r}, j_{r+1}\} \in S\},
\end{align}
where $C$ is the set of the starting indices in the closed path $\overline{j}$ of the connected components of the set $S$ of shared edges (if $\overline{i}$ and $\overline{j}$ do not overlap completely), and $CC$ is the collection of connected components (sub-paths of $\overline{j}$) of $S$. Let us state as a fact that whenever $s = |S| < l$, the parameters $a,b,s,c$ satisfy the identity $l = a+b+s+c$, which is easy to verify.

Now, we claim that $\left|V\left(\overline{i} \triangle \overline{j} \right)\right| = 2l - b - 2s$. To see this, let us consider which vertices in $V\left(\overline{i}\right) \cup V\left(\overline{j}\right)$ belong to $V\left(\overline{i} \triangle \overline{j} \right)$. If $s = |S| = l$, then the two closed walks overlap completely, and the claim obviously holds. So let us consider otherwise. 
\begin{itemize}
    \item First, all vertices in $V\left(\overline{i}\right) \cup V\left(\overline{j}\right)$ that do not correspond to any shared vertex or participate in any shared edges belong to $V\left(\overline{i} \triangle \overline{j} \right)$, and there are $2(l - b - (s+c)) = 2a$ such vertices.
    \item Second, each pair of shared vertices from $\overline{i}$ and $\overline{j}$ that do not participate in shared edges (corresponding to vertices in $\overline{j}$ indexed by $B$) contributes $1$ vertex to $V\left(\overline{i} \triangle \overline{j} \right)$, and there are $b$ of them.
    \item Third, every connected component of shared edges in $S$ contributes $2$ vertices, since each connected component is a sub-path shared by $\overline{i}$ and $\overline{j}$, and only the two endpoints of the sub-path survive the symmetric difference operation. This gives a total of $2c$ vertices.
\end{itemize}
From the analysis above, we have $\left|V\left(\overline{i} \triangle \overline{j} \right)\right| = 2a + b + 2c = 2l - b - 2s$.

Moreover, we may estimate the number of pairs of $\overline{i}$ and $\overline{j}$ with the specific choice of parameters $a,b,s,c$ in the following way:
\begin{itemize}
    \item WLOG, we enumerate the number of $\overline{i}$ as $n(n-1)\dots(n-l+1) \le n^l$.
    \item Next, we enumerate the vertices indexed in $A$ in the order they are traversed in $\overline{j}$, creating $\le n^a$ choices.
    \item Then, we enumerate the vertices in $\overline{j}$ indexed in $B$, by enumerating the indices $t$ and the matching indices $t'$ of $\overline{i}$, creating $\le (l \times l)^b = l^{2b}$ choices.
    \item Finally, we enumerate the vertices that participate in the shared edges according to the connected components the shared edges form. For each of the $c$ connected components, we enumerate the starting index $t$ in $\overline{j}$, the matching index $t'$ in $\overline{i}$, the size $s_t$ of this connected component, and the direction (whether they follow the forward direction $j_t = j_{t'}, j_{t+1} = i_{t'+1}, j_{t+2} = i_{t'+2}, \dots$ or the backward direction $j_t = j_{t'}, j_{t+1} = i_{t'-1}, j_{t+2} = i_{t'-2}, \dots$) of this component of shared edges, creating $\le (l \times l \times l \times 2)^{c} = (2l^3)^c$ choices.
\end{itemize}
It is not hard to see one can uniquely recover a pair of $\overline{i}$ and $\overline{j}$ using the information above. Thus, the total number of pairs of $\overline{i}$ and $\overline{j}$ with the specific choice of parameters $a,b,s,c$ is bounded by 
\begin{align}
    N(a,b,s,c) \le n^l\cdot n^a \cdot l^{2b} \cdot (2l^3)^c = 2^c l^{2b+3c} n^{l+a}.
\end{align}
Now we are ready to prove a bound on the second moment of $f(M)$ under $\mathbb{P}$:
\begin{align}
    \mathbb{E}_{\mathbb{P}}[f(M)^2] &= (1 + o(1))\sum_{\text{paths } \overline{i}, \overline{j}} \left(\frac{k}{n}\right)^{\left|V\left(\overline{i} \triangle \overline{j} \right)\right|}\\
    &= (1 + o(1)) \sum_{a,b,s,c} \sum_{\substack{\text{paths } \overline{i}, \overline{j}:\\ \overline{i}, \overline{j} \text{ satisfies the parameters } a,b,s,c }} \left(\frac{k}{n}\right)^{2l - b - 2s}\\
    &= (1 + o(1)) \sum_{a,b,s,c} N(a,b,s,c) \left(\frac{k}{n}\right)^{2l - b - 2s}\\
    \intertext{note that when $\overline{i}, \overline{j}$ are vertex-disjoint, the parameters are $(a = l, b = 0, s= 0, c=0)$}
    &\le (1 + o(1))\left[ n^{2l} \cdot \left(\frac{k}{n}\right)^{2l} +  \sum_{\substack{a,b,s,c:\\ a < l}} 2^c l^{2b+3c} n^{l+a} \left(\frac{k}{n}\right)^{2l - b - 2s}\right].
\end{align}
Let us examine one term in the sum above:
\begin{align}
    2^c l^{2b+3c} n^{l+a} \left(\frac{k}{n}\right)^{2l - b - 2s} &= 2^cl^{2b+3c} k^{2l - b - 2s} n^{a + b+2s - l}\\
    &= 2^cl^{2b+3c} k^{2l - b - 2s} n^{s - c}\\
    &= k^{2l} \left(\frac{l^2}{k}\right)^b \left(\frac{2l^3}{n}\right)^c \left(\frac{n}{k^2}\right)^s,
\end{align}
where we use the previously stated identity $l = a + b + s + c$ in the second to last line. Moreover, when $\overline{i}$ and $\overline{j}$ are not vertex-disjoint (i.e., $a < l$), either $b > 0$ or $c > 0$. In either situation, when $k = C\sqrt{n}$ for a large constant $C$ and $l = \Theta(\log n)$, the expression $k^{2l} \left(\frac{l^2}{k}\right)^b \left(\frac{2l^3}{n}\right)^c \left(\frac{n}{k^2}\right)^s$ above is $o(k^{2l} / n^{0.49})$. Since there are at most $l^4 = o(n^{0.01})$ choices of parameters $a,b,s,c$ that satisfies $a < l$, and  each term in the sum contributes $o(k^{2l} / n^{0.49})$, we conclude that
\begin{align}
    \mathbb{E}_{\mathbb{P}}[f(M)^2] &\le (1 + o(1))\left[ n^{2l} \cdot \left(\frac{k}{n}\right)^{2l} +  \sum_{\substack{a,b,s,c:\\ a < l}} 2^c l^{2b+3c} n^{l+a} \left(\frac{k}{n}\right)^{2l - b - 2s}\right]\\
    &\le (1 + o(1))\left[ k^{2l} +  k^{2l}/n^{0.48}\right]\\
    &\le (1 + o(1)) k^{2l}.
\end{align}

Together with the first moment under $\mathbb{P}$, we get that
$\Var_{\mathbb{P} }[f(M)] = o(k^{2l})$. From the moment computation for $\mathbb{Q}$, we also have $\Var_{\mathbb{Q} }[f(M)] = O(ln^l) = o(k^{2l})$, and $|\mathbb{E}_{\mathbb{P} }[f(M)] - \mathbb{E}_{\mathbb{Q} }[f(M)] | = (1 - o(1))k^l$. We have thus verified that $f(M)$ achieves strong separation for detecting a planted clique of size $k$ in $G(n, \frac{1}{2})$.

\end{proof}

\section{Conclusion and Future Directions}

In this paper, we initiate the unified study of the computational thresholds in detecting \emph{arbitrary} planted subgraph structures in $G(n,p)$. We give a complete characterization of the strong separation power of \emph{constant} degree polynomials in the regime of $p = \Omega(1)$.  In particular, we reveal that under these assumptions, it is always optimal to count stars among all constant-degree polynomials.

Our work suggests many future directions. \begin{enumerate}
    \item At a conceptual level, we believe our results make a strong case that studying ``unified'' planted random graph models, containing as special cases multiple well-studied models, is very beneficial. In our case, it was the generality of the studied model (Definition \ref{dfn:planted}) that allowed us to reveal the (perhaps surprising) constant-degree optimality of the star counts. To the best of our knowledge, this optimality has not been observed before for any specific case of planted $H$. It is interesting what other common structural computational properties are satisfied by all planted subgraph models.

    \item In terms of specific directions, we consider a very interesting project to study what is the optimal $D$-degree polynomial when either $D=\omega(1)$ or $p=o(1)$ that our established star-count optimality fails (see Section \ref{subsection:failure-of-counting-stars}). Moreover, our proof techniques carefully leverage the graph structure and do not trivially extend to hypergraphs. It is a nice question whether and how these phenomena generalized to planted hypergraph settings.
    
    \item Recent work \cite{montanari2022equivalence} suggests the constant-degree optimality of counting trees (equivalently of Approximate Message Passing) in terms of \emph{recovering} a hidden spike in the spiked Wigner model (with a ``dense'' prior). While no trivial connection with our work is possible (we study the \emph{detection version} between Bernoulli graph models), it is an exciting direction to study possible connections between the star-optimality from our work and the tree-optimality from \cite{montanari2022equivalence}.
    
\end{enumerate} 



\section*{Acknowledgements}
    The authors are thankful to Alex~Wein and Tim Kunisky for useful comments. X.Y.~was partially supported by a Simons Investigator award from the Simons Foundation to Daniel Spielman. P.Z. is supported by a Yale University Fund to Amin Karbasi.

\newpage
\bibliographystyle{plain}
\bibliography{pc}

\begin{thebibliography}{10}

\bibitem{alon1998finding}
Noga Alon, Michael Krivelevich, and Benny Sudakov.
\newblock Finding a large hidden clique in a random graph.
\newblock {\em Random Structures \& Algorithms}, 13(3-4):457--466, 1998.

\bibitem{bagaria2020hidden}
Vivek Bagaria, Jian Ding, David Tse, Yihong Wu, and Jiaming Xu.
\newblock Hidden hamiltonian cycle recovery via linear programming.
\newblock {\em Operations research}, 68(1):53--70, 2020.

\bibitem{barak2019nearly}
Boaz Barak, Samuel Hopkins, Jonathan Kelner, Pravesh~K Kothari, Ankur Moitra,
  and Aaron Potechin.
\newblock A nearly tight sum-of-squares lower bound for the planted clique
  problem.
\newblock {\em SIAM Journal on Computing}, 48(2):687--735, 2019.

\bibitem{brennan2018reducibility}
Matthew Brennan, Guy Bresler, and Wasim Huleihel.
\newblock Reducibility and computational lower bounds for problems with planted
  sparse structure.
\newblock In {\em Conference On Learning Theory}, pages 48--166. PMLR, 2018.

\bibitem{coja2022statistical}
Amin Coja-Oghlan, Oliver Gebhard, Max Hahn-Klimroth, Alexander~S Wein, and
  Ilias Zadik.
\newblock Statistical and computational phase transitions in group testing.
\newblock In {\em Conference on Learning Theory}, pages 4764--4781. PMLR, 2022.

\bibitem{dhawan2023detection}
Abhishek Dhawan, Cheng Mao, and Alexander~S Wein.
\newblock Detection of dense subhypergraphs by low-degree polynomials.
\newblock {\em arXiv preprint arXiv:2304.08135}, 2023.

\bibitem{gamarnik2020low}
David Gamarnik, Aukosh Jagannath, and Alexander~S Wein.
\newblock Low-degree hardness of random optimization problems.
\newblock In {\em 2020 IEEE 61st Annual Symposium on Foundations of Computer
  Science (FOCS)}, pages 131--140. IEEE, 2020.

\bibitem{hajek2015computational}
Bruce Hajek, Yihong Wu, and Jiaming Xu.
\newblock Computational lower bounds for community detection on random graphs.
\newblock In {\em Conference on Learning Theory}, pages 899--928. PMLR, 2015.

\bibitem{hopkins2018statistical}
Samuel Hopkins.
\newblock {\em Statistical inference and the sum of squares method}.
\newblock PhD thesis, Cornell University, 2018.

\bibitem{huleihel2022inferring}
Wasim Huleihel.
\newblock Inferring hidden structures in random graphs.
\newblock {\em IEEE Transactions on Signal and Information Processing over
  Networks}, 8:855--867, 2022.

\bibitem{Jer92}
Mark Jerrum.
\newblock Large cliques elude the metropolis process.
\newblock {\em Random Structures \& Algorithms}, 3(4):347--359, 1992.

\bibitem{jones2022sum}
Chris Jones, Aaron Potechin, Goutham Rajendran, Madhur Tulsiani, and Jeff Xu.
\newblock Sum-of-squares lower bounds for sparse independent set.
\newblock In {\em 2021 IEEE 62nd Annual Symposium on Foundations of Computer
  Science (FOCS)}, pages 406--416. IEEE, 2022.

\bibitem{kahn2007thresholds}
Jeff Kahn and Gil Kalai.
\newblock Thresholds and expectation thresholds.
\newblock {\em Combinatorics, Probability and Computing}, 16(3):495--502, 2007.

\bibitem{kunisky2019notes}
Dmitriy Kunisky, Alexander~S Wein, and Afonso~S Bandeira.
\newblock Notes on computational hardness of hypothesis testing: Predictions
  using the low-degree likelihood ratio.
\newblock In {\em ISAAC Congress (International Society for Analysis, its
  Applications and Computation)}, pages 1--50. Springer, 2019.

\bibitem{massoulie2019planting}
Laurent Massouli{\'e}, Ludovic Stephan, and Don Towsley.
\newblock Planting trees in graphs, and finding them back.
\newblock In {\em Conference on Learning Theory}, pages 2341--2371. PMLR, 2019.

\bibitem{moharrami2021planted}
Mehrdad Moharrami, Cristopher Moore, and Jiaming Xu.
\newblock The planted matching problem: Phase transitions and exact results.
\newblock {\em The Annals of Applied Probability}, 31(6):2663--2720, 2021.

\bibitem{montanari2022equivalence}
Andrea Montanari and Alexander~S Wein.
\newblock Equivalence of approximate message passing and low-degree polynomials
  in rank-one matrix estimation.
\newblock {\em arXiv preprint arXiv:2212.06996}, 2022.

\bibitem{montgomery2019spanning}
Richard Montgomery.
\newblock Spanning trees in random graphs.
\newblock {\em Advances in Mathematics}, 356:106793, 2019.

\bibitem{mossel2023sharp}
Elchanan Mossel, Jonathan Niles-Weed, Youngtak Sohn, Nike Sun, and Ilias Zadik.
\newblock Sharp thresholds in inference of planted subgraphs.
\newblock In {\em The Thirty Sixth Annual Conference on Learning Theory}, pages
  5573--5577. PMLR, 2023.

\bibitem{mossel2022second}
Elchanan Mossel, Jonathan Niles-Weed, Nike Sun, and Ilias Zadik.
\newblock On the second kahn--kalai conjecture.
\newblock {\em arXiv preprint arXiv:2209.03326}, 2022.

\bibitem{park2022proof}
Jinyoung Park and Huy~Tuan Pham.
\newblock A proof of the {K}ahn--{K}alai conjecture.
\newblock arXiv:2203.17207, 2022.

\bibitem{posa1976hamiltonian}
Lajos P{\'o}sa.
\newblock Hamiltonian circuits in random graphs.
\newblock {\em Discrete Mathematics}, 14(4):359--364, 1976.

\bibitem{wein2022optimal}
Alexander~S Wein.
\newblock Optimal low-degree hardness of maximum independent set.
\newblock {\em Mathematical Statistics and Learning}, 4(3):221--251, 2022.

\end{thebibliography}

\clearpage

\appendix

\section{Proof of Auxillary Lemmas}\label{subsection:auxillary-proof}

\begin{proof}[Proof of Proposition \ref{prop:Adv}]
    Recall that the fact that $\chi_S(X) = \prod_{\{i,j\} \in E(S)} \frac{X_{i,j} - p}{\sqrt{p(1-p)}}$ for $S \subseteq \binom{V}{2}: |S| \le D$ form an orthonormal basis of the multilinear polynomials in $\mathbb{R}[X]^{\le D}$ with respect to $\langle \cdot, \cdot \rangle_{\mathbb{Q}}$. Therefore, we may expand any polynomial $f \in \mathbb{R}[X]_{\le D}$ in this basis as
    \begin{align*}
        f(X) = \sum_{S \subseteq \binom{V}{2}: |S| \le D} \langle f, \chi_S \rangle_{\mathbb{Q} } \chi_S = \sum_{S \subseteq \binom{V}{2}: |S| \le D} \hat{f}_{S} \chi_S,
    \end{align*}
    where $\hat{f}_{S} := \langle f, \chi_S \rangle_{\mathbb{Q} }$ are the Fourier coefficients of $f$. Thus, we may compute
    \begin{align}
        \text{Adv}^{\le D} &= \max_{f \in \mathbb{R}[X]^{\le D}}\frac{\mathbb{E}_{\mathbb{P}}[f]}{\sqrt{\mathbb{E}_{\mathbb{Q}}[f^2]}}\\
        &= \max_{\substack{\{\hat{f}_S\}\\ S \subseteq \binom{V}{2}: |S| \le D }}\frac{\mathbb{E}_{\mathbb{P}}\left[\sum_S \hat{f}_S \chi_S\right]}{\sqrt{\mathbb{E}_{\mathbb{Q}}\left[\left(\sum_S \hat{f}_S \chi_S\right)^2\right]}}\\
        &= \max_{\substack{\{\hat{f}_S\}\\ S \subseteq \binom{V}{2}: |S| \le D }}\frac{\sum_S \hat{f}_S \mathbb{E}_{\mathbb{P}}\left[\chi_S\right]}{\sqrt{ \sum_{S, S'} \hat{f}_S \hat{f}_{S'}\mathbb{E}_{\mathbb{Q}}\left[\chi_S\chi_{S'}\right]}}\\
        &= \max_{\substack{\{\hat{f}_S\}\\ S \subseteq \binom{V}{2}: |S| \le D }}\frac{\sum_S \hat{f}_S \mathbb{E}_{\mathbb{P}}\left[\chi_S\right]}{\sqrt{ \sum_{S} \hat{f}_S^2 }}\\
        &= \sqrt{\sum_{S \subseteq \binom{V}{2}: |S| \le D} \mathbb{E}_{\mathbb{P} }[\chi_S]^2 }.
    \end{align}

    We may compute the expectation $\mathbb{E}_{\mathbb{P}}[\chi_S]$ by first conditioning on $\bold{H}$ and then taking the expectation over random $\bold{H}$: 
    \begin{align}
        \mathbb{E}_{\mathbb{P} }[\chi_S] &= \mathbb{E}_{\bold{H}} \mathbb{E}_{\mathbb{P} }[\chi_S | \bold{H}].
    \end{align}
    Conditioning on a fixed $\bold{H}$, $\mathbb{E}_{\mathbb{P} }[\chi_S | \bold{H}] = 0$ whenever $S$ is not fully contained in $\bold{H}$, and $\mathbb{E}_{\mathbb{P} }[\chi_S | \bold{H}] = \left(\frac{1-p}{p}\right)^{|S|/2}$ if $S \subseteq \bold{H}$. Therefore, we have 
    \begin{align}
        \mathbb{E}_{\mathbb{P} }[\chi_S | \bold{H}] = \boldsymbol{1}(S \subseteq \bold{H})\left(\frac{1-p}{p}\right)^{|S|/2}.
    \end{align}
    Thus,
    \begin{align} \label{eq: chi_S_P}
        \mathbb{E}_{\mathbb{P} }[\chi_S] &= \mathbb{E}_{\bold{H}} \boldsymbol{1}(S \subseteq \bold{H})\left(\frac{1-p}{p}\right)^{|S|/2}\nonumber\\
        &= \mathbb{P}(S \subseteq\bold{H}) \left(\frac{1-p}{p}\right)^{|S|/2}\nonumber\\
        &= \frac{M_{S,H}}{M_S} \left(\frac{1-p}{p}\right)^{|S|/2},
    \end{align}
    where the last line uses Lemma \ref{lemma:inclusion-prob}. Plugging this expression back in \eqref{eq:deg-D-Adv}, we get
    \begin{align}
        \left(\text{Adv}^{\le D}\right)^2 = \sum_{S \subseteq \binom{V}{2}: |S| \le D} \frac{M_{S,H}^2}{M_S^2} \left(\frac{1-p}{p}\right)^{|S|}. \label{eq:deg-D-Adv}
    \end{align}
    Finally, we group the summation over subsets $S \subseteq \binom{V}{2}: |S| \le D$ according to the isomorphism classes of the shapes $\mathbf{S}$. For each shape $\mathbf{S}$, there are $\frac{M_{\mathbf{S}}}{|\text{Aut}({\mathbf{S}})|}$ subsets of $\binom{V}{2}$ (unlabelled copies) that are isomorphic to $\mathbf{S}$, and for isomorphic subsets, the expressions in the summation are equal. As a result, we may rewrite the summation as
    \begin{align}
        \left(\text{Adv}^{\le D}\right)^2 &= \sum_{\mathbf{S}: |\mathbf{S}| \le D} \frac{M_{\mathbf{S}}}{|\text{Aut}(\mathbf{S})|} \cdot \frac{M_{{\mathbf{S}},H}^2}{M_{\mathbf{S}}^2} \left(\frac{1-p}{p}\right)^{|{\mathbf{S}}|}\\
        &= \sum_{\mathbf{S} : |\mathbf{S}| \le D } \frac{M_{\mathbf{S}, H}^2}{M_\mathbf{S} \cdot |\text{Aut}(\mathbf{S})|} \left(\frac{1-p}{p}\right)^{|\mathbf{S}|}
    \end{align}
\end{proof}

\begin{proof}[Proof of Lemma \ref{lemma:star-shape-copy}]
    For the enumeration of $M_{\mathbf{S},H}$, we first fix vertex $i \in V(H)$ and count the number of copies of $\mathbf{K}_{1,t}$ that have its root at $i$. Suppose the degree of $i$ is $d_i$, then the number of copies of $\mathbf{K}_{1,t}$ rooted at $i$ is $(d_i)_{(t)}$. The total number of copies of $\mathbf{K}_{1,t}$ in $H$ is obtained by summing over all the vertex $i\in V(H)$:
    \begin{align*}
        M_{\mathbf{S}, H} = \sum_{i\in V(H)} (d_i)_{(t)}.
    \end{align*}
  
\end{proof}

\begin{proof}[Proof of Lemma \ref{lemma: falling-factorial}]
Clearly, one side of the inequality is obvious $\sum_{i \in [k]} (d_i)_{(t)} \le \sum_{i \in [k]} d_i^t$, so we focus on the other inequality.

We have for all $i\in [k]$, $$(d_i)_{(t)} \geq d_i^t (1-t/d_i)^t 1(d_i \geq t) \geq d_i^t(1-t^2/d_i)\boldsymbol{1}\{d_i \geq t\} \geq d_i^t-t^t-t^2 d_i^{t-1}.$$

Hence,

\begin{align}
    \frac{\sum_{i \in [k]} (d_i)_{(t)}}{\sum_{i \in [k]} d_i^t} \geq 1-\frac{k t^t}{\sum_{i \in [k]} d_i^t}-t^2\frac{\sum_{i \in [k]} d_i^{t-1}}{\sum_{i \in [k]} d_i^t
}. \label{ineq:falling-factorial-bound}
\end{align}

Clearly, by our assumption,
\begin{align}
    \frac{k t^t}{\sum_{i \in [k]} d_i^t}=o(1). \label{ineq:falling-factorial-sub-bound-1}
\end{align}

Also by H\"{o}lder inequality, 
\begin{align}
    t^2\frac{\sum_{i \in [k]} d_i^{t-1}}{\sum_{i \in [k]} d_i^t} \leq t^2 \frac{\left(\sum_{i \in [k]} d_i^{t}\right)^{1-1/t}k^{1/t}}{\sum_{i \in [k]} d_i^t}=t^2 \left(\frac{k}{\sum_{i \in [k]} d_i^t}\right)^{1/t}=o(1). \label{ineq:falling-factorial-sub-bound-2}
\end{align}

Inserting \eqref{ineq:falling-factorial-sub-bound-1} and \eqref{ineq:falling-factorial-sub-bound-2} in \eqref{ineq:falling-factorial-bound}, we obtain
\begin{align*}
    \frac{\sum_{i \in [k]} (d_i)_{(t)}}{\sum_{i \in [k]} d_i^t} \geq 1-o(1)
\end{align*}
as desired.
\end{proof}

\begin{proof}[Proof of Lemma \ref{lemma: sum-di-ub}]
    This is almost immediate from $\Delta = \max_i d_i$:
    \begin{align*}
        \sum_{i} d_i^p = \sum_{i} d_i^{p-q} d_i^q \leq \Delta^{p-q} \sum_i d_i^q.
    \end{align*}
\end{proof}

\begin{proof}[Proof of Proposition \ref{prop:simple-moments}]
    Recall that the Walsh-Fourier basis $\chi_S$ form an orthonormal basis with respect to $\mathbb{Q}$. Thus,
    \begin{align*}
        \mathbb{E}_{\mathbb{Q} }[\chi_S] &= 0, \quad \text{ if } S \ne \emptyset,\\
        \mathbb{E}_{\mathbb{Q} }[\chi_S^2] &= 1, \\
        \mathbb{E}_{\mathbb{Q} }[\chi_S\chi_S'] &= 0, \quad \text{ if } S \ne S',
    \end{align*}
    and we may compute the first and the second moments of $f_{\mathbf{S} }$ under $\mathbb{Q}$ as:
    \begin{align*}
        \mathbb{E}_{\QQ}[f_{\mathbf{S}} ] &= \sum_{S \subseteq \binom{V}{2}: S \cong \mathbf{S} } \EE_{\QQ}[\chi_S] \\
        &= 0.\\
        \mathbb{E}_{\QQ}[f^2_{\mathbf{S}} ] &= \sum_{S,S' \subseteq \binom{V}{2}: S, S'\cong \mathbf{S} } \EE_{\QQ}[\chi_S\chi_{S'}]\\
        &= \sum_{S \subseteq \binom{V}{2}: S \cong \mathbf{S} }  \EE_{\QQ}[\chi_S^2]\\
        &= \frac{M_{\mathbf{S}}}{|\text{Aut}(\mathbf{S})|}.
    \end{align*}
    Finally, we calculate the first moment of $f_{\mathbf{S} }$ under $\mathbb{P}$. Using the expectation of $\chi_S$ under $\PP$ in \eqref{eq: chi_S_P}, it holds that
     \begin{align*}
        \mathbb{E}_{\PP }[f_{\mathbf{S} } ] &= \sum_{S \subseteq \binom{V}{2}: S \cong \mathbf{S} } \EE_{\PP}[\chi_S]\\
        &= \sum_{S \subseteq \binom{V}{2}: S \cong \mathbf{S} } 
 \frac{M_{S,H}}{M_S}\left( \frac{1-p}{p} \right)^{|S|/2}\\
 &= \frac{M_{\mathbf{S}, H }}{|\text{Aut}(\mathbf{S})|} \left(\frac{1-p}{p}\right)^{|\mathbf{S} |/2}.
    \end{align*}
    
\end{proof}

\section{Deferred Lemmas and Proof of Lemmas in Applications and Counterexamples}

\begin{lemma} \label{lemma: degree-edge-bound}
    Let $k \le n$ and $0 < q \le 1$ be such that $kq = \omega(\log k )$. With probability $1 - o(1)$, for $H \sim G(k,q)$, its maximum degree is $(1 \pm o(1))\cdot (k-1)q $, and its number of edges is $(1 \pm o(1))\cdot \binom{k}{2}q$.
\end{lemma}
\begin{proof}[Proof of Lemma~\ref{lemma: degree-edge-bound}]
    Let $X \sim \Binom(k-1,q)$, where $k,q$ are given as in the statement. We use the multiplicative Chernoff inequality to obtain the tail bound for $X$
    \begin{align*}
        \PP\left[|X - \EE[X]| \geq \delta\EE[X]\right] &= \PP\left[|X - (k-1)q| \geq \delta \cdot (k-1)q  \right]  \leq 2\exp\left(-\frac{\delta^2(k-1)q}{3}\right)
    \end{align*}
    where $0 < \delta < 1$ is a parameter to be determined.
    
    Now let $H \sim G(k,q)$, and denote the degree of a vertex $i \in V(H)$ as $d_i$. Note $d_i \sim \Binom(k-1,q)$ for every $i \in V(H)$. Using union bound, we compute an upper bound on the probability that the maximum degree of $H$ is too large as
    \begin{align}
        \PP\left[\max_{i \in V(H)} d_i \ge (1+\delta) \cdot (k-1)q \right] &\le k \cdot \PP[X \geq (1+\delta) \cdot (k-1)q ]  \\
        & \le 2k \cdot \exp\left(-\frac{ \delta^2(k-1)q}{3}\right). \label{ineq:max-degree-bound} 
    \end{align}
    In the same way, we can also bound the probability that the minimum degree is too small:
    \begin{align}
        \PP\left[\min_{i \in V(H)} d_i \le (1-\delta) \cdot (k-1)q \right] 
        & \le 2k \cdot \exp\left(-\frac{ \delta^2(k-1)q}{3}\right). \label{ineq:min-degree-bound}
    \end{align}
    Since $k q \geq \omega(\log k)$, we may set $\delta^2  = \frac{C \log k}{k q}$ for an arbitrarily large constant $C$ while ensuring both bounds \eqref{ineq:max-degree-bound} and \eqref{ineq:min-degree-bound} is $o(1)$. This implies, with high probability, that all degrees $d_i$ of $H\sim G(k,q)$ is concentrated in the range $(1 \pm o(1))\cdot (k-1)q$, and consequently, the maximum degree of $H \sim G(k,q)$ is $(1\pm o(1)) \cdot (k-1)q$ and the number of edges is $(1\pm o(1)) \cdot \binom{k}{2} \cdot q$.
\end{proof}







%

\end{document}